\documentclass[12pt,reqno]{amsart}
\usepackage{amssymb,amscd,url}
\usepackage{color}    
\usepackage[all]{xy}  
\usepackage{bbm}      

\begin{document}

\allowdisplaybreaks


\title[Portrait Moduli Spaces]
      {Moduli Spaces for Dynamical Systems with Portraits}
\date{\today}
\author[J.R. Doyle]{John R. Doyle}
\email{jdoyle@latech.edu}
\address{Mathematics \& Statistics Department, Louisiana Tech
  University, Ruston, LA 71272 USA}
\author[J.H. Silverman]{Joseph H. Silverman}
\email{jhs@math.brown.edu}
\address{Mathematics Department, Box 1917
  Brown University, Providence, RI 02912 USA.
  ORCID: https://orcid.org/0000-0003-3887-3248}

\thanks{Silverman's research supported by Simons Collaboration Grant \#241309}




\hyphenation{ca-non-i-cal semi-abel-ian}


\newtheorem{theorem}{Theorem}[section]
\newtheorem{lemma}[theorem]{Lemma}
\newtheorem{sublemma}[theorem]{Sublemma}
\newtheorem{conjecture}[theorem]{Conjecture}
\newtheorem{proposition}[theorem]{Proposition}
\newtheorem{corollary}[theorem]{Corollary}
\newtheorem*{claim}{Claim}

\theoremstyle{definition}
\newtheorem{definition}[theorem]{Definition}
\newtheorem*{intuition}{Intuition}
\newtheorem{example}[theorem]{Example}
\newtheorem{remark}[theorem]{Remark}
\newtheorem{question}[theorem]{Question}

\theoremstyle{remark}
\newtheorem*{acknowledgement}{Acknowledgements}





\newenvironment{notation}[0]{%
  \begin{list}%
    {}%
    {\setlength{\itemindent}{0pt}
     \setlength{\labelwidth}{4\parindent}
     \setlength{\labelsep}{\parindent}
     \setlength{\leftmargin}{5\parindent}
     \setlength{\itemsep}{0pt}
     }%
   }%
  {\end{list}}

\newenvironment{parts}[0]{%
  \begin{list}{}%
    {\setlength{\itemindent}{0pt}
     \setlength{\labelwidth}{1.5\parindent}
     \setlength{\labelsep}{.5\parindent}
     \setlength{\leftmargin}{2\parindent}
     \setlength{\itemsep}{0pt}
     }%
   }%
  {\end{list}}
\newcommand{\Part}[1]{\item[\upshape#1]}

\def\Case#1#2{%
\paragraph{\textbf{\boldmath Case #1: #2.}}\hfil\break\ignorespaces}

\renewcommand{\a}{\alpha}
\renewcommand{\b}{\beta}
\newcommand{\g}{\gamma}
\renewcommand{\d}{\delta}
\newcommand{\e}{\epsilon}
\newcommand{\f}{\varphi}
\newcommand{\fhat}{\hat\varphi}
\newcommand{\bfphi}{{\boldsymbol{\f}}}
\renewcommand{\l}{\lambda}
\renewcommand{\k}{\kappa}
\newcommand{\lhat}{\hat\lambda}
\newcommand{\m}{\mu}
\newcommand{\bfmu}{{\boldsymbol{\mu}}}
\renewcommand{\o}{\omega}
\newcommand{\bfpi}{{\boldsymbol{\pi}}}
\renewcommand{\r}{\rho}
\newcommand{\bfrho}{{\boldsymbol{\rho}}}
\newcommand{\rbar}{{\bar\rho}}
\newcommand{\s}{\sigma}
\newcommand{\sbar}{{\bar\sigma}}
\renewcommand{\t}{\tau}
\newcommand{\z}{\zeta}

\newcommand{\D}{\Delta}
\newcommand{\G}{\Gamma}
\newcommand{\F}{\Phi}
\renewcommand{\L}{\Lambda}

\newcommand{\ga}{{\mathfrak{a}}}
\newcommand{\gb}{{\mathfrak{b}}}
\newcommand{\gM}{{\mathfrak{M}}}
\newcommand{\gn}{{\mathfrak{n}}}
\newcommand{\gp}{{\mathfrak{p}}}
\newcommand{\gP}{{\mathfrak{P}}}
\newcommand{\gq}{{\mathfrak{q}}}

\newcommand{\Abar}{{\bar A}}
\newcommand{\Ebar}{{\bar E}}
\newcommand{\kbar}{{\bar k}}
\newcommand{\Kbar}{{\bar K}}
\newcommand{\Pbar}{{\bar P}}
\newcommand{\Sbar}{{\bar S}}
\newcommand{\Tbar}{{\bar T}}
\newcommand{\gbar}{{\bar\gamma}}
\newcommand{\lbar}{{\bar\lambda}}
\newcommand{\ybar}{{\bar y}}
\newcommand{\phibar}{{\bar\f}}
\newcommand{\nubar}{{\overline\nu}}

\newcommand{\Acal}{{\mathcal A}}
\newcommand{\Bcal}{{\mathcal B}}
\newcommand{\Ccal}{{\mathcal C}}
\newcommand{\Dcal}{{\mathcal D}}
\newcommand{\Ecal}{{\mathcal E}}
\newcommand{\Fcal}{{\mathcal F}}
\newcommand{\Gcal}{{\mathcal G}}
\newcommand{\Hcal}{{\mathcal H}}
\newcommand{\Ical}{{\mathcal I}}
\newcommand{\Jcal}{{\mathcal J}}
\newcommand{\Kcal}{{\mathcal K}}
\newcommand{\Lcal}{{\mathcal L}}
\newcommand{\Mcal}{{\mathcal M}}
\newcommand{\Ncal}{{\mathcal N}}
\newcommand{\Ocal}{{\mathcal O}}
\newcommand{\Pcal}{{\mathcal P}}
\newcommand{\Qcal}{{\mathcal Q}}
\newcommand{\Rcal}{{\mathcal R}}
\newcommand{\Scal}{{\mathcal S}}
\newcommand{\Tcal}{{\mathcal T}}
\newcommand{\Ucal}{{\mathcal U}}
\newcommand{\Vcal}{{\mathcal V}}
\newcommand{\Wcal}{{\mathcal W}}
\newcommand{\Xcal}{{\mathcal X}}
\newcommand{\Ycal}{{\mathcal Y}}
\newcommand{\Zcal}{{\mathcal Z}}

\renewcommand{\AA}{\mathbb{A}}
\newcommand{\BB}{\mathbb{B}}
\newcommand{\CC}{\mathbb{C}}
\newcommand{\FF}{\mathbb{F}}
\newcommand{\GG}{\mathbb{G}}
\newcommand{\NN}{\mathbb{N}}
\newcommand{\PP}{\mathbb{P}}
\newcommand{\QQ}{\mathbb{Q}}
\newcommand{\RR}{\mathbb{R}}
\newcommand{\ZZ}{\mathbb{Z}}

\newcommand{\bfa}{{\boldsymbol a}}
\newcommand{\bfb}{{\boldsymbol b}}
\newcommand{\bfc}{{\boldsymbol c}}
\newcommand{\bfd}{{\boldsymbol d}}
\newcommand{\bfe}{{\boldsymbol e}}
\newcommand{\bff}{{\boldsymbol f}}
\newcommand{\bfg}{{\boldsymbol g}}
\newcommand{\bfi}{{\boldsymbol i}}
\newcommand{\bfj}{{\boldsymbol j}}
\newcommand{\bfk}{{\boldsymbol k}}
\newcommand{\bfm}{{\boldsymbol m}}
\newcommand{\bfn}{{\boldsymbol n}}
\newcommand{\bfp}{{\boldsymbol p}}
\newcommand{\bfr}{{\boldsymbol r}}
\newcommand{\bfs}{{\boldsymbol s}}
\newcommand{\bft}{{\boldsymbol t}}
\newcommand{\bfu}{{\boldsymbol u}}
\newcommand{\bfv}{{\boldsymbol v}}
\newcommand{\bfw}{{\boldsymbol w}}
\newcommand{\bfx}{{\boldsymbol x}}
\newcommand{\bfy}{{\boldsymbol y}}
\newcommand{\bfz}{{\boldsymbol z}}
\newcommand{\bfA}{{\boldsymbol A}}
\newcommand{\bfF}{{\boldsymbol F}}
\newcommand{\bfB}{{\boldsymbol B}}
\newcommand{\bfD}{{\boldsymbol D}}
\newcommand{\bfG}{{\boldsymbol G}}
\newcommand{\bfI}{{\boldsymbol I}}
\newcommand{\bfM}{{\boldsymbol M}}
\newcommand{\bfP}{{\boldsymbol P}}
\newcommand{\bfQ}{{\boldsymbol Q}}
\newcommand{\bfT}{{\boldsymbol T}}
\newcommand{\bfU}{{\boldsymbol U}}
\newcommand{\bfX}{{\boldsymbol X}}
\newcommand{\bfY}{{\boldsymbol Y}}
\newcommand{\bfzero}{{\boldsymbol{0}}}
\newcommand{\bfone}{{\boldsymbol{1}}}

\newcommand{\Aut}{\operatorname{Aut}}
\newcommand{\Berk}{{\textup{Berk}}}
\newcommand{\Birat}{\operatorname{Birat}}
\newcommand{\characteristic}{\operatorname{char}}
\newcommand{\codim}{\operatorname{codim}}
\newcommand{\Crit}{\operatorname{Crit}}
\newcommand{\crit}{{\textup{crit}}}
\newcommand{\critwt}{\operatorname{critwt}} 
\newcommand{\Cycle}{\operatorname{Cycles}}
\newcommand{\diag}{\operatorname{diag}}
\newcommand{\dimEnd}{{M}}  
\newcommand{\Disc}{\operatorname{Disc}}
\newcommand{\Div}{\operatorname{Div}}
\newcommand{\Df}{{Df}}  
\newcommand{\Dom}{\operatorname{Dom}}
\newcommand{\dyn}{{\textup{dyn}}}
\newcommand{\End}{\operatorname{End}}
\newcommand{\PortEndPt}{{\textup{endpt}}} 
\newcommand{\END}{\smash[t]{\overline{\operatorname{End}}}\vphantom{E}}
\newcommand{\EndPoint}{E}  
\newcommand{\ExtOrbit}{\mathcal{EO}} 
\newcommand{\Fbar}{{\bar{F}}}
\newcommand{\Fix}{\operatorname{Fix}}
\newcommand{\Fiber}{\operatorname{Fiber}}
\newcommand{\FOD}{\operatorname{FOD}}
\newcommand{\FOM}{\operatorname{FOM}}
\newcommand{\Frame}{\operatorname{Fr}}
\newcommand{\Gal}{\operatorname{Gal}}
\newcommand{\GITQuot}{/\!/}
\newcommand{\GL}{\operatorname{GL}}
\newcommand{\GR}{\operatorname{\mathcal{G\!R}}}
\newcommand{\Hom}{\operatorname{Hom}}
\newcommand{\Index}{\operatorname{Index}}
\newcommand{\Image}{\operatorname{Image}}
\newcommand{\Isom}{\operatorname{Isom}}
\newcommand{\hhat}{{\hat h}}
\newcommand{\Ker}{{\operatorname{ker}}}
\newcommand{\Ksep}{K^{\text{sep}}}  
\newcommand{\Length}{\operatorname{Length}}
\newcommand{\Lift}{\operatorname{Lift}}
\newcommand{\limstar}{\lim\nolimits^*}
\newcommand{\limstarn}{\lim_{\hidewidth n\to\infty\hidewidth}{\!}^*{\,}}
\newcommand{\Mat}{\operatorname{Mat}}
\newcommand{\maxplus}{\operatornamewithlimits{\textup{max}^{\scriptscriptstyle+}}}
\newcommand{\MOD}[1]{~(\textup{mod}~#1)}
\newcommand{\Model}{\operatorname{Model}}
\newcommand{\Mor}{\operatorname{Mor}}
\newcommand{\Moduli}{\mathcal{M}}
\newcommand{\MODULI}{\overline{\mathcal{M}}}
\newcommand{\Norm}{{\operatorname{\mathsf{N}}}}
\newcommand{\notdivide}{\nmid}
\newcommand{\normalsubgroup}{\triangleleft}
\newcommand{\NS}{\operatorname{NS}}
\newcommand{\onto}{\twoheadrightarrow}
\newcommand{\ord}{\operatorname{ord}}
\newcommand{\Orbit}{\mathcal{O}}
\newcommand{\Pcase}[3]{\par\noindent\framebox{$\boldsymbol{\Pcal_{#1,#2}}$}\enspace\ignorespaces}
\newcommand{\pd}{p}       
\newcommand{\bfpd}{\bfp}  
\newcommand{\Per}{\operatorname{Per}}
\newcommand{\Perp}{\operatorname{Perp}}
\newcommand{\PrePer}{\operatorname{PrePer}}
\newcommand{\PGL}{\operatorname{PGL}}
\newcommand{\Pic}{\operatorname{Pic}}
\newcommand{\Portrait}{\mathfrak{Port}}  
\newcommand{\prim}{\textup{prim}}
\newcommand{\Prob}{\operatorname{Prob}}
\newcommand{\Proj}{\operatorname{Proj}}
\newcommand{\Qbar}{{\bar{\QQ}}}
\newcommand{\rank}{\operatorname{rank}}
\newcommand{\Rat}{\operatorname{Rat}}
\newcommand{\Resultant}{\operatorname{Res}}
\newcommand{\Residue}{\operatorname{Residue}} 
\renewcommand{\setminus}{\smallsetminus}
\newcommand{\sgn}{\operatorname{sgn}}
\newcommand{\shafdim}{\operatorname{ShafDim}}
\newcommand{\SL}{\operatorname{SL}}
\newcommand{\Span}{\operatorname{Span}}
\newcommand{\Spec}{\operatorname{Spec}}
\renewcommand{\ss}{{\textup{ss}}}
\newcommand{\stab}{{\textup{stab}}}
\newcommand{\Stab}{\operatorname{Stab}}
\newcommand{\SemiStable}[1]{\textup{(SS$_{#1}$)}}  
\newcommand{\Stable}[1]{\textup{(St$_{#1}$)}}      
\newcommand{\Support}{\operatorname{Supp}}
\newcommand{\Sym}{\operatorname{Sym}}  
\newcommand{\TableLoopSpacing}{{\vrule height 15pt depth 10pt width 0pt}} 
\newcommand{\tors}{{\textup{tors}}}
\newcommand{\Trace}{\operatorname{Trace}}
\newcommand{\trianglebin}{\mathbin{\triangle}} 
\newcommand{\tr}{{\textup{tr}}} 
\newcommand{\UHP}{{\mathfrak{h}}}    
\newcommand{\val}{\operatorname{val}} 
\newcommand{\wt}{\operatorname{wt}} 
\newcommand{\<}{\langle}
\renewcommand{\>}{\rangle}

\newcommand{\pmodintext}[1]{~\textup{(mod}~#1\textup{)}}
\newcommand{\ds}{\displaystyle}
\newcommand{\longhookrightarrow}{\lhook\joinrel\longrightarrow}
\newcommand{\longonto}{\relbar\joinrel\twoheadrightarrow}
\newcommand{\SmallMatrix}[1]{%
  \left(\begin{smallmatrix} #1 \end{smallmatrix}\right)}


\begin{abstract}
A $\textit{portrait}$ $\mathcal{P}$ on $\mathbb{P}^N$ is a pair of finite point sets $Y\subseteq{X}\subset\mathbb{P}^N$, a map $Y\to X$, and an assignment of weights to the points in $Y$. We construct a parameter space $\operatorname{End}_d^N[\mathcal{P}]$ whose points correspond to degree $d$ endomorphisms $f:\mathbb{P}^N\to\mathbb{P}^N$ such that $f:Y\to{X}$ is as specified by a portrait $\mathcal{P}$, and prove the existence of the GIT quotient moduli space $\mathcal{M}_d^N[\mathcal{P}]:=\operatorname{End}_d^N//\operatorname{SL}_{N+1}$ under the $\operatorname{SL}_{N+1}$-action $(f,Y,X)^\phi=\bigl(\phi^{-1}\circ{f}\circ\phi,\phi^{-1}(Y),\phi^{-1}(X)\bigr)$ relative to an appropriately chosen line bundle.  We also investigate the geometry of $\mathcal{M}_d^N[\mathcal{P}]$ and give two arithmetic applications.
\end{abstract}

\maketitle

\setcounter{tocdepth}{1}  

\tableofcontents

\section{Introduction}
\label{section:introduction}

The primary goal of this paper is the construction of moduli spaces
for dynamical systems that come equipped with a collection of points
having specified (partial) orbits and weights. We start with the
space $\End_d^N$ of degree~$d$ endomorphisms $f:\PP^N\to\PP^N$ and the
natural conjugation action of $\f\in\Aut(\PP^N)\cong\PGL_{N+1}$
on~$\End_d^N$ given by $f^\f:=\f^{-1}\circ f\circ\f$. The maps~$f$
and~$f^\f$ have the same dynamics, so the quotient space
$\Moduli_d^N:=\End_d^N\GITQuot\PGL_{N+1}$ is the moduli space of
degree~$d$ dynamical systems on~$\PP^N$. It is well-known
that~$\Moduli_d^N$ exists as a GIT quotient space;
see~\cite{MR2741188,milnor:quadraticmaps,MR2567424,silverman:modulirationalmaps}.

We now add a \emph{portrait structure}, which is a natural dynamical notion of level
structure, to a rational map~$f$. Roughly speaking, we
aim to classify triples~$(f,Y,X)$, where $Y\subset X\subset\PP^N$ are
finite sets of points such that $f(Y)\subset X$ and such that the map
$f:Y\to{X}$ has a specified dynamical structure.  Our goal is to construct a
moduli space that classifies the set of such triples, modulo the
action of~$\f\in\PGL_{N+1}$ via
$(f,Y,X)^\f:=\bigl(f^\f,\f^{-1}(X),\f^{-1}(Y)\bigr)$.

For example, we might take $Y=X$ and require $f:Y=X\to X$ to be the
identity map, in which case we are classifying maps together with a
collection of fixed points. As a second example with $Y=X$, we might
require that $f:Y=X\to{X}$ be a cyclic permutation, so we are
classifying maps together with a cycle of a given period. More
generally, we may further assign weights to the points in~$Y$ and
require that the multiplicity of~$f$ exceed the specified
multiplicity. In this way, for example, we obtain a moduli space for
maps on~$\PP^1$ whose critical points have a specified orbit
structure, a topic that has been much studied due to its importance in
complex dynamics.  Leaving formal definitions for later, we briefly
state some of our main results.

\begin{theorem}
\label{theorem:summaryofmainresults}
Let $N\ge1$ and $d\ge2$, and let~$\Pcal$ be a portrait. Then with
appropriate choices of structure sheaves, the following are
true\textup:
\begin{parts}
  \Part{(a)}
  There is a $\ZZ$-scheme $\Moduli_d^N[\Pcal]$ that classifies
  equivalence classes of dynamical systems on~$\PP^N$ with marked
  points that model~$\Pcal$; see
  Theorem~$\ref{theorem:MNdwithportrait}$.
  \Part{(b)}
  If either~$\Pcal$ is unweighted or $N=1$, and if the variety
  $\Moduli_d^N[\Pcal]\otimes_\ZZ\QQ$ is not empty, then it has the
  expected dimension; see Theorems~$\ref{theorem:UnweightedNonempty}$
  and~$\ref{thm:dimMd1}$.
\end{parts}
\end{theorem}

We refer the reader to Section~\ref{section:abstractportraits} for the
formal definition of abstract portraits and to
Section~\ref{section:ModelsOfPortraits} for the description of the
moduli problem for an endomorphism of~$\PP^N$ to model an weighted
portrait.

It often happens that a portrait has non-trivial automorphisms. For
example, a collection of fixed points may be arbitrarily permuted,
while an $n$-cycle may be cyclically permuted.  In the setting of
Theorem~\ref{theorem:summaryofmainresults}, for
$\Acal\subseteq\Aut(\Pcal)$, we construct more general moduli
spaces~$\Moduli_d^N[\Pcal|\Acal]$ that classify dynamical systems with
marked points that model~$\Pcal$, modulo the action of~$\Acal$. For
example, if~$\Pcal$ is a single $n$-cycle and~$\Acal\cong\ZZ/n\ZZ$ is
the group of cyclic permutations, then~$\Moduli_d^N[\Pcal]$ classifies
maps~$f$ with~$n$ marked points forming an~$n$-cycle,
while~$\Moduli_d^N[\Pcal|\Acal]$ classifies maps~$f$ together with an
$n$-cycle of points, where the points in the $n$-cycle are
indistinguishable.\footnote{For those who are familiar with elliptic
  modular curves, we can make an analogy with the curves~$X_1(n)$
  and~$X_0(n)$. The former classifies pairs~$(E,P)$ consisting of an
  elliptic curve and a point of exact order~$n$, while the latter
  classifies pairs~$(E,C)$ consisting of an elliptic curve and a
  cyclic subgroup of order~$n$.  Then~$X_0(n)$ is the quotient
  of~$X_1(n)$ via the natural action of~$(\ZZ/n\ZZ)^\times$ on~$X_1(n)$.}

It is a truth universally acknowledged that algebro-geometric objects
in possession of an underlying structure must be in want of a moduli
space.\footnote{\dots with apologies to J.A.} This provides our
primary justification for constructing and
studying~$\Moduli_d^N[\Pcal]$, but we also briefly describe two
arithmetic applications. The first concerns the dynamical uniform
boundedness conjecture, whose exact statement is given in
Section~\ref{section:uniformbdedness}.  We prove that this conjecture
is equivalent to a statement about algebraic points
on~$\Moduli_d^N[\Pcal]$; see Theorem~\ref{theorem:UBCifffMBC}.  To do
this, we must first extend a field-of-moduli versus
field-of-definition result from an earlier paper~\cite{fomfodPN2018};
see Section~\ref{section:fomfoddynportdegbd}.  The second application
concerns good reduction of dynamical systems with portrait structure.
In Section~\ref{section:goodreduction} we briefly discuss dynamical
good reduction and its relationship to integral points on the moduli
space~$\Moduli_d^N[\Pcal]$; we refer the reader
to~\cite{arxiv1703.00823} for a more extensive discussion of this last
topic.

We close this introduction with a summary of the contents of our
paper.  Section~\ref{section:earlier} surveys the literature on
dynamical moduli spaces. Section~\ref{section:abstractportraits}
describes the abstract theory of portraits, which are ``pictures''
that are used to classify ``maps with orbits.''
Sections~\ref{section:GITsummary} and~\ref{section:notationEnddN}
briefly recall notation and properties, the former from geometric
invariant theory, the latter for the space~$\End_d^N$.  In
Section~\ref{section:ModelsOfPortraits} we define models of portraits
and the associated parameter and moduli problems.
Section~\ref{section:EnddNP} is devoted to constructing parameter
spaces~$\End_d^N[\Pcal]$ for maps with portraits, first for unweighted
portraits, and then for weighted portraits. We next turn in
Section~\ref{section:SLN1stableEndNPNnz} to GIT and show
that~$\End_d^N\times(\PP^N)^n$ is $\SL_{N+1}$-stable relative to a
suitably chosen line bundle. We use this result and a covering lemma
in Section~\ref{section:stabilityend} to construct the moduli space
space~$\Moduli_d^N[\Pcal]$ classifying $\SL_{N+1}$-isomorphism classes
of maps with portraits, and more generally we construct moduli
spaces~$\Moduli_d^N[\Pcal|\Acal]$ obtained by taking the quotient by a
group~$\Acal$ of automorphisms of~$\Pcal$.
Section~\ref{section:P1critmarkedc} contains a more precise
description of~$\Moduli_d^1[\Pcal]$ in the case that~$\Pcal$ is a
complete preperiodic marking of the critical points.  We then use
Section~\ref{section:Md1Pexamples} to describe a number of explicit
examples of~$\Moduli_d^1[\Pcal]$, and more
generally~$\Moduli_d^1[\Pcal|\Acal]$, in both the unweighted and
weighted cases.  Returning to the general theory, we note that there
may be geometric or topological reasons that
force~$\Moduli_d^N[\Pcal]=\emptyset$. Assuming
that~$\Moduli_d^N[\Pcal]$ is non-empty, we compute its dimension in
Section~\ref{section:dimMdNP} for all~$N$ and unweighted portraits,
and in Section~\ref{section:dimMd1P} for~$N=1$ and weighted portraits.
Sections~\ref{section:fomfoddynportdegbd},~\ref{section:uniformbdedness},
and~\ref{section:goodreduction} give arithmetic applications, as
discussed earlier.  We conclude in Section~\ref{section:dynportraits}
with a discussion of multiplier systems. These are used to construct
regular maps from~$\Moduli_d^N[\Pcal]$ to affine space whose images
are constrained by Ueda's fixed point multiplier relations.

\section{Earlier Results}
\label{section:earlier}
In this section we briefly summarize some of the work that has been
done on moduli spaces for dynamical systems and relate it to the
present work. Our discussion provides a path into the existing
literature, but is far from exhaustive.

Various general constructions of parameter and moduli spaces for
dynamical systems on projective space have been given by a variety of
people, including:
\begin{itemize}
  \setlength{\itemsep}{0pt}
\item
  Degree~$2$ maps on~$\PP^1$:
  DeMarco~\cite{arxiv0412438},
  Milnor~\cite{milnor:quadraticmaps},
  Rees~\cite{MR1149864,MR1317518}.
\item
  Degree~$3$ maps on~$\PP^1$:
    Hutz--Tepper~\cite{MR3136590},
    West~\cite{arxiv1408.3247,MR3427313}.
\item
  Degree~$d$ maps on~$\PP^1$:
  Silverman~\cite{silverman:modulirationalmaps}.
\item
  Degree~$d$ maps on~$\PP^N$:
  Levy~\cite{MR2741188},
  Petsche--Szpiro--Tepper~\cite{MR2567424}.
\end{itemize}

Turning now to dynamical systems with marked points, there is a large
literature studying so-called \emph{dynatomic curves}, which are
1-dimensional moduli spaces whose points classify pairs consisting of
a map in a 1-parameter family and a point of given period or
preperiod. The classical example is the dynatomic
curve~$X_1^{\textup{dyn}}(n)$ whose points classify pairs $(c,\a)$
such that~$\a$ is a point of exact (or formal) period~$n$ for the
map~$x^2+c$, or more generally~$x^d+c$.  (See
Example~\ref{example:dynatomiccurve} for further information about
dynatomic curves and their description as portrait moduli spaces.)
The following papers are among those that investigate~$X_1^{\textup{dyn}}(n)$:
  Bousch~\cite{bousch:thesis},
  Buff--Epstein--Koch~\cite{buffepsteinkoch2018},
  Buff--Lei~\cite{MR3289906},
  Douady--Hubbard~\cite{MR762431,MR812271},
  Doyle~\cite{doyle17dynmod},
  Doyle et al.\ \cite{arxiv1703.04172},
  Doyle--Poonen~\cite{arxiv1711.04233},
  Gao~\cite{MR3460633},
  Gao--Ou~\cite{MR3202010},
  Krumm~\cite {arxiv1805.11152,arxiv1707.02501},
  Lau--Schleicher~\cite{lauschleicher:irreducibility},
  Morton~\cite {morton:dynamicalmoduli}.
These papers study topics such as smoothness, irreducibility,
genus, and gonality of~$X_1^{\textup{dyn}}(n)$, as well as 
reduction mod~$p$ and specialization properties.

There are also many papers that study spaces of maps on~$\AA^1$
or~$\PP^1$ with the property that some or all of their critical points
match a specified preperiodic portrait. A fundamental result in this
area is Thurston's rigidity theorem,
see~\cite{MR1251582,arxiv1702.02582}, which we will make crucial use
of in our investigations. Among other research in this area, we
mention work of Bielefeld--Fisher--Hubbard~\cite{MR1149891} and
Poirier~\cite {MR2496235} that describes all critical point portraits
that can arise from polynomial maps on~$\AA^1$, a paper of Cordwell et
al.\ \cite{MR3323420} classifying rational functions on~$\PP^1$ whose
critical points are all fixed points, and papers of
Faber--Granville~\cite{MR2863906},
Ghioca--Nguyen--Tucker~\cite{MR3349337}, and
Ingram--Silverman~\cite{MR2475968} that describe preperiodic portraits
that occur under reduction modulo primes.

The literature on higher dimensional moduli spaces classifying
dynamical systems with marked points is smaller.  Probably closest in
spirit to the present paper is recent work of Arfeux~\cite{MR3645509},
in which he constructs a moduli space for triples $(f,Y,Z)$, where
$f:\PP^1\to\PP^1$ is a rational map of degree~$d$, and where
$Y,Z\subset\PP^1$ specify marked sets of points such that
$Y=f^{-1}(Z)$ and such that~$Z$ contains the critical values of~$f$.
Arfeux defines an abstract portrait to be the 4-tuple
\text{$(Y,Z,Y\to{Z},\deg:Y\to\NN)$}, so more-or-less the same as our
portrait data $(\Vcal^\circ,\Vcal,\F,\e)$. Our construction, in
addition to being somewhat more general for~$\PP^1$, also deals with
portrait moduli spaces for dynamical systems on~$\PP^N$. We note that
our work is primarily algebraic in nature, while Arfeux's work has a
more geometric flavor.  The two papers thus provide complementary
approaches, both of which are useful.

\section{The Category of Abstract Portraits}
\label{section:abstractportraits}

In this section we describe an abstract theory of portraits, which are
the ``pictures'' that we use to classify ``maps with orbits.''

\begin{definition}
A \emph{portrait} is a $4$-tuple
\[
  \Pcal = (\Vcal^o,\Vcal,\Phi,\e)
\]
such that:
\begin{parts}
  \Part{\textbullet}
  $\Vcal$ is a finite set.
  \Part{\textbullet}
  $\Vcal^\circ\subseteq\Vcal$ is a subset.
  \Part{\textbullet}
  $\Phi:\Vcal^\circ\to\Vcal$ is a function.
  \Part{\textbullet}
  $\e:\Vcal^\circ\to\ZZ_{\ge1}$ is a (multiplicity or weight)
  function.
\end{parts}
A portrait is \emph{unweighted} if $\e(w)=1$ for every
$w\in\Vcal^\circ$, in which case we
write~$\Pcal=(\Vcal^\circ,\Vcal,\Phi)$.  We say that~$\Pcal$ is
\emph{preperiodic} if $\Vcal=\Vcal^\circ$, or equivalently,
since~$\Vcal$ is finite, if every element of~$\Vcal^\circ$ has a
well-defined closed forward orbit under the map~$\Phi$. In this case
we may omit~$\Vcal^\circ$ from the notation and write
simply~$(\Vcal,\Phi,\e)$, or~$(\Vcal,\Phi)$.
\end{definition}

\begin{definition}
\label{definition:critorbitportrait}
Let $\Pcal = (\Vcal^o,\Vcal,\Phi,\e)$ be a portrait. The \emph{set of
critical points of~$\Pcal$} is
\[
  \Crit(\Pcal):=\bigl\{v\in\Vcal^\circ:\e(v)\ge2\bigr\}.
\]
The (\emph{forward}) \emph{orbit} of a point~$v\in\Vcal^\circ$ is the set
\[
  \Orbit_\F(v) := \bigl\{\F^n(v) : \text{$n \ge 0$ and $\F^{n-1}(v)\in\Vcal^\circ$} \bigr\},
\]
where we abuse notation and let $\Phi^0(v) = v$ for any $v \in \Vcal$.
In other words, $\Orbit_\F(v)$ is the set of points obtained by 
starting at~$v$ and repeatedly applying~$\F$ until reaching a point where~$\F$
is not defined.
\end{definition}

\begin{definition}
Let $\Pcal=(\Vcal^\circ,\Vcal,\Phi_\Vcal,\e_\Vcal)$. A
\emph{subportrait} of~$\Pcal$ is a portrait
$(\Wcal^\circ,\Wcal,\Phi_\Wcal,\e_\Wcal)$ satisfying
\[
  \Wcal\subseteq\Vcal,\quad
  \Wcal^\circ\subseteq\Vcal^\circ,\quad
  \Phi_\Wcal = \Phi_\Vcal|_{\Wcal^\circ},\quad
  \e_\Wcal \le \e_\Vcal|_{\Wcal^\circ}.
\]
\end{definition}
  
\begin{definition}
\label{definition:portraitmorphism}
Let~$\Pcal_1=(\Vcal_1^\circ,\Vcal_1,\Phi_1,\e_1)$
and~$\Pcal_2=(\Vcal_2^\circ,\Vcal_2,\Phi_2,\e_2)$ be portraits.  A
\emph{portrait morphism} from~$\Pcal_1$ to~$\Pcal_2$ is an injective
set map
\[
  \a : \Vcal_1 \longhookrightarrow \Vcal_2
\]
satisfying:
\begin{parts}
  \Part{\textbullet}
  $\a(\Vcal_1^\circ) \subseteq \Vcal_2^\circ$.
  \Part{\textbullet}
  $\a\circ\Phi_1(v) = \Phi_2\circ\a(v)$ for all $v\in\Vcal_1^\circ$.
  \Part{\textbullet}
  $\e_2\circ\a(v) \ge \e_1(v)$ for all $v\in\Vcal_1^\circ$.
\end{parts}
We say $\a$ is a \emph{portrait isomorphism} if it has a two-sided inverse, or equivalently,
if it satisfies
\[
  \a(\Vcal_1) = \Vcal_2,\quad
  \a(\Vcal_1^\circ) = \Vcal_2^\circ,\quad
  \e_2\circ\a(v) = \e_1(v)\quad\text{for all $v\in\Vcal_1^\circ$.}
\]
We write~$\Hom(\Pcal_1,\Pcal_2)$ for the set of portrait morphisms
from~$\Pcal_1$ to~$\Pcal_2$, and~$\Aut(\Pcal)$ for the group of portrait
isomorphisms from a portrait~$\Pcal$ to itself. We note that~$\Aut(\Pcal)$ is
a subgroup of the symmetric group on the set~$\Vcal$ of vertices of~$\Pcal$. 
\par
Note that we have defined a category $\Portrait$ whose objects are
portraits and whose morphisms are portrait morphisms.  Subcategories
include the category of unweighted portraits and the category of
preperiodic portraits.
\end{definition}

Table~\ref{table:someautgps} lists the automorphism groups of some
unweighted preperiodic portraits whose vertex set contains~$4$ points.

\begin{table}[ht]
\[
\begin{array}{|c|c|c|c|c|c|} \hline
  \Pcal & \Aut(\Pcal) \\ \hline\hline
  \xymatrix{ {\bullet}  \ar@(dr,ur)[]_{}} \quad
  \xymatrix{ {\bullet}  \ar@(dr,ur)[]_{}} \quad
  \xymatrix{
     {\bullet} \ar@(dl,dr)[r]_{}   & {\bullet}   \ar@(ur,ul)[l]_{} \\ } & \ZZ/2\ZZ\times\ZZ/2\ZZ \\ \hline
  \xymatrix{ {\bullet} \ar[r] & {\bullet}  \ar@(dr,ur)[]_{} } \quad
  \xymatrix{
     {\bullet} \ar@(dl,dr)[r]_{}   & {\bullet}   \ar@(ur,ul)[l]_{} \\
  }  & \ZZ/2\ZZ \\ \hline
  \xymatrix{ {\bullet}  \ar@(dr,ur)[]_{} } \quad
  \xymatrix{
     {\bullet} \ar[r] &  {\bullet} \ar@(dl,dr)[r]_{}   & {\bullet}   \ar@(ur,ul)[l]_{} \\
  }  & 1 \\ \hline
  \xymatrix{
    {\bullet} \ar[r]_{} & {\bullet} \ar[r]_{} &
     {\bullet} \ar@(dl,dr)[r]_{}   & {\bullet}   \ar@(ur,ul)[l]_{} \\
  }  & 1 \\ \hline
  \xymatrix{
    {\bullet} \ar[r]_{} & 
     {\bullet} \ar@(dl,dr)[r]_{}   & {\bullet}   \ar@(ur,ul)[l]_{} & {\bullet} \ar[l]_{} \\
  }  & \ZZ/2\ZZ \\ \hline
    \xymatrix{ {\bullet}  \ar[r]_{} & {\bullet} \ar[d]_{} \\
    {\bullet} \ar[u]_{} & {\bullet} \ar[l]_{} \\
  } &  \raisebox{-15pt}{$\ZZ/4\ZZ$} \\ \hline
\end{array}
\]
\caption{Examples of Portrait Automorphism Groups}
\label{table:someautgps}
\end{table}

\subsection{Portraits, Graphs, and Pseudoforests}
\label{section:portraitisgraph}  
We have called our objects \emph{portraits} because that is what they
are typically called in dynamical systems, especially in the study of
critical point portraits. But our portraits can also be described
using terminology from graph theory. Specifically, a portrait is a
\emph{finite multi-directed pseudoforest}. To unpack this phrase, we
start with a multi-digraph, which is a directed graph~$(V,E)$ that
comes with a function \text{$\mu:E\to\ZZ_{\ge1}$} that assigns a multiplicity
to each edge. We next require that it be a finite graph, which simply
means that the set of vertices~$V$ is finite. Finally, a directed
graph is a pseudoforest if every vertex has at most one out-arrow,
which is equivalent to the condition that every connected component of
the graph is either a directed tree or a directed graph containing at
most one cycle.

With these definitions, a portrait $\Pcal=(\Vcal^\circ,\Vcal,\Phi,\e)$
can be used to create a finite multi-directed pseudoforest~$(V,E,\mu)$
by taking
\[
  V=\Vcal\quad\text{and}\quad E = \bigl\{ \bigl(v,\Phi(v)\bigr) : v\in\Vcal^\circ \bigr\},
\]
and using the multiplicity function
\[
  \mu\bigl(v,\Phi(v)\bigr) = \e(v)\quad\text{for $v\in\Vcal^\circ$.}
\]
Ignoring the multiplicities, the directed graph is
essentially the function
graph of~$\Phi$ in~$\Vcal^\circ\times\Vcal$.
Conversely, a finite multi-directed pseudoforest~$(V,E,\mu)$ can be
used to define a portrait by taking
\begin{align*}
  \Vcal &= V,\quad
  \Vcal^\circ = \{v\in V : \text{$v$ has an out-arrow} \},\\*
  \F(v) &= \text{vertex at other end of $v$'s out-arrow},\\*
  \e(v) &= \mu(\text{arrow originating at $v$}).
\end{align*}
This gives a bijection between the set of portraits and the set of
finite multi-directed pseudoforests, but the portrait characterization
seems more conducive to dynamical discussions.

\section{A Brief Summary of Geometric Invariant Theory}
\label{section:GITsummary}

We set notation and briefly recall some basic facts from geometric
invariant theory.  For details,
see~\cite{MR2004511,mumford:geometricinvarianttheory,MR2884382}.

Let~$X$ be a scheme, let~$G$ be a reductive group that acts on~$X$,
let~$\Lcal$ be an invertible sheaf on~$X$ with a $G$-linearization,
and let $x \in X$ be a geometric point. We consider the following
properties:

There is an integer $n\ge0$ and a $G$-invariant global section
$s\in\G(X,\Lcal^{\otimes n})$ such that:
\begin{parts}
  \Part{(i)}
  $s(x)\ne0$.
  \Part{(ii)}
  $X_{s}:=\{s\ne0\}$ is affine.
  \Part{(iii)}
  The action of $G$ on $X_s$ is closed, i.e., for every $x'\in X_s$, the orbit $Gx'$
  is a closed subset of~$X_s$.
  \Part{(iv)}
  The stabilizer $\Stab_G(x):=\{\s\in G:\s x=x\}$ is $0$-dimensional.
\end{parts}

\begin{definition}
The set of \emph{semi-stable}, respectively \emph{stable}, points
(relative to the sheaf~$\Lcal$) is the set
\begin{align*}
  X^\ss(\Lcal) &:= \bigl\{ x\in X : \text{(i) and (ii) are true} \bigr\} &&(\textbf{semi-stable}) \\*
  X^\stab(\Lcal) &:= \bigl\{ x\in X : \text{(i), (ii), (iii) and (iv) are true} \bigr\} &&(\textbf{stable}) 
\end{align*}
A point that is not semi-stable is called \emph{unstable}. 
\end{definition}

\begin{remark}
There is some variation in the literature regarding the definition of
stability. Thus
Mumford~\cite[Section~1.4]{mumford:geometricinvarianttheory} defines
$x\in X$ to be \emph{stable} if conditions~(i),~(ii), and~(iii) hold,
while~$x$ is \emph{properly stable} if condition~(iv) is also true.
However, it seems to have become more common to include condition~(iv)
in the definition of stability; see for example the standard
reference~\cite[page~115]{MR2004511}. We further note that the
\emph{Hilbert--Mumford numerical criterion} for
stability~\cite[Theorem~2.1]{mumford:geometricinvarianttheory} gives a
criterion for stable points that includes condition~(iv).  This is not an
entirely empty distinction, since as shown
in~\cite{mumford:geometricinvarianttheory}, a geometric quotient
exists even if~(iv) is violated.  For example, consider the diagonal
action of~$\SL_2$ on~$(\PP^1)^2$.  Then~(i)--(iii) hold on
$(\PP^1)^2\setminus\D$, while no points satisfy~(iv).  The geometric quotient
$((\PP^1)^2\setminus\D)\GITQuot\SL_2$ exists and consists of a single
point.
\end{remark}

We now describe the Hilbert--Mumford numerical criterion;
see~\cite[Theorem~2.1]{mumford:geometricinvarianttheory} for details.
Let $\ell:\GG_m\to\SL_{m+1}$ be a $1$-parameter subgroup, and
diagonalize the action
\[
\ell(\a) = \begin{pmatrix}
  \a^{r_0} \\
  & \a^{r_1} \\
  && \ddots \\
  &&& \a^{r_m} \\
\end{pmatrix},\quad
\text{i.e.,}\quad
\ell(\a)\star\hat x_i=\a^{r_i}\hat x_i,
\]
where $(\hat x_0,\ldots,\hat x_m)$ are coordinates on~$\AA^{m+1}$. The
\emph{numerical invariant} of~$\ell$ at~$x$ is
\begin{equation}
  \label{eqn:muxlge0}
  \mu^\Lcal(x,\ell) := \max\{-r_i : \text{$i$ such that $\hat x_i\ne0$} \}.
\end{equation}

\begin{theorem}[Mumford~\cite{mumford:geometricinvarianttheory}]
\label{theorem:numcrit}
Let $G\subset\SL_{m+1}$ act on~$\PP^m$, 
and let $P\in\PP^m$ be a geometric point. Then
\begin{align*}
P \in (\PP^m)^\ss\bigl(\Ocal_{\PP^m}(1)\bigr)
&\quad\Longleftrightarrow\quad
\mu^{\Ocal_{\PP^m}(1)}(P,\ell)\ge0~\text{for all $\ell:\GG_m\hookrightarrow G$.} \\
P \in (\PP^m)^\stab\bigl(\Ocal_{\PP^m}(1)\bigr)
&\quad\Longleftrightarrow\quad
\mu^{\Ocal_{\PP^m}(1)}(P,\ell)>0~\text{for all $\ell:\GG_m\hookrightarrow G$.} 
\end{align*}
\end{theorem}

\section{Notation and Basic Properties of $\End_d^N$}
\label{section:notationEnddN}
In this section we set some standard notation that will be used
throughout this paper, and we recall some of the basic properties of
the space~$\End_d^N$ that parameterizes degree~$d$ endomorphisms
of~$\PP^N$.  For~$N\ge1$ and~$d\ge1$, we recall that there are
$\binom{N+d}{d}$ monomials of degree~$d$ in the
variables~$x_0,\ldots,x_N$. We let
\[
  \dimEnd = \dimEnd(N,d) := (N+1)\binom{N+d}{d}-1,
\]
so~$\dimEnd+1$ is the total number of coefficients that appear
in~$N+1$ generic homogeneous polynomials of degree~$d$.

We let
\[
  \End_d^N := \text{the space of degree $d$ endomorphisms $\PP^N\to\PP^N$}.
\]
We denote elements
\[
  f=[f_0,\ldots,f_N] \in \End_d^N
\]
using homogeneous polynomials $f_i$ of degree~$d$ in the
variables~$x_0,\ldots,x_N$. For each $0\le\rho\le N$, we denote the
coefficients of each~$f_\rho$ by
\begin{equation}
  \label{eqn:frhoxdef}
  f_\rho(\bfx) = \sum_{|\bfe|=d} a_\rho(\bfe)\bfx^\bfe
  := \sum_{\substack{e_0,\ldots,e_N\ge0\\e_0+e_1+\cdots+e_N=d\\}}
  a_\rho(e_0,\ldots,e_n)x_0^{e_0}\cdots x_N^{e_N},
\end{equation}
where the notation is self-explanatory. The coefficients of
the~$f_\rho$ are only defined up to multiplication by a non-zero
scalar, so fixing some ordering for the coefficients, each~$f$
corresponds to a point in~$\PP^\dimEnd$, i.e., there is a natural embedding
\[
  \End_d^N\longhookrightarrow\PP^\dimEnd,
\]
and the~$a_\rho(\bfe)$ are global sections
of~$\Ocal_{\PP^\dimEnd}(1)$. The image of this embedding is the complement of the
vanishing locus of the \emph{Macaulay resultant}~$\Rcal$, which is a homogeneous
polynomial in the $a_\rho(\bfe)$ of degree $(N+1)d^N$.

We define an action of~$\SL_{N+1}$ on~$\End_d^N$ and on~$\PP^N$ as follows:
For  $\f\in\SL_{N+1}$, $f\in\End_d^N$, and~$P\in\PP^N$, we set
\begin{equation}
  \label{eqn:SLactonEndPN}
  \f\star f := \f^{-1}\circ f\circ\f
  \quad\text{and}\quad
  \f\star P := \f^{-1}(P).
\end{equation}
We note that the actions are designed to be compatible with
composition and evaluation,
\[
  \f\star(f^{k}) = (\f\star f)^{k}
  \quad\text{and}\quad
  (\f\star f)(\f\star P) = \f\star\bigl(f(P)\bigr).
\]

For~$f\in\End_d^N$, we write
\[
  \Aut(f) := \{\f\in\SL_{N+1} : \f\star f=f\}
\]
for the stabilizer of~$f$.

The action of~$\SL_{N+1}$ on~$\End_d^N$ extends to an action
on~$\PP^\dimEnd$, and the following classical result says
that~$(\PP^\dimEnd)^\stab\bigl(\Ocal_{\PP^\dimEnd}(1)\bigr)$
contains~$\End_d^N$, i.e., the action of~$\SL_{N+1}$ on~$\End_d^N$ is
stable in the sense of geometric invariant theory.

\begin{theorem}
\label{theorem:EnddNstableAutfinite}
\textup{[Petsche--Szpiro--Tepper, Levy]}
\par\text{Let $N\ge1$ and $d\ge2$.}%
\begin{parts}
\Part{(a)}
For all geometric points $f\in\End_d^N$, the group~$\Aut(f)$ is finite.
\Part{(b)}
The scheme $\End_d^N$ is contained in the $\SL_{N+1}$-stable locus
of~$\PP^\dimEnd$ relative to the invertible
sheaf~$\Ocal_{\PP^\dimEnd}(1)$. The GIT quotient is typically denoted
by $\Moduli_d^N:=\End_d^N\GITQuot\SL_{N+1}$.
\end{parts}
\end{theorem}
\begin{proof}
(a)\enspace Finiteness is proven in
both~\cite{MR2741188} and~\cite{MR2567424}, but the former has an even
stronger result, namely that there is a bound for~$\#\Aut(f)$ that
depends only on~$N$ and~$d$. Levy further proves that for any finite
group~$G\ne1$, the set
\[
  \bigl\{f\in\End_d^N:\text{$\Aut(f)$ has a subgroup that is
    isomorphic to $G$}\bigr\}
\]
is an $\SL_{N+1}$-invariant Zariski closed subscheme of~$\End_d^N$,
hence descends to a Zariski closed subscheme of~$\Moduli_d^N$.
\par\noindent(b)\enspace
Semi-stability of~$\End_d^N$ is immediate from the definition, since
the Macaulay resultant~$\Rcal\in\G\bigl(\PP^\dimEnd,\Ocal_{\PP^\dimEnd}\bigl((N+1)d^N\bigr)\bigr)$
is a global section whose non-vanishing defines~$\End_d^N$, so
$\End_d^N=\{\Rcal\ne0\}$ is an affine variety.
(See~\cite[Theorem~1.8]{MR2884382} for information about the Macaulay resultant.)

Stability of~$\End_d^N$ is more delicate, the first step always being
to prove that~$\Aut(f)$ is finite.  The original proof of stability
for~$\End_d^1$ is in~\cite{silverman:modulirationalmaps}, and there
are two proofs in the literature for~$\End_d^N$.  The first, by
Petsche--Szpiro--Tepper~\cite{MR2567424}, is very direct. They apply
\cite[argument on page~10]{mumford:geometricinvarianttheory} to deduce
that the action of~$\SL_{N+1}$ in~$\End_d^N$ is closed, after which
they cite
\cite[Amplification~1.3 on page~30]{mumford:geometricinvarianttheory}
to deduce that the geometric quotient exists. This last deduction
depends crucially on the fact that $\End_d^N$ is affine, which as noted
earlier follows from the fact that~$\End_d^N$ is the complement of a
hypersurface in $\PP^M$.

The proof by Levy~\cite{MR2741188} uses the numerical criterion
(Theorem~\ref{theorem:numcrit}), so is somewhat more cumbersome, but
it has a number of advantages.  First, as Levy points out, it allows
one to describe more accurately the full semi-stable and stable
loci. Second, as we show in Section~\ref{section:SLN1stableEndNPNnz},
with some additional work Levy's method allows us to deal with
non-affine schemes such as $\End_d^N\times(\PP^N)^n$.
\end{proof}

\section{Models of Portraits and Portrait Moduli Problems}
\label{section:ModelsOfPortraits}
In this section we describe what it means for a map $f\in\End_d^N$ and
a collection of points in~$\PP^N$ to be a model for the
portrait~$\Pcal$. We then define the $\Pcal$-moduli problem to be
that of classifying the set of isomorphism classes of models. The
first order of business is to recall the definition of the
multiplicity of a map~$f$ at a point~$P$, which is required in order
to deal with weighted portraits.  Readers who are interested only in
unweighted portraits may skip this definition.

\begin{definition}
\label{definition:multiplicityPN}
Let $(f,P)\in\End_d^N(k)\times\PP^N(k)$ be a geometric point. The
\emph{multiplicity of~$f$ at~$P$}, denoted~$e_f(P)$, is defined as
follows: Let
\begin{align*}
  \Ocal_P&:=\text{the local ring of~$\PP^N$ at~$P$,} \\
  \gM_P&:=\text{the maximal ideal of~$\Ocal_P$.}
\end{align*}
Write~$f=[f_0,\ldots,f_N]$ using homogeneous polynomials of
degree~$d$, choose an index~$j$ with~$f_j(P)\ne0$,
and let~$I_{f,P}$ be the ideal of~$\Ocal_P$ defined by
\[
  I_{f,P} :=   \left(  \frac{f_0}{f_j}-\frac{f_0}{f_j}(P),\frac{f_1}{f_j}-\frac{f_1}{f_j}(P),\ldots,
  \frac{f_N}{f_j}-\frac{f_N}{f_j}(P)\right) \subseteq \gM_P.
\]
To see that the ideal~$I_{f,P}$ is independent of the choice of~$j$,
we note that if~$k$ is some other index with~$f_k(P)\ne0$, then we
have the identity
\begin{multline*}
  \frac{f_\ell}{f_k}-\frac{f_\ell}{f_k}(P) \\
  = \left[ \left(
    \frac{f_\ell}{f_j}-\frac{f_\ell}{f_j}(P) \right)\cdot
    \frac{f_k}{f_j}(P) - \left( \frac{f_k}{f_j}-\frac{f_k}{f_j}(P)
    \right)\cdot \frac{f_\ell}{f_j}(P) \right] \cdot \frac{f_j}{f_k}
  \cdot \frac{f_j}{f_k}(P).
\end{multline*}
We thus obtain a well-defined quantity
\[
  e_f(P) := \dim_k \Ocal_P/I_{f,P}.
\]
For basic properties of multiplicities, and alternative equivalent
definitions in various settings, see for
example~\cite[page~669]{MR1288523}
or~\cite{arxiv:0801.3643,arxiv:1011.5155}.
\end{definition}
\begin{remark}
In dimension~$1$, i.e., for~$N=1$, the quantity~$e_f(P)$ is also
called the \emph{ramification index of~$f$ at~$P$}.  If~$e_f(P)\ge2$,
then dynamicists say that~$P$ is a \emph{critical point} and
that~$f(P)$ is a \emph{critical value}, while algebraic geometers say
that~$P$ is a \emph{ramification point} and~$f(P)$ is a \emph{branch
  point}.  If $P\ne\infty$ and $f(P)\ne\infty$, then the definition
of~$e_f(P)$ simplifies to
\[
  e_f(P) = \ord_P\bigl( f - f(P) \bigr).
\]
Further, working over an algebraically closed field~$k$ of
characteristic~$0$, we have the standard formulas
\[
  \deg(f) = \sum_{Q\in f^{-1}(P)} e_f(Q)
  \quad\text{and}\quad
  2\deg(f)-2=\sum_{P\in\PP^1(k)} \bigl(e_f(P)-1\bigr).
\]
\end{remark}

\begin{definition}
\label{definition:unwtedmoduliprob}
Let $\Pcal=(\Vcal^\circ,\Vcal,\F,\e)$ be a portrait, and
let~$k$ be an algebraically closed field. A \emph{model for~$\Pcal$
  over~$k$} (\emph{of dimension~$N$ and degree~$d$}) is a
pair~$(f,\gamma)$, where~$f$ and~$\g$ are maps
\begin{equation}
  \label{eqn:modelforP}
  f\in\End_d^N(k)\quad\text{and}\quad
  \g:\Vcal\longhookrightarrow\PP^N(k)
\end{equation}
such that the diagram
\begin{equation}
  \label{eqn:VcircPNk}
  \begin{CD}
    \Vcal^\circ @>\g>> \PP^N(k) \\
    @VV\F V @VV f V \\
    \Vcal @>\g>> \PP^N(k) \\
  \end{CD}
\end{equation}
commutes, and such that the multiplicities of~$f$ satisfy
\begin{equation}
  \label{eqn:efgvgeev}
  e_f\bigl(\g(v)\bigr) \ge \e(v)\quad\text{for all $v\in\Vcal^\circ$.}
\end{equation}
We note that~$\g$ is required to be injective, and also that
if~$\Pcal$ is unweighted, then the multiplicity
condition~\eqref{eqn:efgvgeev} is vacuous. We write
$\Model_d^N[\Pcal]$ for the set of pairs~$(f,\g)$
satisfying~\eqref{eqn:VcircPNk} and~\eqref{eqn:efgvgeev}.

Two models~$(f,\g)$ and~$(f',\g')$ are \emph{equivalent} if there is a
change of variables $\f\in\PGL_{N+1}(k)$ such that
\[
  f' = \f^{-1}\circ f\circ\f
  \quad\text{and}\quad
  \g' = \f^{-1}\circ\g.
\]
The \emph{$(N,d,\Pcal)$-moduli problem} is to classify equivalence
classes of $(N,d)$-models for~$\Pcal$. 
\end{definition}

\begin{lemma}
\label{lemma:portmodmaps}
Let~$\Pcal_1$ and~$\Pcal_2$ be portraits such that
$\Model_d^N(\Pcal_2)\ne\emptyset$. There is a natural injective map
\begin{align*}
  \Hom(\Pcal_1,\Pcal_2) &\longrightarrow
  \Hom\bigl(\Model_d^N(\Pcal_2),\Model_d^N(\Pcal_1)\bigr) \\
  \a &\longmapsto \Bigl( (f,\g) \longmapsto (f,\g\circ\a) \Bigr).
\end{align*}
In other words, if $\a:\Pcal_1\to\Pcal_2$ is a portrait morphism as
described in Definition~$\ref{definition:unwtedmoduliprob}$ and
if~$(f,\g)$ is a model for~$\Pcal_2$, then~$(f,\g\circ\a)$ is a model
for~$\Pcal_1$.
\end{lemma}
\begin{proof}
Let $\a:\Pcal_1\to\Pcal_2$ be a portrait morphism and let~$(f,\g)$
be a model for~$\Pcal_2$. The fact that~$(f,\g\circ\a)$ is a model
for~$\Pcal_1$ is immediate from the following facts, all of which
come directly from the definitions:
\begin{parts}
\Part{(1)}
$\F(\Vcal_1^\circ)\subseteq\Vcal_2^\circ$.
\Part{(2)}
$\a:\Vcal_1\hookrightarrow\Vcal_2$ and $\g:\Vcal\hookrightarrow\PP^N$ are injective.
\Part{(3)}
The diagram
\[
\begin{CD}
  \Vcal_1^\circ @>\a>> \Vcal_2^\circ @>\g>> \PP^N \\
  @VV\F_1 V @VV\F_2 V @VV f V \\
  \Vcal_1 @>\a>> \Vcal_2 @>\g>> \PP^N \\  
\end{CD}
\]
has commutative squares.
\Part{(4)}
The weights behave properly, as we see from the following calculation:
\begin{align*}
  e_f\circ\g&\ge\e_2 &&\text{since $(f,\g)$ is a model for $\Pcal_2$,} \\*
  e_f\circ\g\circ\a&\ge\e_2\circ\a &&\text{composing with $\a$,} \\
  e_f\circ\g\circ\a&\ge\e_1 &&\text{since $\e_2\circ\a\ge\e_1$
    by definition of portrait}\\*[-1\jot]
    &&&\text{morphism (Definition~\ref{definition:portraitmorphism}).}
\end{align*}
\end{parts}
Hence~$(f,\g\circ\a)$ is a model for~$\Pcal_1$.
\par
Finally, let $\a':\Pcal_1\to\Pcal_2$ be another portrait morphism, and
suppose that
$\bigl(f,\gamma\circ\a\bigr)=\bigl(f,\gamma\circ\a'\bigr)$ for every
model~$(f,\gamma)$ of~$\Pcal_2$.
This means that $\g\circ\a=\g\circ\a'$, and then the assumed
injectivity of~$\g$,~$\a$, and~$\a'$ forces~$\a=\a'$. 
\end{proof}

\section{The Parameter Space $\End_d^N[\Pcal]$ for Models of a Portrait}
\label{section:EnddNP}

Our goal in this section is to construct a universal parameter space
for the set of models of~$\Pcal$.  
We start with a bit of useful notation and a definition.

\begin{definition}
For any finite index set~$\Ncal$ and any scheme~$S$, we write
\[
  S^\Ncal := \prod_{n\in\Ncal} S_n\quad\text{with each $S_n\cong S$}
\]
for the product of~$\#\Ncal$ copies of~$S$, indexed by~$\Ncal$.
For~$n\in\Ncal$, we let
\[
\pi_{S,n} : S^\Ncal\longrightarrow S_n\cong S
\]
denote the projection map.  If the scheme~$S$ is clear from context,
we may omit~$S$ from the notation and write~$\pi_n$.
\end{definition}

\begin{definition}
\label{definition:upsNdP}
Let $\Pcal$ be a portrait.  A \emph{universal parameter space for
  $(N,d)$-models of~$\Pcal$} is a scheme~$E$ with the following
property: Let~$T$ be a scheme, let $f:T\to\End_d^N$ be a morphism, and
let $\g:T\to(\PP^N)^\Vcal$ be a morphism such that for every geometric
point~$t\in T(k)$, the pair~$(f_t,\g_t)$ is a model for~$\Pcal$
over~$k$. Then there is a unique morphism $F:T\to{E}$ such that
$F_t=(f_t,\g_t)$ for all geometric points~$t\in T(k)$.
\end{definition}

The construction of the universal parameter space, which we denote
by~$\End_d^N[\Pcal]$, proceeds in two steps. We start with the case
that~$\Pcal$ is an unweighted portrait. We then
construct~$\End_d^N[\Pcal]$ for general~$\Pcal$ by taking a closed
subscheme of the space for the associated unweighted portrait.

\subsection{The Parameter Space $\End_d^N[\Pcal]$ for Unweighted Portraits}
\label{subsection:unwtEnddNP}

In this section we assume that~$\Pcal$ is an unweighted portrait. We
construct a scheme $\End_d^N[\Pcal]$, which we prove is a universal
parameter space for $(N,d)$-models of~$\Pcal$.  We start with
$\End_d^N$, the parameter space of degree~$d$ endomorphisms
$f:\PP^N\to\PP^N$, as described in
Section~\ref{section:notationEnddN}.  In particular,~$\End_d^N$ is
identified with an affine subscheme of~$\PP^\dimEnd$ by listing the
coefficients of the homogeneous polynomials that define an
endomorphism.

We recall that~$(\PP^N)^\Vcal$ is a product of copies of~$\PP^N$
indexed by~$\Vcal$, and that for each~$v\in\Vcal$, we have a
projection map~$\pi_v:(\PP^N)^\Vcal\to\PP^N$.  For each pair of
vertices~$v,v'\in\Vcal$, we similarly have a bi-projection map
\[
  \pi_v\times\pi_{v'} : (\PP^N)^{\Vcal} \longrightarrow \PP^N\times \PP^N.
\]
We let
\[
 \D\subset\PP^N\times \PP^N
\]
denote the diagonal, and then we define the \emph{complement of the
  big diagonal of $(\PP^N)^\Vcal$} to be the scheme
\[
  (\PP^N)^\Vcal_\D := (\PP^N)^\Vcal \setminus \bigcup_{v,v'\in\Vcal,\,v\ne v'} \pi_{v,v'}^{-1}(\D).
\]
For example, if~$k$ is an algebraically closed field, then
$(\PP^N)^\Vcal_\D(k)$ is isomorphic to the usual complement of
the big diagonal of~$(\PP^N(k))^{\#\Vcal}$, i.e.,
$(\PP^N)^\Vcal_\D(k)$ is the set of $\#\Vcal$-tuples of distinct points
in~$\PP^N(k)$.

We observe that~$\End_d^N$ and~$(\PP^N)^\Vcal_\D$ are naturally
quasi-projective schemes over~$\Spec(\ZZ)$, since they are open
subschemes, respectively, of the schemes~$\PP^\dimEnd$
and~$(\PP^N)^{\#\Vcal}$.  We want to define $\End_d^N[\Pcal]$ to be
the subscheme of $\End_d^N\times(\PP^N)^{\Vcal}_\D$ that consists of
maps and points that are ``models of~$\Pcal$.''

We start with a map~$G$ that we define at the level of geometric
points by
\[
  G:\End_d^N\times(\PP^N)^\Vcal \longrightarrow (\PP^N)^\Vcal \times (\PP^N)^\Vcal,
  \quad
  G(f,\bfP) = \bigl(f(\bfP),\bfP\bigr).
\]
More formally, let
\[
  \bff:\End_d^N\times\PP^N\to\PP^N
\]
be the map giving the natural action of~$\End_d^N$ on~$\PP^N$, i.e.,
at the level of geometric points~$\bff$ sends~$(f,P)\to f(P)$. Then
the map~$G$ is the composition
\begin{align*}
  G:\End_d^N\times(\PP^N)^\Vcal
  &\xrightarrow{\hbox to5em{\hfil$\scriptstyle\operatorname{diag}\times\operatorname{diag}$\hfil}}
  (\End_d^N)^\Vcal\times\bigl((\PP^N)^\Vcal\times(\PP^N)^\Vcal\bigr) \\*
  &\xrightarrow{\hbox to5em{\hfil$\scriptstyle\sim$\hfil}}
  (\End_d^N\times\PP^N)^\Vcal \times(\PP^N)^\Vcal\\*
  &\xrightarrow{\hbox to5em{\hfil$\scriptstyle\bff^\Vcal\times\operatorname{id}$\hfil}}
  (\PP^N)^\Vcal\times(\PP^N)^\Vcal.
\end{align*}

Next, for each pair of vertices~$v,v'\in\Vcal$, we define a map
\begin{equation}
  \label{eqn:Psivvprime}
  \Psi_{v,v'} : \End_d^N\times(\PP^N)^\Vcal_\D \to \PP^N\times\PP^N,
  \quad
  \Psi_{v,v'} = (\pi_v\times\pi_{v'})\circ G
\end{equation}
that first applies~$G$ and then projects onto the indicated factors.
Recalling that~$\D$ denotes the diagonal of~$\PP^N\times\PP^N$, we
define~$\End_d^N[\Pcal]$ to be the scheme
\begin{equation}
  \label{eqn:EnddNPcaldef}
  \End_d^N[\Pcal] := \bigcap_{v\in\Vcal^\circ}  \Psi_{v,\Phi(v)}^{-1}(\D).
\end{equation}

To trace out what this means, let
\[
  \bigl(f,(P_w)_{w\in\Vcal}\bigr) \in \End_d^N(k) \times (\PP^N)^\Vcal_\D(k)
\]
be a geometric point, where we interpret the second coordinate as
being given by an injective map~$P:\Vcal\hookrightarrow\PP^N(k)$.
(The injectivity comes from the fact that we are required to be
in~$(\PP^N)^\Vcal_\D(k)$, which is the complement of the big
diagonal.)  Then
\begin{align*}
  \bigl(f,(P_w)_{w\in\Vcal}\bigr) &\in \End_d^N[\Pcal](k) \\
  &\quad\Longleftrightarrow\quad
  \Psi_{v,\Phi(v)}\bigl(f,(P_w)_{w\in\Vcal}\bigr)\in\D
  \quad\text{for all $v\in\Vcal^\circ$.}
\end{align*}
We compute
\begin{align*}
  \Psi_{v,\Phi(v)}\bigl(f,&(P_w)_{w\in\Vcal}\bigr) \\
  &= (\pi_v\times\pi_{\Phi(v)})\circ G \bigl(f,(P_w)_{w\in\Vcal}\bigr)
  &&\text{definition of $\Psi_{v,v'}$,} \\
  &= (\pi_v\times\pi_{\Phi(v)}) \bigl((f(P_w))_{w\in\Vcal},(P_w)_{w\in\Vcal}\bigr)
  &&\text{definition of $G$,} \\
  &= \bigl( f(P_v), P_{\Phi(v)} \bigr)
  &&\text{definition of $\pi_v$.}
\end{align*}
Hence
\begin{equation}
  \label{eqn:EnddNPdesc}
  \bigl(f,(P_w)_{w\in\Vcal}\bigr) \in \End_d^N[\Pcal](k) 
  \;\Longleftrightarrow\;
  f(P_v) = P_{\Phi(v)}
  \;\text{for all $v\in\Vcal^\circ$,}
\end{equation}
which says exactly that~$(f,P)$ is a model for~$\Pcal$.

\begin{proposition}
\label{proposition:EnddNPunweighted}
Let~$\Pcal$ be an unweighted portrait, and let
\[
  \End_d^N[\Pcal] \subset \End_d^N \times (\PP^N)^\Vcal \subset \PP^\dimEnd \times (\PP^N)^\Vcal
\]
be the scheme described by~\eqref{eqn:EnddNPcaldef}.
\begin{parts}
\Part{(a)}
$\End_d^N[\Pcal]$ is a universal parameter space for models of~$\Pcal$.
\vspace{1\jot}
\Part{(b)}
$\End_d^N[\Pcal]$ is an $\SL_{N+1}$-invariant subvariety of $ \PP^\dimEnd \times (\PP^N)^\Vcal$.
\end{parts}
\end{proposition}

\begin{proof}
\par\noindent(a) 
Let~$T$ be a scheme, and let
\[
  f:T\to\End_d^N
  \quad\text{and}\quad
  \g:T\to(\PP^N)^\Vcal
\]
be morphisms such that for every geometric point~$t\in T(k)$, the
pair $\bigl(f(t),\g(t)\bigr))$ is a model for~$\Pcal$ over~$k$. We
need to show that there is a unique morphism $F:T\to\End_d^N[\Pcal]$
such that $F(t)=\bigl(f(t),\g(t)\bigr))$ for all geometric
points~$t\in T(k)$. The given~$f$ and~$\g$ define a morphism
$(f,\g):T\to\End_d^N\times(\PP^N)^\Vcal$, and~\eqref{eqn:EnddNPdesc}
combined with the assumptions on~$f$ and~$\g$ imply that the image of
this morphism lies in the subscheme~$\End_d^N[\Pcal]$.
\par\noindent(b)
We need to verify that~$\End_d^N[\Pcal]$ is an~$\SL_{N+1}$-invariant
subscheme of~$\PP^\dimEnd\times(\PP^N)^\Vcal$, so first we need to describe
the action.  We let~$\f\in\SL_{N+1}$ act on~$f\in\End_d^N$ by
conjugation $f^\f:=\f^{-1}\circ{f}\circ\f$ and we extend this action
to~$\PP^\dimEnd$ as usual. Similarly, we let~$\f\in\SL_{N+1}$ act
on~$(P_v)\in(\PP^N)^\Vcal$ via the diagonal action
$(P_v)^\f:=\bigl(\f^{-1}(P_v)\bigr)$.  We observe
that~$(\PP^N)^\Vcal_\D$ is an~$\SL_{N+1}$-invariant subscheme
of~$(\PP^N)^\Vcal$, since the big diagonal is~$\SL_{N+1}$-invariant.

It suffices to check the~$\SL_{N+1}$-invariance on geometric points,
so we need to show that
\begin{multline*}
  \bigl(f,(P_w)_{w\in\Vcal}\bigr) \in \End_d^N[\Pcal](k) 
  \quad\text{and}\quad \f\in\SL_{N+1}(k)\\
  \quad\Longrightarrow\quad
  \bigl(f^\f,(P_w^\f)_{w\in\Vcal}\bigr) \in \End_d^N[\Pcal](k).
\end{multline*}
Applying the characterization~\eqref{eqn:EnddNPdesc} of geometric
points in~$\End_d^N[\Pcal]$, we need to show that
\[
  f(P_v) = P_{\Phi(v)}
  ~\text{for all $v\in\Vcal^\circ$}
  \quad\Longrightarrow\quad
  f^\f(P^\f_v) = P^\f_{\Phi(v)}
  ~\text{for all $v\in\Vcal^\circ$.}
\]
But the definitions have been set up so that for all $v \in \Vcal^\circ$,
\[
  f^\f(P_v^\f) = (\f^{-1}\circ f\circ \f)\bigl(\f^{-1}(P_v)\bigr)
  = \f^{-1}\bigl(f(P_v)\bigr)
  = \f^{-1}(P_{\Phi(v)})
  = P_{\Phi(v)}^\f.
\]
This completes the verification that~$\End_d^N[\Pcal]$
is~$\SL_{N+1}$-invariant.
\end{proof}

\begin{remark}
\label{remark:EnddNPexplicit}
We describe~$\End_d^N[\Pcal]$ more explicitly.  For convenience, we
let $\Vcal=\{1,2,\ldots,n\}$ with
$\Vcal^\circ=\{1,2,\ldots,m\}$. (Note that if $\Pcal = \emptyset$ is
the empty portrait, then $\End_d^N[\Pcal] = \End_d^N$.)  We start with
the scheme
\[
  (\PP^N)^n := \prod_{i=1}^n \Proj \ZZ[X_0^{(i)},X_1^{(i)},\ldots,X_N^{(i)}].
\]
For each $i\ne j$, the diagonal~$\D_{i,j}$ corresponding to the
condition $P_i=P_j$ is the subscheme described by the equations
\[
  \D_{i,j} := \bigl( X_k^{(j)} X_\ell^{(i)} - X_\ell^{(j)} X_k^{(i)} \bigl)_{0\le k,\ell\le N}.
\]
We let
\[
  (\PP^N)^n_\D = (\PP^N)^n \setminus \bigcup_{1\le i<j\le n} \D_{i,j}
\]
be the complement of the big diagonal.

Similarly, for each $i$ and $j$, the condition $f(P_i)=P_j$ defines a
closed subscheme~$\G_{i,j}$ of $\End_d^N\times(\PP^N)^n_\D$.
Explicitly, if we write $f=[f_0,\ldots,f_N]$, then~$\G_{i,j}$ is the
subscheme described by the equations
\[
\G_{i,j} := \bigl(
  X_k^{(j)}f_\ell(X_0^{(i)},\ldots,X_N^{(i)})
  - X_\ell^{(j)}f_k(X_0^{(i)},\ldots,X_N^{(i)}) \bigr)_{0\le k,\ell\le N}.
\]
Finally, we can describe~$\End_d^N[\Pcal]$ as the intersection
\[
  \End_d^N[\Pcal] := \bigcap_{i=1}^m \G_{i,\Phi(i)}.
\]
\end{remark}

\begin{remark}
\label{remark:whenemptyset}
There are certain natural geometric constraints that~$\Pcal$ must
satisfy in order for~$\End_d^N[\Pcal]$ to be non-empty.  See
Theorem~\ref{theorem:UnweightedNonempty} for unweighted
portraits, and also Proposition~\ref{proposition:whenisMPempty} for
weighted portraits when~$N=1$.
\end{remark}

\begin{remark}
We have constructed $\End_d^N[\Pcal]$ as a subscheme of
$\PP^\dimEnd\times(\PP^N)^\Vcal$. In practice, it is often possible to
specify the portrait~$\Pcal$ using fewer than~$\#\Vcal$ points.  This,
in turn, can be used to give a simpler description
of~$\Moduli_d^N[\Pcal]$, i.e., a description using fewer equations in
a lower-dimensional ambient space. This is the approach taken by the
first author in~\cite{doyle17dynmod}, where dynamical modular curves
are constructed for polynomials~$x^2+c$ with portraits~$\Pcal$
specified by generating sets. However, there are subtleties in this
approach, since disconnectedness of~$\Pcal$ may lead to extraneous
components in the naively constructed moduli space.  We refer the
reader to~\cite[Example~2.16]{doyle17dynmod} for details.
\end{remark}

\subsection{The Parameter Space $\End_d^N[\Pcal]$ for Weighted Portraits}
\label{subsection:wtEnddNP}

In this section we construct a universal parameter
space~$\End_d^N[\Pcal]$, where now~$\Pcal=(\Vcal^\circ,\Vcal,\Phi,\e)$
is allowed to include non-trivial weights.  We write
$\Pcal'=(\Vcal^\circ,\Vcal,\Phi,\e')$ for the associated unweighted
portrait, i.e.,~$\Pcal$ and~$\Pcal'$ have the same vertices and the
same directed edges, but $\e'(v)=1$ for all $v\in\Vcal^\circ$.  In
Section~\ref{subsection:unwtEnddNP} we constructed the
space~$\End_d^N[\Pcal']$.  In this section we show that we can
take~$\End_d^N[\Pcal]$ to be a Zariski closed subscheme
of~$\End_d^N[\Pcal']$.

\begin{theorem}
\label{theorem:EnddNPweighted}
Let~$\Pcal$ be a portrait, let~$\Pcal'$ be the associated unweighted
portrait, and let~$\End_d^N[\Pcal']$ be the scheme described in
Proposition~$\ref{proposition:EnddNPunweighted}$.  There is a Zariski
closed subscheme of~$\End_d^N[\Pcal']$, which we denote
by~$\End_d^N[\Pcal]$, whose geometric points are specified by
\begin{equation}
  \label{eqn:EnddNPefPgeev}
  \End_d^N[\Pcal] := \bigl\{ (f,\bfP) \in \End_d^N[\Pcal'] :
  e_f(P_v) \ge \e(v)~\text{for all $v\in\Vcal^\circ$} \bigr\}.
\end{equation}
The space $\End_d^N[\Pcal]$ has the following properties\textup:
\begin{parts}
\Part{(a)}
$\End_d^N[\Pcal]$ is a universal parameter space for models of~$\Pcal$.
\vspace{1\jot}
\Part{(b)}
$\End_d^N[\Pcal]$ is an $\SL_{N+1}$-invariant subvariety of $\End_d^N[\Pcal']$.
\end{parts}
\end{theorem}

\begin{proposition}
\label{proposition:EnddNP1toP2}
Let~$\Pcal_1$ and~$\Pcal_2$ be weighted portraits. There is a natural map
\[
  \Hom(\Pcal_1,\Pcal_2) \longrightarrow
  \Hom\bigl(\End_d^N[\Pcal_2],\End_d^N[\Pcal_1]\bigr) 
\]
induced by the map on models described in Lemma~$\ref{lemma:portmodmaps}$,
i.e., a portrait morphism $\Pcal_1\to\Pcal_2$ induces a morphism
$\End_d^N[\Pcal_2]\to\End_d^N[\Pcal_1]$ of schemes.
\end{proposition}

We start with the case of a single point.

\begin{proposition}
\label{proposition:efPgeepsilon}
Let $\e\ge0$ be an integer. Then the locus
\begin{equation}
  \label{eqn:fPefPgeepsilon}
  \bigl\{ (f,P) \in \End_d^N \times \PP^N : e_f(P) \ge \e  \bigr\}
\end{equation}
is a Zariski closed subscheme of $\End_d^N\times\PP^N$.  
\end{proposition}

The key to proving Proposition~\ref{proposition:efPgeepsilon} is the
following lemma, which is used to show that the
set~\eqref{eqn:fPefPgeepsilon} is defined by finitely many polynomial
relations.

\begin{lemma}
\label{lemma:IfpcontainsMe}  
With notation as in Definition~\textup{\ref{definition:multiplicityPN}},
we have\textup:
\begin{parts}
\Part{(a)}
$e_f(P)$ is finite.
\Part{(b)}
For all integers $\e\ge0$, we have
\[
  e_f(P) \ge \dim_k\left( \frac{\Ocal_P}{I_{f,P}+\gM_P^\e}\right)
  \ge \min\bigl\{\e,e_f(P)\bigr\}.
\]
\end{parts}
\end{lemma}
\begin{proof}
(a)\enspace Let~$j$ be an index for which $f_j(P) \ne 0$, and let~$V_{f,P,j}$
be the Zariski closure in~$\PP^N$ of the set
\[
  \bigl\{Q\in\PP^N : \text{$f_j(Q)\ne0$ and
    $(f_i/f_j)(Q)=(f_i/f_j)(P)$ for all $0\le i\le N$}\bigr\}.
\]
Then $e_f(P)=\infty$ if and only if~$V_{f,P,j}$ has dimension greater
than~$0$. But if~$\dim V_{f,P,j}\ge1$, then
$V_{f,P,j}\cap\{f_j=0\}\ne\emptyset$, and points in this intersection
are indeterminacy points of~$f$.
\par\noindent(b)\enspace
The first inequality is immediate from the definition of~$e_f(P)$.
To ease notation, let~$M:=\Ocal_P/I_{f,P}$, where we view~$M$ as a
finitely generated~$\Ocal_P$-module.
Let~$r\ge0$ be the largest integer (or $\infty$) such that
\[
  \gM_P^iM\ne\gM_P^{i+1}M\quad\text{for all $0\le i< r$.}
\]
Then we have a descending chain of~$\Ocal_P$-modules,
\begin{equation}
  \label{eqn:OPchain}
  M \supsetneq \gM_P M \supsetneq \gM_P^2 M \supsetneq \gM_P^3 M
  \supsetneq \cdots  \supsetneq \gM_P^r M = \gM_P^{r+1} M,
\end{equation}
and considering the successive quotients, we see that
\[
  e_f(P) := \dim_k M = \sum_{i=0}^\infty \dim_k \left(\frac{\gM_P^i M}{\gM_P^{i+1} M}\right) \ge \sum_{i=0}^{r-1} 1 = r.
\]
This proves that~$r$ is finite, since we know from~(a) that~$e_f(P)$ is finite.

Further, our choice of~$r$ implies that
\text{$\gM_P\cdot\gM_P^rM=\gM_P^rM$}.  Applying Nakayama's lemma to
the finitely generated $\Ocal_P$-module~$\gM_P^rM$, we conclude that
$\gM_P^rM=0$. Hence the chain~\eqref{eqn:OPchain} terminates with~$0$,
and we obtain an equality
\begin{equation}
  \label{eqn:dfPeqsumr1}
  e_f(P) := \dim_k M = \sum_{i=0}^{r-1} \dim_k \left(\frac{\gM_P^i M}{\gM_P^{i+1} M}\right) .
\end{equation}

We use the isomorphisms
\[
  \frac{\Ocal_P}{I_{f,P}+\gM_P^\e}
  \cong \frac{\Ocal_P/I_{f,P}}{(I_{f,P}+\gM_P^\e)/I_{f,P}}
  \cong \frac{\Ocal_P/I_{f,P}}{\gM_P^\e/(I_{f,P}\cap\gM_P^\e)}
  \cong \frac{M}{\gM_P^\e M}
\]
and the chain~\eqref{eqn:OPchain} to compute
\begin{equation}
  \label{eqn:dimkOPIfPMPe}
  \dim_k\left( \frac{\Ocal_P}{I_{f,P}+\gM_P^\e}\right)
  = \dim_k\left( \frac{M}{\gM_P^\e M} \right)
  = \sum_{i=0}^{\e-1} \dim_k \left(\frac{ \gM_P^i M}{\gM_P^{i+1} M} \right).
\end{equation}
If $\e\le r-1$, then every term in the sum is at least~$1$, so we get
a lower bound of~$\e$ as desired.  On the other hand, if $\e\ge r$,
then~\eqref{eqn:dfPeqsumr1} tells us that the sum
in~\eqref{eqn:dimkOPIfPMPe} is equal to~$e_f(P)$.
\end{proof}

\begin{proof}[Proof of Proposition $\ref{proposition:efPgeepsilon}$]
We consider a polynomial ring
\[
  \ZZ[\bfa,\bfx,\bfu]
\]
where $\bfa := \bigl(a_\rho(\bfe)\bigr)$ are the coefficients
of~$f_0,\ldots,f_N$ as described by~\eqref{eqn:frhoxdef}, and
where~$\bfx=[x_0,\ldots,x_N]$ and~$\bfu=[u_0,\ldots,u_N]$ are
independent sets of coordinates on~$\PP^N$. The idea is that~$\bfx$ is
used for the variables in~$f$, while~$\bfu$ is a point at which we
want to compute the multiplicity of~$f$.  We view~$\ZZ[\bfa,\bfx,\bfu]$ as a
trihomogeneous polynomial ring, so
\[
  \Proj \ZZ[\bfa,\bfx,\bfu] \cong \PP_\bfa^\dimEnd\times\PP_\bfx^N\times\PP_\bfu^N,
\]
where the subscripts indicate the homogeneous coordinates on the
three factors.

We fix an integer~$\e\ge0$ and want to describe the set
of~$(\bfa,\bfx,\bfu)$ such that~$f_\bfa(\bfx)$ has multiplicity at
least~$\e$ at~$\bfu$. It suffices to do this on the open
sets~$x_iu_j\ne0$ for all $0\le i,j\le N$, and by symmetry it is
enough to work on the open set~$x_0u_0\ne0$. So we dehomogenize by
setting
\[
  X_i := x_i/x_0 \quad\text{and}\quad U_i := u_i/u_0
\]
and we let
\[
  R := \ZZ[\bfa,\bfX,\bfU] \quad\text{with}\quad
  \Proj \ZZ[\bfa] \times \Spec \ZZ[\bfX,\bfU] \cong \PP_\bfa^\dimEnd \times \AA^N_\bfX \times \AA^N_\bfU.
\]

We consider the prime ideal~$\gM$ of~$R$ defined by 
\[
  \gM := (X_1-U_1,X_2-U_2,\ldots,X_N-U_N)\subset R.
\]

We let~$R_\bfU$ be the localization of $R$ at $\gM$,
\[
  R_\bfU := \left\{ {A}/{B} : A\in R,\; B\in R\setminus\gM \right\},
\]
we let~$\gM_\bfU$ be its maximal ideal,
\[
  \gM_\bfU := \left\{ {A}/{B} : A\in\gM,\; B\in R\setminus\gM \right\},
\]
and we define an ideal~$I_{\bfU,f}$ in~$R_\bfU$ by\footnote{The
  polynomial~$f_i$ is bihomogeneous of bidegree~$(1,d)$ in the
  variables~$\bfa$ and~$\bfx$, and the notation $f_i(\bfX)$ is
  shorthand for $f_i(\bfX):=x_0^{-d}f_i(\bfx)=f_i(1,X_1,\ldots,X_N)$.}
\[
  I_{\bfU,f} = \bigl( f_i(\bfX)f_j(\bfU) - f_j(\bfX)f_i(\bfU) \bigr)_{1\le i,j\le N}.
\]
We note that $I_{\bfU,f}\subseteq\gM_\bfU$, since the functions in~$I_{\bfU,f}$
vanish at $\bfX=\bfU$.

If we take particular values for~$\bfa$ and~$\bfU$ defined over an
algebraically closed field~$k$, i.e., if we take a
particular~$f\in\End_d^N(k)$ and a particular point~$P\in\PP^N(k)$,
then unsorting the definitions, we see that
\[
  e_f(P) = \dim_k R_{\bfU(P)}/I_{\bfU(P),f}.
\]
Further, Lemma~\ref{lemma:IfpcontainsMe} tells us that
for any integer~$\e\ge0$, we have
\begin{equation}
  \label{eqn:efPgeeiff}
   e_f(P) \ge \e
  \quad\Longleftrightarrow\quad
  \dim_k R_{\bfU(P)}/\bigl(I_{\bfU(P),f}+\gM_{\bfU(P)}^\e\bigr) \ge \e.
\end{equation}

We want to turn this around and describe a Zariski closed condition
on~$(\bfa,\bfU)$ that corresponds to requiring~$e_f(P)\ge\e$.  We
note that~\eqref{eqn:efPgeeiff} tells us that it suffices to work
modulo the ideal~$\gM_\bfU^\e$, and that the quotient ring $R_\bfU/\gM_\bfU^\e$
is a free~$\ZZ[\bfa,\bfU]$-module with basis given by the following
finite set of polynomials:
\[
  \Scal := \bigl\{ (X_1-U_1)^{\d_1} (X_2-U_2)^{\d_2} \cdots (X_N-U_N)^{\d_N}
  : \d_1+\cdots+\d_N<\e \bigr\}.
\]
From~\eqref{eqn:efPgeeiff}, we see that we need to describe the locus where
\[
  \rank_{\ZZ[\bfa,\bfU]}  R_\bfU/\bigl(I_{\bfU,f}+\gM_\bfU^\e\bigr) \ge \e.
\]

To ease notation, we let
\[
  n := \#\Scal,\quad
  A := \ZZ[\bfa,\bfU],\quad
  \overline{R}:=R_\bfU/\gM_\bfU^\e,
  \quad
  \overline{R}':=R_\bfU/\bigl(I_{\bfU,f}+\gM_\bfU^\e\bigr).
\]
As noted earlier, the ring~$\overline{R}$ is a free~$A$-module with
basis~$\Scal$, so $n=\rank_A\overline{R}$. To form~$\overline{R}'$, we take the quotient
of~$\overline{R}$ by the elements of the ideal~$I_{\bfU,f}$,
where~$I_{\bfU,f}$ is generated as an~$R_\bfU$-module by the finitely many
polynomials $f_i(\bfX)f_j(\bfU)-f_j(\bfX)f_i(\bfU)$. And although the rank
of~$I_{\bfU,f}$ as an~$A$-module is infinite, since we are also taking
the quotient by~$\gM_\bfU^\e$, it suffices to reduce the elements in~$I_{f,\bfU}$
modulo~$\gM_\bfU^\e$ and thereby express them as~$A$-linear combinations
of elements of~$\Scal$.

Explicitly, for each~$i,j\in\{1,\ldots,N\}$ and
each~$h(\bfX,\bfU)\in\Scal$, we multiply the~$(i,j)$ generator
of~$I_{f,P}$ by~$h$, reduce modulo~$\gM_\bfU^\e$, and write the result as
an $A$-linear combination of the basis~$\Scal$. Thus
\begin{multline}
  \label{eqn:fiUjfjUih}
  \bigl(f_i(\bfX)f_j(\bfU)-f_j(\bfX)f_i(\bfU)\bigr)h(\bfX,\bfU) \\*
  \equiv
  \smash[b]{  \sum_{g(\bfX,\bfU)\in\Scal} \a_g(i,j,h)g(\bfX,\bfU) \pmod{\gM_\bfU^\e} }\\*
  \text{with $\a_g(i,j,h)\in A$.}
\end{multline}
Then as an~$A$-module, the ring~$\overline{R}'$ is generated by the
elements of~$\Scal$, and the relations among those elements are
exactly those described by~\eqref{eqn:fiUjfjUih}.

We define~$M$ to be a matrix with columns indexed by~$g\in\Scal$,
with rows indexed by triples~$(i,j,h)$ with~$1\le i,j\le n$ and $h\in\Scal$.
We set the $g\times(i,j,h)$ entry of~$M$ to be~$\a_g(i,j,h)$. Then~$M$ is
an~$N^2n$-by-$n$ matrix with coefficients in~$A$.

We are identifying~$A^n$ with~$A^\Scal$, where
an element $\bfa\in A^\Scal$ is a map $\bfa:\Scal\to A$,
and the map $A^\Scal\to\overline{R}'$ is given by
\[
  A^\Scal\to\overline{R}',\quad
  \bfa\longmapsto \sum_{g(\bfX,\bfU)\in\Scal} \bfa_g \cdot g(\bfX,\bfU)
  \pmod{I_{\bfU,f}+\gM_\bfU^\e}.
\]
The kernel of this map is generated by the
relations~\eqref{eqn:fiUjfjUih}, where there is one relation for each
choice of~$(i,j,h)$. Hence the kernel is generated by the row-span
of~$M$, i.e., by $A^{N^2n}M$. This gives the isomorphism
\[
  \overline{R}' \cong A^n/A^{N^2n}M,
\]
so in particular we have
\[
   \rank_{\ZZ[\bfa,\bfU]} {R_\bfU}/({I_{\bfU,f}+\gM_\bfU^\e})
  = \rank_A \overline{R}'
  = \rank_A {A^n}/{A^{N^2n}M}.
\]

This yields
\[
  \rank_{\ZZ[\bfa,\bfU]} {R_\bfU}/({I_{\bfU,f}+\gM_\bfU^\e}) \ge \e
  \quad\Longleftrightarrow\quad
  \rank_A M \le n - \e,
\]
where~$\rank_A M$ denotes the row-rank, or equivalently the column
rank, of the matrix~$M$ as an~$A$-module.

Finally, we observe that $\rank_A M \le n - \e$ if and only if all of
the $n-\e+1$ square minors of~$M$ have determinant equal to~$0$. Each
of these determinantal equations defines a Zariski closed set, and
their intersection is the required locus in~$\End_d^N\times\PP^N$.
This completes the proof of
Proposition~\ref{proposition:efPgeepsilon}.
\end{proof}

\begin{remark}
For a given $f\in\End_d^N$, the condition $e_f(P)\ge2$ is a
codimension~$1$ condition, since
\[
  \bigl\{P\in\PP^N : e_f(P)\ge2\bigr\}=\Crit(f)
\]
is a subvariety of codimension~$1$ and degree $(N+1)(d-1)$.  More
generally for~$\e\ge3$, it seems interesting to study the subscheme
$\bigl\{P\in\PP^N:e_f(P)\ge\e\bigr\}$ and to understand how it varies
with~$f$.  
\end{remark}

\begin{proof}[Proof of Theorem~$\ref{theorem:EnddNPweighted}$]
The scheme $\End_d^N[\Pcal']$ as constructed in
Proposition~\ref{proposition:EnddNPunweighted} is a closed subscheme
of~$\End_d^N\times(\PP^N)^\Vcal$, which is in turn an open subscheme
of~$\PP^\dimEnd\times(\PP^N)^\Vcal$.  For each~$v\in\Vcal^\circ$, we
consider the projection map
\[
  (\operatorname{id}\times\pi_v) : \End_d^N\times(\PP^N)^\Vcal \longrightarrow \End_d^N\times \PP^N,
\]
and we define~$\End_d^N[\Pcal]$ to be the intersection of the inverse images
\begin{multline}
  \label{eqn:capfPefPgev}
  \End_d^N[\Pcal] := \End_d^N[\Pcal'] \\
  \cap
  \bigcap_{v\in\Vcal^\circ} (1\times\pi_v)^{-1}
  \bigl\{ (f,\bfP)\in\End_d^N\times\PP^N : e_f(P_v)\ge\e(v) \bigr\}.
\end{multline}
Proposition~\ref{proposition:efPgeepsilon} tells us that the
individual sets in the intersection~\eqref{eqn:capfPefPgev} are closed
subschemes of~$\End_d^N\times\PP^N$, so their pull-backs by the
projection maps are closed, as is their intersection.
Since~$\End_d^N[\Pcal']$ is also a closed subscheme
of~$\End_d^N\times\PP^N$, the same is true of~$\End_d^N[\Pcal]$.
\par\noindent(a)\enspace
Let~$(T,f,\g)$ be as in the definition~\ref{definition:upsNdP}
of universal parameter space. Ignoring the weights for the moment,
Proposition~\ref{proposition:EnddNPunweighted}(a) tells us that there
is a unique morphism $F:T\to\End_d^N[\Pcal']$ such that
$F_t=(f_t,\g_t)$ for all geometric points~$t$ of~$T$. To complete the
proof, it suffices to show that the assumption that every~$F_t$ is a
model for~$\Pcal$ implies that the image~$F(T)$ is contained in the
subscheme~$\End_d^N[\Pcal]$ of~$\End_d^N[\Pcal']$. But this is clear,
since~$F_t$ being a model for~$\Pcal$ means that
\[
  e_{f_t}\bigl(\g_t(v)\bigr) \ge \e(v)\quad\text{for all $v\in\Vcal^\circ$,}
\]
which is exactly the condition~\eqref{eqn:EnddNPefPgeev} needed to
ensure that~$(f_t,\g_t)$ is in the subscheme~$\End_d^N[\Pcal]$
of~$\End_d^N[\Pcal']$.
\par\noindent(b)\enspace
We observe that for any $(f,P)\in\End_d^N\times\PP^N$
and any~$\f\in\SL_{N+1}$, we have
\[
  e_{\f^{-1}\circ f\circ\f}\bigl(\f^{-1}(P)\bigr)
  = e_{\f^{-1}\circ f}(P) = e_f(P),
\]
where we use the fact that automorphisms of~$\PP^N$ leave the
multiplicity invariant.  Since we already know from
Proposition~\ref{proposition:EnddNPunweighted}(b)
that $\End_d^N[\Pcal']$ is~$\SL_{N+1}$-invariant, it follows that the
subscheme~$\End_d^N[\Pcal]$ is also~$\SL_{N+1}$-invariant.
\end{proof}

\begin{proof}[Proof of Proposition $\ref{proposition:EnddNP1toP2}$]
Let~$\a\in\Hom(\Pcal_1,\Pcal_2)$, where we view~$\a$ as a
map~$\a:\Vcal_1\to\Vcal_2$ on vertices.
We claim that this gives a morphism
\begin{align*}
  \a^* : \End_d^N \times(\PP^N)^{\Vcal_2}_\D
  &\longrightarrow   \End_d^N \times(\PP^N)^{\Vcal_1}_\D, \\
  (f,P) &\longmapsto (f,P\circ\a).
\end{align*}
The injectivity of~$\a$ ensures that $P\circ\a:\Vcal_1\to\PP^N$ avoids
the big diagonal.

It remains to check that if $(f,P)\in\End_d^N[\Pcal_2]$, then
$\a^*(f,P)\in\End_d^N[\Pcal_1]$. It suffices to check this at the
level of geometric points:
\begin{align*}
  (f,P)\in&\End_d^N[\Pcal_2]\\
  &\quad\Longleftrightarrow\quad
  f\circ P = P\circ \Phi_2~\text{and}~e_f\circ P \ge \e_2\\*[-1\jot]
  &\omit\hfill\hspace{4.5em} as maps $\Vcal_2^\circ\to\PP^N$ and  $\Vcal_2^\circ\to\ZZ_{\ge1}$,  respectively, \\
  &\quad\Longrightarrow\quad
  f\circ P\circ\a = P\circ \Phi_2\circ\a~\text{and}~e_f\circ P\circ\a\ge\e_2\circ\a \\*[-1\jot]
  &\omit\hfill since $\a:\Vcal_1^\circ\hookrightarrow\Vcal_2^\circ$,\\  
  &\quad\Longrightarrow\quad
  f\circ P\circ\a = P\circ\a\circ \Phi_1~\text{and}~e_f\circ P\circ\a\ge\e_1 \\*[-1\jot]
  &\omit\hfill since $\a\in\Hom(\Pcal_1,\Pcal_2)$, \\
  &\quad\Longrightarrow\quad
  (f,P\circ\a)\in\End_d^N[\Pcal_1].
\end{align*}
Hence $\a^*\bigl(\End_d^N[\Pcal_2]\bigr)\subseteq\End_d^N[\Pcal_1]$.
\end{proof}

\section{The $\SL_{N+1}$-Stable Locus of $\End_d^N\times(\PP^N)^n$}
\label{section:SLN1stableEndNPNnz} 
In Section~\ref{section:EnddNP} we constructed a parameter
space~$\End_d^N[\Pcal]$ whose points classify models for~$\Pcal$. We
would like to take the quotient of this space by the action
of~$\SL_{N+1}$ so as to construct a moduli space for the isomorphism
classes of models for~$\Pcal$. In order to obtain a nice quotient, we
need to prove that~$\End_d^N[\Pcal]$ is in the stable locus for the
action of~$\SL_{N+1}$ on~$\PP^\dimEnd\times(\PP^N)^\Vcal$ relative to
a suitable ample sheaf.

Our primary goal in this section is to prove a simple numerical
criterion on the integers $(N,d,n,m_0,\ldots,n_n)$ which ensures that
$\End_d^N\times(\PP^N)^n$ is $\SL_{N+1}$-stable relative to the
invertible sheaf~$\Ocal(m_0,\ldots,m_n)$. This criterion is given in
Corollary~\ref{corollary:dgenm1z}; see also Section~\ref{section:refinedP1}
for a more refined result for~$\PP^1$.
Although the corollary is an immediate consequence of the first implication in
Theorem~\ref{theorem:stableEnddNtimesPNnz}(b), it requires little
extra work to prove Theorem~\ref{theorem:stableEnddNtimesPNnz}, and
since the more precise statement in the theorem may be of independent
interest, we give the full proof. 

\begin{theorem}
\label{theorem:stableEnddNtimesPNnz}
Fix integers
\[
  N\ge1,\quad d\ge2,\quad n\ge0,\quad \bfm:=(m_0,\ldots,m_n)\in\ZZ_{\ge0}^n,
\]
and let
\[
  \Ocal(\bfm) := \Ocal_{\PP^\dimEnd\times(\PP^N)^n}(m_0,m_1,\ldots,m_n)
  \quad\text{and}\quad
  m_\Sigma := m_1 + m_2 + \cdots + m_n.
\]
For each geometric point $\bfP=(P_1,\ldots,P_n)\in(\PP^N)^n$, each
linear subspace $L\subseteq\PP^N$, and each integer $\e\ge0$,
let\footnote{We view~$C_\bfP(L)$ as a counting function for the points
  of~$\bfP$ that are in~$L$, where the points are weighted
  by~$m_1,\ldots,m_n$.}
\begin{align}
  \label{eqn:CPLDLdef1z}
  C_\bfP(L) &:= \sum_{{1\le\nu\le n,\, P_\nu\in L}} m_\nu, \\    
  \label{eqn:CPLDLdef2z}
  D_\e(L) &:=  \frac{m_\Sigma(\dim L + 1)+m_0(d-1)\codim L}{N+1} + \e.  
\end{align}
We let~$\SL_{N+1}$ act on~$\End_d^N\times(\PP^N)^n$ as
in~\eqref{eqn:SLactonEndPN}, i.e., the action on~$\End_d^N$ is via
conjugation and on each~$\PP^N$ via the inverse.  For a geometric
point
\[
  (f,\bfP)\in \End_d^N\times(\PP^N)^n,
\]
we consider the following statements\textup:
\begin{align*}
\SemiStable{}:&&& \text{$(f,\bfP)$ is~$\SL_{N+1}$-semistable relative to~$\Ocal(\bfm)$.}  \hspace{3em}\\
\SemiStable{\e}:&&&   C_\bfP(L)\le D_\e(L)\quad\text{for all linear $L\subsetneq\PP^N$.} \\
\Stable{}:&&& \text{$(f,\bfP)$ is~$\SL_{N+1}$-stable relative to~$\Ocal(\bfm)$.}  \\
\Stable{\e}:&&&   C_\bfP(L)< D_\e(L) \quad\text{for all linear $L\subsetneq\PP^N$.}  
\end{align*}
These statements satisfy the following implications\textup:
\par\noindent\begin{tabular}{lrlrlr}
\textup{(a)}&\SemiStable{0} &$\Longrightarrow$& \SemiStable{} &$\Longrightarrow$& \SemiStable{m_0} . \\[1\jot]
\textup{(b)}&\Stable{0} &$\Longrightarrow$& \Stable{} &$\Longrightarrow$& \Stable{m_0} . \\
\end{tabular}
\end{theorem}

\begin{remark}
Theorem~\ref{theorem:stableEnddNtimesPNnz} gives necessary and
sufficient conditions for a pair $(f,\bfP)\in\End_d^N\times(\PP^N)^n$
to be stable relative to~$\Ocal(\bfm)$. These conditions are
stated in terms of the numerical quantities~$N,d,n,\bfm$ and the
$n$-tuple of points~$\bfP$.  It would be nice if there were a single
condition depending on these quantities that was simultaneously
necessary and sufficient, but unfortunately, life is not that simple,
and any such condition must also depend on~$f$.
See Proposition~\ref{proposition:stableEnd1P1nz}(c) for an example.
We also note that Theorem~\ref{theorem:stableEnddNtimesPNnz} 
only gives the (semi)stable locus in~$\End_d^N\times(\PP^N)^n$; it does not
describe the full (semi)stable locus in $\PP^\dimEnd\times(\PP^N)^n$.
\end{remark}

Before proving Theorem~\ref{theorem:stableEnddNtimesPNnz}, we state a
corollary giving a simple numerical criterion for the stability of
$\End_d^N\times(\PP^N)^n$.  See Section~\ref{section:refinedP1} for a refined
statement for~$\PP^1$.

\begin{corollary}
\label{corollary:dgenm1z}
With notation as in Theorem~$\ref{theorem:stableEnddNtimesPNnz}$, we have:
\begin{align*}
  m_0 \ge \frac{m_\Sigma}{d-1}
  &\Longrightarrow
  \text{$\End_d^N\times(\PP^N)^n$ is $\SL_{N+1}$-semistable for $\Ocal(\bfm)$.} \\*
  m_0 > \frac{m_\Sigma}{d-1}
  &\Longrightarrow
  \text{$\End_d^N\times(\PP^N)^n$ is  $\SL_{N+1}$-stable for $\Ocal(\bfm)$.}
\end{align*}
\end{corollary}

We note that the $n=0$ case of
Theorem~\ref{theorem:stableEnddNtimesPNnz} is the classical result
that~$\End_d^N$ is~$\SL_{N+1}$-stable relative to~$\Ocal(1)$;
cf.\ Theorem~\ref{theorem:EnddNstableAutfinite}.  It is also
instructive to note that the~$\bfm=(0,1,\ldots,1)$ case of
Theorem~\ref{theorem:stableEnddNtimesPNnz} gives the following
classical result on stability of points in~$(\PP^N)^n$.  Thus our
Theorem~\ref{theorem:stableEnddNtimesPNnz} is, in some sense, an
amalgamation of Theorems~\ref{theorem:EnddNstableAutfinite}
and~\ref{theorem:classicalz}, although it is not a direct consequence
of these two results.

\begin{theorem}[Mumford]
\label{theorem:classicalz}
Let $n\ge1$ and $N\ge1$. A geometric point
$\bfP=(P_1,\ldots,P_n)\in(\PP^N)^n$ is~$\SL_{N+1}$-stable relative
to~$\Ocal(1,1,\ldots,1)$ if and only if\/\footnote{One can
  show that \eqref{eqn:numiPiinLz} implies that $\bfP$ has trivial stabilizer;
  see \cite[Proposition~3.6]{mumford:geometricinvarianttheory}.}
\begin{equation}
  \label{eqn:numiPiinLz}
  \#\{\nu : P_\nu\in L\} < \frac{n(\dim L + 1)}{N+1}
  \quad\text{for all linear $L\subsetneq\PP^N$.}
\end{equation}
\end{theorem}
\begin{proof}
Mumford proves this directly from the definition of stability in
\cite[Chapter~3]{mumford:geometricinvarianttheory}; see
Definition~3.7/Proposition~3.4 and Theorem~3.8.  A proof via the
numerical criterion may be found in~\cite[Example~3.3.21]{MR1659282}.
\end{proof}

\begin{proof}[Proof of Theorem~$\ref{theorem:stableEnddNtimesPNnz}$]
We note that every~$(f,\bfP)\in\End_d^N\times(\PP^N)^n$ has finite
stabilizer in~$\SL_{N+1}$, since
Theorem~\ref{theorem:EnddNstableAutfinite}(b) says that~$f$ itself has
finite stabilizer.

In order to use the numerical criterion
(Theorem~\ref{theorem:numcrit}) to study stability and semistability,
we choose an arbitrary non-trivial 1-para\-meter subgroup
$\ell:\GG_m\hookrightarrow\SL_{N+1}(k)$, and we choose coordinates
on~$\PP^N$ to diagonalize the action as follows:
\begin{equation}
  \label{eqn:1PSz}
  \ell(\a):=\begin{pmatrix}
  \a^{k_0}\\
  &\a^{k_1}\\
  && \ddots \\
  &&&\a^{k_N}\\
  \end{pmatrix}\in\SL_{N+1}
\end{equation}
with
\begin{equation}
  \label{eqn:k0kNsum0z}
  k_0\le k_1\le\cdots\le k_N,
  \quad
  k_0+k_1+\cdots+k_N=0,
  \quad\text{and}\quad
  k_N > 0.
\end{equation}

We write endomorphisms $f\in\End_d^N$ using the
notation~\eqref{eqn:frhoxdef} described in
Section~\ref{section:notationEnddN}.  An $\ell$-diagonalized basis
for the global sections of $\Ocal(1)$ on~$\End_d^N$ is given by
\[
  \bigl\{ a_\rho(\bfe) : \text{$0\le\rho\le N$ and $|\bfe|=d$} \bigr\}.
\]
The action of the 1-parameter subgroup~\eqref{eqn:1PSz} on this basis is
\[
  \ell(\a)\star a_\rho(\bfe) = \a^{k_0e_0+\cdots+k_Ne_N-k_\rho} a_\rho(\bfe).
\]

For later use, we define
\[
  \bfd := (d,0,0,\ldots,0),
\]
so~$a_\rho(\bfd)$ is the coefficient of~$x_0^d$ in~$f_\rho$.  The
assumption that $f\in\End_d^N$ implies that there exists some~$\rho'$
with
\begin{equation}
  \label{eqn:arhoprimedne0z}
  a_{\rho'}(\bfd) \ne 0,
\end{equation}
since otherwise~$f_0,\ldots,f_N$ would all vanish at~$[1,0,\ldots,0]$,
say.

Next we turn to~$(\PP^N)^n$.  We let~$\bigl[U(0),\ldots,U(N)\bigr]$ be
homogeneous coordinates on~$\PP^N$, so each~$U(i)$ is a section
of~$\Ocal_{\PP^N}(1)$, and the action of~$\ell$ on these sections is
via
\[
  \ell(\a)\star U(i) = \a^{-k_i}U(i).
\]
(The negative exponent is due to the fact that~$\ell$ acts on points
via the inverse matrix.)

More generally, for $1\le \nu\le n$, we let
\[
\bigl[U_\nu(0),\ldots,U_\nu(N)\bigr]
= \left\{ \begin{tabular}{@{}l@{}}
  homogeneous coordinates on\\
  the $\nu$th copy of~$\PP^N$ in~$(\PP^N)^n$\\
\end{tabular} \right\}.
\]
(Formally, we have $U_\nu(i) := U(i)\circ\operatorname{proj}_\nu$,
where $\operatorname{proj}_\nu:(\PP^N)^n\to\PP^N$ is projection on the
$\nu$th factor.)

For our given point~$\bfP=(P_1,\ldots,P_n)\in(\PP^N)^n$, we let
\begin{align}
  \label{eqn:jnudefz}
  j_\nu &:= \text{largest $j$ so that the $j$th coordinate of $P_\nu$ is non-zero} \\*
  &:= \text{the index with $U_\nu(j_\nu)\ne0$ and $U_\nu(j)=0$ for $j>j_\nu$.} \notag
\end{align}
Then with these definitions, we have
\begin{equation}
  \label{eqn:bfUbfjbfPne0z}
  U_\nu(j_\nu)\ne 0\quad\text{for all $1\le\nu\le n$.}
\end{equation}

We now consider a general $\ell$-diagonalized global section
of~$\Ocal(\bfm)$.  Such a section is a product of~$m_0$ copies of
sections of the form~$a_\rho(\bfe)$ multiplied by~$m_1$ copies of
sections of the form~$U_1(i)$, multiplied by~$m_2$ copies of sections
of the form~$U_2(i)$, etc. Thus such a section has the form
\begin{equation}
  \label{eqn:arhoeUibasisLz}
  s := \prod_{\mu=1}^{m_0} a_{\rho_\mu}(\bfe_\mu)
  \cdot \prod_{\nu=1}^n  \prod_{\l=1}^{m_\nu} U_\nu(i_{\nu,\l}) .
\end{equation}
The section~$s$ depends on the choice of the
quantities
\[
  \rho_\mu,\quad \bfe_\mu=(e_{\mu,0},\ldots,e_{\mu,N}),\quad i_{\nu,\l},
\]
but for convenience we omit this dependence from the notation and
always take~$s$ to be defined by~\eqref{eqn:arhoeUibasisLz}.

The action of~$\ell$ on~$s$ is given by
\begin{equation}
  \label{eqn:actiononsz}
  \ell(\a)\star s = \a^{E(s)}s
  \;\text{with}\;
  E(s) := 
  \sum_{\mu=1}^{m_0} \biggl( \sum_{\s=0}^N k_\s e_{\mu,\s} - k_{\rho_\mu} \biggr) 
  - \sum_{\nu=1}^n  \sum_{\l=1}^{m_\nu}   k_{i_{\nu,\l}}    .
\end{equation}

For $-1\le r\le N$, we define an increasing sequence of nested linear
subspaces
\begin{equation}
  \label{eqn:LrUr1uN0z}
  L_r := \{ U(r+1)=\cdots=U(N)=0 \} \subset \PP^N,
\end{equation}
so in particular $L_{-1}=\emptyset$ and $L_N=\PP^N$. We have
\begin{equation}
  \label{eqn:dimLrz}
  \dim L_r = r\quad\text{for all $0\le r\le N$.}
\end{equation}

We note that the definition of~$j_\nu$ says exactly that
\begin{equation}
  \label{eqn:PnuinLjLj1xz}
  P_\nu \in L_{j_\nu} \setminus L_{j_\nu-1}.
\end{equation}
This criterion plays the key role in relating the numerical invariant
to membership of the~$P_\nu$ in linear subspaces.

Using the notation~$C_\bfP(L)$ and~$D_\e(L)$ defined
by~\eqref{eqn:CPLDLdef1z} and~\eqref{eqn:CPLDLdef2z}, our next step is
to prove the following two formulas, which will be needed later:
\begin{align}
  \label{eqn:bfkbfjbfPsumz}
   \sum_{r=0}^N C_\bfP(L_r) (k_r - k_{r+1}) &= \sum_{\nu=1}^n m_\nu k_{j_\nu}, \\
  \label{eqn:sumDLrkrz}
  \sum_{r=0}^N D_0(L_r)(k_r-k_{r+1}) &= m_0(d-1)k_0,
\end{align}
where for convenience we set $k_{N+1} = 0$.

In order to prove~\eqref{eqn:bfkbfjbfPsumz}, we compute
\begin{align*}
  \sum_{\nu=1}^n m_\nu k_{j_\nu}
  &= \sum_{r=0}^N \biggl(\sum_{\substack{1\le\nu\le n\\ j_\nu=r\\}} m_\nu\biggr) k_r
    &&\text{grouping $j_\nu$ by value,}  \\
  &= \sum_{r=0}^N \biggl(\sum_{\substack{1\le\nu\le n\\ P_\nu\in L_r\setminus L_{r-1}\\}} m_\nu\biggr) k_r
    &&\text{from \eqref{eqn:PnuinLjLj1xz},}  \\
  &= \sum_{r=0}^N C_\bfP(L_r)k_r - \sum_{r=0}^N C_\bfP(L_{r-1})k_r
    &&\text{definition of $C_\bfP(L)$,} \\
  &= 
    \sum_{r=0}^N C_\bfP(L_r)k_r - \sum_{r=0}^N C_\bfP(L_r)k_{r+1} 
    &&\text{using $L_{-1}=\emptyset$, $k_{N+1}=0$,}  \\*
  &= \sum_{r=0}^N C_\bfP(L_r) (k_r - k_{r+1}).
\end{align*}
This completes the proof of~\eqref{eqn:bfkbfjbfPsumz}.

We now turn to the proof of~\eqref{eqn:sumDLrkrz}.  To ease notation,
we let
\[
  A = \frac{m_\Sigma-m_0(d-1)}{N+1}\quad\text{and}\quad
  B = \frac{m_\Sigma+m_0(d-1)N}{N+1}.
\]
Using $\dim L_r=r$ and $\codim L_r=N-r$, we see that
from~\eqref{eqn:CPLDLdef2z} that
\begin{equation}
  \label{eqn:DLrArBz}
  D_0(L_r) = \frac{m_\Sigma(r+1)+m_0(d-1)(N-r)}{N+1} = Ar + B.
\end{equation}
Then 
\begin{align*}
  \sum_{r=0}^N D_0(L_r)(k_r-k_{r+1})
  &= \sum_{r=0}^N (Ar+B)(k_r-k_{r+1}) 
  \quad\text{from \eqref{eqn:DLrArBz},} \\
  &=  \sum_{r=0}^N Ark_r - \sum_{r=0}^N Ark_{r+1} + B k_0 \\
  &=  \sum_{r=0}^N Ark_r - \sum_{r=1}^{N+1} A(r-1)k_{r}  + B k_0 \\
  &=  A\sum_{r=1}^N k_r + Bk_0
  \quad\text{since $k_{N+1}=0$,} \\
  &= (-A+B)k_0
  \quad\text{since $\textstyle\sum_{r=0}^Nk_r=0$ from \eqref{eqn:k0kNsum0z},} \\
  &= m_0(d-1)k_0  \quad\text{since $B-A=m_0(d-1)$.}
\end{align*}
This completes the proof of~\eqref{eqn:sumDLrkrz}.

We now resume the proof of Theorem~\ref{theorem:stableEnddNtimesPNnz},
so we consider the product $\End_d^N\times(\PP^N)^n$ and the
$\ell$-diagonalized global sections~$s$ for the invertible sheaf
${\Ocal(\bfm)}=\Ocal(m_0,m_1,\ldots,m_n)$ as defined
by~\eqref{eqn:arhoeUibasisLz}.

A key observation is that we can specify parameters for~$s$ so
that~$s$ does not vanish at our given point~$(f,\bfP)$:
\begin{equation}
  \label{eqn:aUne0fPz}
  \text{$s\ne0$ at $(f,\bfP)$ if}
  \quad\left\{\begin{aligned}
    \bfe_\mu &= \bfd&&\text{for all $\mu$,}\\
    \rho_\mu &= \rho'&&\text{for all $\mu$, and}\\ 
    i_{\nu,\l} & =j_\nu&&\text{for all~$\nu$ and all $\l$,}\\
  \end{aligned}
  \right.
\end{equation}
where $\bfd$ and~$\rho'$ are defined
by~\eqref{eqn:arhoprimedne0z} and~$j_\nu$ is defined
by~\eqref{eqn:jnudefz}.

\par\noindent
\framebox{\textbf{Proof of
  \SemiStable{0} $\boldsymbol\Longrightarrow$ \SemiStable{} and
  \Stable{0} $\boldsymbol\Longrightarrow$ \Stable{}  } }
\par
The numerical criterion (Theorem~\ref{theorem:numcrit}) says
that~$(f,\bfP)$ is semistable relative to ${\Ocal(\bfm)}$ if and only
if $\mu^{\Ocal(\bfm)}\bigl((f,\bfP),\ell\bigr)\ge0$ for every
1-parameter subgroup $\ell:\GG_m\to\SL_{N+1}$, and similarly for
stable with a strict inequality for~$\mu$. Without loss of generality,
we choose coordinates so that~$\ell$ is diagonalized and has the form
given by~\eqref{eqn:1PSz} and~\eqref{eqn:k0kNsum0z}, and then we
write~$f$ and~$\bfP$ using the $\ell$-diagonalized global sections~$s$
of~${\Ocal(\bfm)}$ described in~\eqref{eqn:arhoeUibasisLz}. Having
done this, and recalling from~\eqref{eqn:actiononsz} the notation
$\ell\star s=\a^{E(s)}s$, we estimate the numerical invariant as
follows:
\begin{align*}
  \mu^{\Ocal(\bfm)}&\bigl((f,\bfP),\ell\bigr) \\
  &= \max \bigl\{ -E(s) : \text{$s$ is defined by \eqref{eqn:arhoeUibasisLz} and $s(f,\bfP)\ne0$} \bigr\} \\
  &= \max_{s(f,\bfP)\ne0} - \left(  \sum_{\mu=1}^{m_0} \biggl( \sum_{\s=0}^N k_\s e_{\mu,\s} - k_{\rho_\mu} \biggr) 
  - \sum_{\nu=1}^n  \sum_{\l=1}^{m_\nu}   k_{i_{\nu,\l}} \right)  \quad\text{from~\eqref{eqn:actiononsz},} \\\\
  & \ge \sum_{\mu=1}^{m_0} (  -k_0 d + k_{\rho'} ) 
  + \sum_{\nu=1}^n  \sum_{\l=1}^{m_\nu}   k_{j_\nu}
  \quad\text{using the $s$ from~\eqref{eqn:aUne0fPz},} \\
  &\ge - m_0 k_0 (d-1)  + \sum_{\nu=1}^n m_\nu k_{j_\nu} 
  \quad\text{since $k_0=\min k_\rho$ from \eqref{eqn:k0kNsum0z},} \\
  &= 
  \smash[b]{ - \sum_{r=0}^N D_0(L_r)(k_r-k_{r+1}) + \sum_{r=0}^N C_\bfP(L_r) (k_r - k_{r+1}) } \\
  &\omit\hfill\text{from \eqref{eqn:bfkbfjbfPsumz} and \eqref{eqn:sumDLrkrz},} \\
  &= \sum_{r=0}^N \bigl(D_0(L_r)-C_\bfP(L_r)\bigr) \bigl(k_{r+1}-k_r\bigr).
\end{align*}
We next use the fact that $L_N=\PP^N$ to observe that
\[
  D_0(L_N)=C_\bfP(L_N)=m_\Sigma,
\]
so the $r=N$ term in the sum vanishes. Hence
\begin{equation}
  \label{eqn:mugesumr0N1z}
  \mu^{\Ocal(\bfm)}\bigl((f,\bfP),\ell\bigr) 
  \ge \sum_{r=0}^{N-1} \bigl(D_0(L_r)-C_\bfP(L_r)\bigr)
  \bigl( \underbrace{k_{r+1}-k_r}_{\text{${}\ge0$ from \eqref{eqn:k0kNsum0z}}} \bigr).
\end{equation}
We stress that eliminating the $r=N$ term is crucial, since the
inequality $k_{r+1}-k_r\ge0$ is not true for~$r=N$.

Suppose now that~\SemiStable{0} holds. Applying~\SemiStable{0} with
$L=L_r$ tells us that $C_\bfP(L_r)\le D_0(L_r)$ for all $0\le r<N$,
while as noted earlier, we have $k_{r+1}\ge k_r$ for all
$0\le{r}<N$. (Again, note that this is not true for $r=N$.)
Thus the lower bound
for $\mu^{\Ocal(\bfm)}\bigl((f,\bfP),\ell\bigr)$ given
by~\eqref{eqn:mugesumr0N1z} is a sum of non-negative terms, and hence
$\mu^{\Ocal(\bfm)}\bigl((f,\bfP),\ell\bigr)\ge0$. The numerical
criterion (Theorem~\ref{theorem:numcrit}) then tells us
that~$(f,\bfP)$ is a semistable point, i.e.,~\SemiStable{} is true.

We turn now to stability and assume that~\Stable{0} holds.
Applying~\Stable{0} with $L=L_r$ gives
the strict inequality $C_\bfP(L_r)<D_0(L_r)$, so the lower bound for
$\mu^{\Ocal(\bfm)}\bigl((f,\bfP),\ell\bigr)$ in~\eqref{eqn:mugesumr0N1z} is
strictly positive unless
\[
  k_{r+1} = k_r \quad\text{for all $0\le r<N$.}
\]
But this would imply that $k_0=k_1=\cdots=k_N$, and then the
fact~\eqref{eqn:k0kNsum0z} that $\sum_{r=0}^Nk_r=0$ would force
$k_0=\cdots=k_N=0$, contradicting the assumption that~$\ell$ is a
non-trivial 1-parameter subgroup. Hence the lower bound is strictly
positive, i.e., $\mu^{\Ocal(\bfm)}\bigl((f,\bfP),\ell\bigr)>0$, and the
numerical criterion then tells us that~$(f,\bfP)$ is a stable point.
Hence~\Stable{} is true.
\par\noindent
\framebox{\textbf{Proof of
  \SemiStable{} $\boldsymbol\Longrightarrow$ \SemiStable{m_0} and
  \Stable{} $\boldsymbol\Longrightarrow$ \Stable{m_0}  } }
\par
Let $(f,\bfP)\in\End_d^N\times(\PP^N)^n$ and let $L\subset\PP^N$.  We
claim that there exists a non-trivial 1-parameter subgroup
$\ell:\GG_m\hookrightarrow\SL_{N+1}$, depending on $(N,d,n,\bfm,\bfP)$,
such that
\begin{equation}
  \label{eqn:muLleDmDPN1z}
  \mu^{\Ocal(\bfm)}\bigl((f,\bfP),\ell) \le
  \bigl(D_{m_0}(L) - C_\bfP(L)\bigr)(N+1).
\end{equation}

To prove this claim, we let
\[
  r := \dim L \quad\text{and}\quad  c := \#\{\nu : P_\nu\in L\}.
\]
Changing coordinates, we may assume without loss of generality that
\[
  L = \{U(r+1)=\cdots=U(N)=0\}\subset\PP^N.
\]
By definition, exactly~$c$ of the~$P_\nu$ are in~$L$, so relabeling
the~$P_\nu$ (hence also the $m_\nu$), we may assume that
\[
  P_1,P_2,\ldots,P_c \in L. 
\]
In other words, we have
\begin{equation}
  \label{eqn:Uine0impliesinulerz}
  U_\nu(i) = 0\quad\text{for all $1\le\nu\le c$ and all $i\ge r+1$.}
\end{equation}

We take for~$\ell$ the following 1-parameter subgroup:
\[
  \ell(\a) :=
  \bordermatrix{
    & \overbrace{\smash{\phantom{\a^{NNN}\;\dots\;\a^{NNN}}}}^{r+1}
    & \overbrace{\smash{\phantom{\a^{NN}\;\dots\;\a^{N}}}}^{N-r} \cr
    &\begin{array}{@{}c@{}c@{}c@{}}
      \a^{-(N-r)} \\
      & \ddots \\
      && \a^{-(N-r)} \\
    \end{array}
    \cr
    &&\hspace{-4\jot}
    \begin{array}{@{}c@{}c@{}c@{}}
      \a^{r+1} \\
      & \ddots \\
      && \a^{r+1} \\
     \end{array}
    \cr
  }
  \in\SL_{N+1}.
\]

It follows from~\eqref{eqn:Uine0impliesinulerz} and our choice
of~$\ell$ that if the section~$s$ defined by~\eqref{eqn:arhoeUibasisLz}
is non-zero at~$(f,\bfP)$, then in particular we must have
\begin{align}
  \label{eqn:sumnulambdakiz}
  \sum_{\nu=1}^n  \sum_{\l=1}^{m_\nu}   k_{i_{\nu,\l}}
  &= \sum_{\nu=1}^c  \sum_{\l=1}^{m_\nu}   \underbrace{k_{i_{\nu,\l}}}_{\text{$i_{\nu,\l}\le r$ from \eqref{eqn:Uine0impliesinulerz}}} +
    \sum_{\nu={c+1}}^n  \sum_{\l=1}^{m_\nu}   k_{i_{\nu,\l}} \notag \\
  &\le \sum_{\nu=1}^c  \sum_{\l=1}^{m_\nu}   -(N-r) +
    \sum_{\nu={c+1}}^n  \sum_{\l=1}^{m_\nu}  (r+1) \notag \\
  &= -(N-r) \sum_{\nu=1}^c m_\nu  +  \sum_{\nu=c+1}^n m_\nu (r+1) \notag \\
  &= -(N-r) C_\bfP(L)  + \biggl( \sum_{\nu=1}^n m_\nu - C_\bfP(L) \biggr)(r+1)  \notag \\    
  &= -(N-r) C_\bfP(L)  + \bigl( m_\Sigma - C_\bfP(L) \bigr)(r+1)  \notag \\
  &= - C_\bfP(L) (N+1) + m_\Sigma (r+1) . 
\end{align}

We combine this with the trivial bounds 
\begin{align}
  \label{eqn:sumksigmaemuz}
  -\sum_{\mu=1}^{m_0} \sum_{\s=0}^N k_\s e_{\mu,\s}
  &\le \sum_{\mu=1}^{m_0} \sum_{\s=0}^N (N-r) e_{\mu,\s} 
  = m_0(N-r)d, \\
  \label{eqn:summum0krhomuz}
  \sum_{\mu=1}^{m_0} k_{\rho_\mu}
  &\le m_0(r+1),
\end{align}
to estimate
\begin{align*}
  \mu^{\Ocal(\bfm)}\bigl(&(f,\bfP),\ell\bigr)  \\*
  &= \max \bigl\{ -E(s) : \text{$s$ is defined by \eqref{eqn:arhoeUibasisLz} and $s(f,\bfP)\ne0$} \bigr\} \\
  &= \max_{s(f,\bfP)\ne0}
  -\sum_{\mu=1}^{m_0} \sum_{\s=0}^N k_\s e_{\mu,\s} + \sum_{\mu=1}^{m_0} k_{\rho_\mu}
  + \sum_{\nu=1}^n  \sum_{\l=1}^{m_\nu}   k_{i_{\nu,\l}}
  \quad\text{from \eqref{eqn:actiononsz},} \\  
  &\le m_0(N-r)d + m_0(r+1) - C_\bfP(L) (N+1) + m_\Sigma (r+1) \\
  &\omit\hfill\text{from \eqref{eqn:sumnulambdakiz}, \eqref{eqn:sumksigmaemuz}, and \eqref{eqn:summum0krhomuz},} \\
  &= \bigl( D_{m_0}(L)  - C_\bfP(L) \bigr)(N+1)
  \quad\text{by definition of $D_\e(L)$.}
\end{align*}
This completes the proof of~\eqref{eqn:muLleDmDPN1z}.

We are going to prove contrapositives. So we assume first
that~\SemiStable{m_0} is false, which means that there is a
subspace~$L\subsetneq\PP^N$ such that $C_\bfP(L) > D_{m_0}(L)$. It
follows from~\eqref{eqn:muLleDmDPN1z} that there is an~$\ell$
satisfying $\mu^{\Ocal(\bfm)}\bigl((f,\bfP),\ell\bigr)<0$. The
numerical criterion (Theorem~\ref{theorem:numcrit}) then implies
that~$(f,\bfP)$ is not semistable, i.e.,~\SemiStable{} is false.

Similarly, if we assume that that~\Stable{m_0} is false, then there is a
subspace~$L\subsetneq\PP^N$ such that $C_\bfP(L) \ge D_{m_0}(L)$,
and~\eqref{eqn:muLleDmDPN1z} gives an~$\ell$ with
$\mu^{\Ocal(\bfm)}\bigl((f,\bfP),\ell\bigr)\le0$. So the numerical criterion
implies that~$(f,\bfP)$ is not stable, i.e.,~\Stable{} is false.
\end{proof}

\begin{proof}[Proof of Corollary $\ref{corollary:dgenm1z}$]
We always have the trivial inequality
\[
  C_\bfP(L)\le m_\Sigma,
\]
where the upper bound comes from the case that~$L$ contains every point in~$\bfP$.  So
condition~\SemiStable{0} is automatically satisfied provided
\begin{equation}
  \label{eqn:nlenr1md1Nrz}
  m_\Sigma \le \frac{m_\Sigma(r + 1)+m_0(d-1)(N-r)}{N+1} 
  \quad\text{for all $0\le r<N$,}
\end{equation}
and similarly~\Stable{0} is true if~\eqref{eqn:nlenr1md1Nrz} holds with a
strict inequality. Since
\[
  \frac{m_\Sigma(r+1)+m_0(d-1)(N-r)}{N+1} - m_\Sigma = \frac{(m_0(d-1)-m_\Sigma)(N-r)}{N+1},
\]
so we see that $m_0(d-1)\ge m_\Sigma$ ensures the required non-negativity, and
$m_0(d-1)> m_\Sigma$ ensures positivity.
\end{proof}

\section{Formal Models and Stability of $\End_d^N[\Pcal]$}
\label{section:stabilityend} 

In this section we describe the portrait parameter space
$\End_d^N[\Pcal]$ as a finite cover of the parameter space associated
to its subportrait of endpoints. In particular, if~$\Pcal$ is
preperiodic, i.e., has no endpoints, then we obtain a description of
the natural projection map $\End_d^N[\Pcal]\to\End_d^N$.  If the
endpoint portrait parameter space is $\SL_{N+1}$-stable, we use
standard GIT arguments to show that~$\End_d^N[\Pcal]$ is also
$\SL_{N+1}$-stable.

\begin{theorem}  
\label{theorem:EndENDpt}
Let $N\ge1$, let $d\ge2$, let~$\Pcal=(\Vcal^\circ,\Vcal,\F,\e)$ be a
portrait such that the scheme
\[
  \End_d^N[\Pcal] \subset\PP^\dimEnd\times(\PP^N)^\Vcal
\]
constructed in Theorem~$\ref{theorem:EnddNPweighted}$
is non-empty,\footnote{See Theorems~\ref{theorem:UnweightedNonempty}
  and~\ref{thm:dimMd1} and Proposition~\ref{proposition:whenisMPempty}
  for various necessary and sufficient conditions on~$\Pcal$ so that
  $\End_d^N[\Pcal]\ne\emptyset$, and formulas for
  $\dim\End_d^N[\Pcal]$.} let
\[
  \END_d^N[\Pcal] :=
  \text{Zariski closure of $\End_d^N[\Pcal]$ in $\End_d^N\times(\PP^N)^\Vcal$},
\]
and let
\[
  \Pi : \PP^\dimEnd\times(\PP^N)^\Vcal \longrightarrow \PP^\dimEnd \times (\PP^N)^{\Vcal\setminus\Vcal^\circ}
\]
be the projection map that discards the copies of~$\PP^N$ associated
to the vertices in~$\Vcal^\circ$, i.e., $\Pi$ discards the vertices
of~$\Pcal$ at which~$\F$ is defined.
\begin{parts}
\Part{(a)}
The map
\[
  \Pi_\Pcal : \END_d^N[\Pcal] \longrightarrow \End_d^N \times (\PP^N)^{\Vcal\setminus\Vcal^\circ}
\]
induced by~$\Pi$ is a finite morphism.\footnote{But note that if we
  replace~$\END_d^N[\Pcal]$ with $\End_d^N[\Pcal]$, then the map need
  not be proper; see Example~\ref{example:Endnotproper}.}
\Part{(b)}
Both $\END_d^N[\Pcal]$ and
$\End_d^N[\Pcal]\times(\PP^N)^{\Vcal\setminus\Vcal^\circ}$ are
$\SL_{N+1}$-invariant subschemes of $\PP^\dimEnd\times(\PP^N)^\Vcal$.
\Part{(c)}
Let $m$ be an integer satisfying\footnote{In particular, if~$\Pcal$ is
  a preperiodic portrait, it suffices to take $m=1$. One only needs
  larger values of~$m$ if~$\Pcal$ has endpoints, or in graph-speak, if
  the directed graph~$\Pcal$ has sinks.}
\begin{equation}
  \label{eqn:mgt1d1VVo}
  m > \frac{1}{d-1} \#(\Vcal\setminus\Vcal^\circ). 
\end{equation}
Then the schemes $\END_d^N[\Pcal]$ and $\End_d^N[\Pcal]$ are contained
in the $\SL_{N+1}$-stable locus of~$\PP^\dimEnd\times(\PP^N)^\Vcal$
relative to the invertible sheaf
$\Pi^*\Ocal_{\PP^\dimEnd\times(\PP^N)^{\Vcal\setminus\Vcal^\circ}}(m,1,\ldots,1)$.\footnote{More
  generally, if we let $n=\#(\Vcal\setminus\Vcal^\circ)$ and we
  replace~\eqref{eqn:mgt1d1VVo} with $m_0>(m_1+\cdots+m_n)/(d-1)$,
  then the indicated schemes are in the stable locus relative to
  $\Pi^*\Ocal(m_0,m_1,\ldots,m_n)$, and similarly in
  Theorem~\ref{theorem:MNdwithportrait}.}
\end{parts}
\end{theorem}

\begin{definition}
Let $\END_d^N[\Pcal]$ be as in Theorem~\ref{theorem:EndENDpt}. A point
\[
  \bigl(f,(P_v)_{v\in\Vcal}\bigr)\in\END_d^N[\Pcal](k)
\]
is called a \emph{formal model for the portrait~$\Pcal$ over the
  field~$k$}. Note that if we view~$P$ as a map $P:\Vcal\to\PP^N(k)$,
then~$(f,P)$ is a model for~$\Pcal$ as per
Definition~\ref{definition:unwtedmoduliprob} if and only if~$P$ is
injective.  However, there may be points in $\END_d^N[\Pcal](k)$ for
which this injectivity is not true.  On the other hand, for any
given~$f$, there may or may not be a non-injective map
$P:\Vcal\to\PP^N(k)$ for which~$(f,P)$ is in~$\END_d^N[\Pcal]$. This
is a familiar phenomenon in dynamics, even for polynomials
in~$\CC[z]$, where a point may have formal period strictly larger than
its actual period; cf.\ \cite[Section~4.1]{MR2316407}.  The
terminology ``formal period'' for rational maps on~$\PP^1$ is due
originally to Milnor.
\end{definition}

\begin{proof}[Proof of Theorem~$\ref{theorem:EndENDpt}$]
(a)\enspace
We first treat the case that the portrait~$\Pcal$ is unweighted, i.e.,
we assume that $\e(v)=1$ for all $v\in\Vcal^\circ$.  
We  modify the definition of~$\End_d^N[\Pcal]$ as described
by~\eqref{eqn:Psivvprime} and~\eqref{eqn:EnddNPcaldef} 
to create a somewhat larger subscheme by putting the big diagonal back
into~$(\PP^N)_\D^\Vcal$.  Explicitly,  we let
\[
  \overline{\Psi}_{v,v'} : \End_d^N\times(\PP^N)^\Vcal \to \PP^N\times\PP^N,
  \quad
  \overline{\Psi}_{v,v'} = (\pi_v\times\pi_{v'})\circ F,
\]
where note  that the domain now uses $(\PP^N)^\Vcal$ instead of
$(\PP^N)_\D^\Vcal$. We then define a scheme~$W$ by
\begin{equation}
  \label{eqn:EnddNPcaldefx}
  W := \bigcap_{v\in\Vcal^\circ}  \overline{\Psi}_{v,\Phi(v)}^{-1}(\D)
  \subset \End_d^N\times(\PP^N)^\Vcal.
\end{equation}
We start with three observations about~$W$.
\begin{parts}
\Part{(1)}
$W$ is a closed subscheme of
$\End_d^N\times(\PP^N)^\Vcal$. 
\emph{Proof\/}:\enspace
Since~$\D$ is a closed
subscheme of~$\PP^N\times\PP^N$, and since~$\overline{\Psi}_{v,w}$ is a morphism,
each $\overline{\Psi}_{v,w}^{-1}(\D)$ is a closed subscheme of
$\End_d^N\times(\PP^N)^\Vcal$.  Thus~$W$ is a finite intersection of
closed subschemes, hence~$W$ is a closed subscheme.
\Part{(2)}
$\End_d^N[\Pcal]$ is contained in~$W$.
\emph{Proof\/}:\enspace
From the definitions of~$\Psi$ and~$\overline{\Psi}$ we see that
$\Psi_{v,w}^{-1}(\D)\subseteq\overline{\Psi}_{v,w}^{-1}(\D)$, and then 
$\End_d^N[\Pcal]\subset W$ follows by taking an intersection
of these inclusions.
\Part{(3)}
$\END_d^N[\Pcal]$ is equal to the closure of~$\End_d^N[\Pcal]$ in~$W$. 
\emph{Proof\/}:\enspace
By definition,~$\END_d^N[\Pcal]$ is the closure of~$\End_d^N[\Pcal]$
in $\End_d^N\times(\PP^N)^\Vcal$. From~(2), it follows
that~$\END_d^N[\Pcal]$ is the closure of~$\End_d^N[\Pcal]$ in the
closure of~$W$, but from~(1), the closure of~$W$ is
just~$W$.\footnote{In general,~$W$ may be significantly
  larger than~$\END_d^N[\Pcal]$; see
  Example~\ref{example:WbiggerEND}.}
\end{parts}

To ease notation, we let
\[
  \Zcal=\Vcal\setminus\Vcal^\circ.
\]
We define
\[
  \Pi_W : W \longrightarrow \End_d^N \times (\PP^N)^\Zcal
\]
to be the restriction of~$\Pi$ to~$W$, so in particular~$\Pi_\Pcal$ is
the restriction of~$\Pi_W$ to~$\END_d^N[\Pcal]$.  Let~$k$ be an
algebraically closed field, 
and consider a
geometric point
\begin{equation}
  \label{eqn:PiPcalinvf}
  \bigl(f,(P_v)_{v\in\Vcal}\bigr) 
  \in W(k).
\end{equation}
By construction we have
\begin{equation}
  \label{eqn:fPvPFvEND}
  f(P_v) = P_{\F(v)} \quad\text{for all $v\in\Vcal^\circ$,}
\end{equation}
since the only change between the definition~\eqref{eqn:EnddNPcaldef}
of~$\End_d^N[\Pcal]$ and the definition~\eqref{eqn:EnddNPcaldefx}
of~$W$ is that we have dropped the requirement that the
points~$P_v$ be distinct. The use of the inverse image $\overline{\Psi}^{-1}_{v,\F(v)}(\Delta)$
in~\eqref{eqn:EnddNPcaldefx} still forces the geometric points
of~$W$ to satisfy~\eqref{eqn:fPvPFvEND}.

We start by verifying that the map~$\Pi_W$ is quasi-finite, i.e., has
finite fibers at all geometric points. Fix a geometric point
\[
  \bigl(f,(Q_z)_{z\in\Zcal}\bigr)\in\bigl(\End_d^N\times(\PP^N)^\Zcal\bigr)(k),
\]
and consider the points
\begin{equation}
  \label{eqn:fPvinvimagefQz}
  \bigl(f,(P_v)_{v\in\Vcal}\bigr) \in \Pi_W^{-1}\bigl(f,(Q_z)_{z\in\Zcal}\bigr).
\end{equation}
This means that $P_z=Q_z$ for all $z\in\Zcal$, and we need
to show that for each $v\in\Vcal^\circ$, there are only finitely many
possibilities for~$P_v$. We consider two cases.

First, if $v\in\Vcal^\circ$ is preperiodic for~$\F$,
then~\eqref{eqn:fPvPFvEND} implies that~$P_v$ is preperiodic for~$f$,
and the total length of the orbit is clearly bounded by~$\#\Vcal$.  We
now use the fact that a given endomorphism~$f$ has only finitely many
points of any given period,\footnote{The points of period dividing~$n$
  are the points in the intersection $\G_{\!f^n}\cap\D$, where
  $\G_{\!f^n}\subset(\PP^N)^2$ is the graph of~$f^n$
  and~$\D\subset(\PP^N)^2$ is the diagonal. Both~$\G_{\!f^n}$ and~$\D$
  are irreducible, so they intersect properly unless they are
  equal. But $\G_{\!f^n}=\D$ implies that $f^n$ is the identity map,
  contradicting $\deg(f^n)=d^n>1$.}  and the fact that each point
in~$\PP^N(k)$ has at most~$d$ preimages by~$f$.  Hence there are only
finitely many possibilities for~$P_v$ when~$v\in\Vcal^\circ$ is
preperiodic.

Second, if $v\in\Vcal^\circ$ is not preperiodic for~$\F$, then there
is some $n\ge1$ such that $\F^n(v)\in\Zcal$, where we note that $n$ is
bounded, e.g., we certainly have $n\le\#\Vcal$.
Then~\eqref{eqn:fPvPFvEND} and~\eqref{eqn:fPvinvimagefQz} imply 
that
\[
  f^n(P_v)=P_{\F^n(v)}=Q_{\F^n(v)},
\]
and since we have assumed that~$Q:\Zcal\to\PP^n(k)$ is fixed, this
determines the value of~$f^n(P_v)$.  Hence there are only finitely
many possible values for~$P_v$, since~$f$ is at most~$d$-to-$1$.

This completes the verification that the map~$\Pi_W$ is quasi-finite,
and the same is true of~$\Pi_\Pcal$, since~$\Pi_\Pcal$ is the
restriction of~$\Pi_W$ to the subscheme~$\END_d^N[\Pcal]$ of~$W$.

We next observe that the projection map
\[
  \End_d^N \times (\PP^N)^\Vcal \longrightarrow \End_d^N \times (\PP^N)^\Zcal
\]
is proper, since its fibers are products of projective
spaces. Thus~$\Pi_\Pcal$, which is the restriction of this map to the
closed subscheme~$\END_d^N[\Pcal]$ of $\End_d^N \times (\PP^N)^\Vcal$
is also proper.\footnote{This follows from: (1) If $Z$ is a closed
  subscheme of $X$, then the inclusion $Z\hookrightarrow X$ is
  proper. (2) A composition of proper morphisms is proper. Both~(1)
  and~(2) are exercises using the valuative criterion, or
  see~\cite[II.4.8]{hartshorne}.} We have now shown that~$\Pi_\Pcal$
is quasi-finite and proper,
so~\cite[Chapter~I,~Proposition~1.10]{MR559531} implies
that~$\Pi_\Pcal$ is finite. This completes the proof of~(a) for
unweighted portraits.
\par
For a general weighted portrait~$\Pcal$, let~$\Pcal'$ be the
associated unweighted portrait. We know from
Theorem~\ref{theorem:EnddNPweighted} that~$\END_d^N[\Pcal]$ is a
closed subscheme of~$\End_d^N[\Pcal']$, and the restriction of a
finite map to a closed subscheme is finite, which gives the desired
result for~$\End_d^N[\Pcal]$.
\par\noindent(b)\enspace 
We proved earlier in Theorem~\ref{theorem:EnddNPweighted}(b)
that $\End_d^N[\Pcal]$ is $\SL_{N+1}$-in\-var\-iant, from which it follows
that its Zariski closure~$\END_d^N[\Pcal]$ is likewise
$\SL_{N+1}$-invariant.
\par\noindent(c)\enspace
Every finite morphism is affine, so it follows from~(a)
that~$\Pi_\Pcal$ is an affine morphism.  We also remark that since
every geometric point of~$\End_d^N$ has
$0$-dimensional~$\SL_{N+1}$-stabilizer, the same is true for every
point in~$\End_d^N\times(\PP^N)^\Vcal$, and hence also true for every
point in the subscheme~$\END_d^N[\Pcal]$.  We are going to
use~\cite[Proposition~1.18]{mumford:geometricinvarianttheory}, which
says that if $F:X\to Y$ is a quasi-affine morphism commuting with the
action of a reductive group~$G$, and if~$\Lcal$ is a $G$-linearized
invertible sheaf on~$Y$, then
\begin{equation}
  \label{eqn:finvstable}
  F^{-1}\left(\left\{ \begin{tabular}{@{}c@{}}  
   $G$-stable points of $Y$ \\ relative to $\Lcal$\\
  \end{tabular}\right\}\right)
  \subseteq \left\{ \begin{tabular}{@{}c@{}}
   $G$-stable points of $X$\\    relative to $F^*\Lcal$\\
  \end{tabular}\right\}.
\end{equation}

Using notation from Corollary~\ref{corollary:dgenm1z}, we have
\[
  n=m_\Sigma=\#\Zcal=\#(\Vcal\setminus\Vcal^\circ)\quad\text{and}\quad m_0=m,
\]
so our assumption~\eqref{eqn:mgt1d1VVo} says that
$m_0>m_\Sigma/(d-1)$. This is exactly what is required to apply
Corollary~\ref{corollary:dgenm1z}, so we conclude that
\begin{equation}
  \label{eqn:EnddNPNnstabm1}
  \text{$\End_d^N\times(\PP^N)^\Zcal$ is $\SL_{N+1}$-stable
    for~$\Ocal(m,1,\ldots,1)$.}
\end{equation}

Applying this, we find that 
\begin{align*}
  \END_d^N[\Pcal]
  &= \Pi_\Pcal^{-1}(\End_d^N\times(\PP^N)^\Zcal) \\
  &= \Pi_\Pcal^{-1}\bigl(\{\text{$\SL_{N+1}$-stable points of $\End_d^N\times(\PP^N)^\Zcal$}\}\bigr)
  \quad\text{from \eqref{eqn:EnddNPNnstabm1},}\\*
  &\subseteq \bigl\{\text{$\SL_{N+1}$-stable points of $\END_d^N[\Pcal]$}\bigr\} \quad\text{from \eqref{eqn:finvstable}}\\*
  &\omit\hfill with  $G=\SL_{N+1}$,  $\Lcal=\Ocal_{\End_d^N\times \left(\PP^N\right)^\Zcal}(m,1,\ldots,1)$,\\*
  &\omit\hfill   $F=\Pi_\Pcal$, $X=\END_d^N[\Pcal]$, $Y=\End_d^N\times(\PP^N)^\Zcal$.%
\end{align*}
Hence~$\END_d^N[\Pcal]$ is contained in the $\SL_{N+1}$-stable locus
of $\PP^\dimEnd\times(\PP^N)^\Vcal$, which completes the proof of
Theorem~\ref{theorem:MNdwithportrait}.
\end{proof}

Combining Theorem~\ref{theorem:EndENDpt} with
Theorem~\ref{theorem:UnweightedNonempty}, a result that we will prove
later, leads to a useful fact about formal models for unweighted
portraits.

\begin{corollary}
Let~$\Pcal=(\Vcal^\circ,\Vcal,\F)$ be an unweighted portrait such
that~$\End_d^N[\Pcal]$ is non-empty, and let~$k$ be an algebraically
closed field of characteristic~$0$. Then every~$f\in\End_d^N(k)$
admits a formal model for the portrait~$\Pcal$ over the field~$k$.
This is true even if one first arbitrarily assigns values
in~$\PP^N(k)$ to the vertices in~$\Vcal\setminus\Vcal^\circ$.
\end{corollary}
\begin{proof}
As before, we let $\Zcal:=\Vcal\setminus\Vcal^\circ$.
  We first note that
\begin{align*}
  \dim \End_d^N\times(\PP^N)^{\Zcal}
  &= \dim \End_d^N + N\cdot\#\Zcal \\*
  &= \dim \End_d^N[\Pcal]
  \quad\text{from Theorem~\ref{theorem:UnweightedNonempty}(d),} \\*
  &= \dim \END_d^N[\Pcal]
  \begin{tabular}[t]{l}
    since  $\END_d^N[\Pcal]$ is the Zariski\\
    closure of $\End_d^N[\Pcal]$.\\
  \end{tabular}
\end{align*}
Next we use the fact from Theorem~\ref{theorem:EndENDpt}(a)
that
\[
  \Pi_\Pcal:\END_d^N[\Pcal]\to\End_d^N\times(\PP^N)^{\Zcal}
\]
is finite, so in particular it is proper, so the image
$\Pi_\Pcal\bigl(\END_d^N[\Pcal]\bigr)$ is a closed subscheme
of~$\End_d^N\times(\PP^N)^{\Zcal}$. Thus the image is both closed and
has dimension equal to the dimension
of the irreducible variety~$\End_d^N\times(\PP^N)^{\Zcal}$, hence the image is all
of~$\End_d^N\times(\PP^N)^{\Zcal}$.
\end{proof}

\begin{example}
\label{example:Endnotproper}
Theorem~\ref{theorem:EndENDpt} says in particular that if~$\Pcal$ is a preperiodic
portrait, then the map $\END_d^N[\Pcal]\to\End_d^N$ is finite and
the map $\End_d^N[\Pcal]\to\End_d^N$ is quasi-finite.  We note that
the latter map need not be finite, since it need not be proper. To
illustrate, let~$\Pcal$ be the portrait consisting of two fixed
points, and consider the family of maps
\[
  t \longmapsto f_t := z(z-t+1) \in \End_2^1,
\]
valid for all~$t$. On the set $t\ne0$, we can extend this to a family
\[
  t \longmapsto (f_t,0,t) \in \End_2^1[\Pcal],
\]
but we clearly cannot extend to $t=0$, since~$\End_2^1[\Pcal]$
classifies triples $(f,P,Q)$ with $f(P)=P$, $f(Q)=Q$, and $P\ne Q$,
so $(f_0,0,0)$ is not a point of~$\End_2^1[\Pcal]$.
\end{example}

\begin{example}
\label{example:WbiggerEND}
Let~$W$ be the scheme defined by~\eqref{eqn:EnddNPcaldefx}. We proved
that $\END_d^N[\Pcal]$ is contained in~$W$, but we want to give an
example showing that they need not be equal. Let~$\Pcal$ be the
portrait consisting of two fixed points. Then~$W$ contains the subscheme
\[
  \bigl\{ (f,P,Q) : \text{$P=Q$ and $f(P)=P$} \bigr\},
\]
which is isomorphic
to~$\End_d^N[\text{$1$~fixed~point}]$. But~$(f,P,P)\in\END_d^N[\Pcal]$
if and only if~$P$ is a fixed point of multiplicity strictly larger
than~$1$. For example, on~$\PP^1$ we have
\[
  (f,P,P)\in\END_d^1[\Pcal]
  \quad\Longleftrightarrow\quad
  f'(P)=1,
\]
where we write~$f'(P)$ for the multiplier of~$f$ at~$P$. Thus
for~$N=1$, the scheme~$W$ has two irreducible components, and for
higher values of~$N$ and/or more complicated portraits~$\Pcal$, it may
have many more components.
\end{example}

\section{The Moduli Space $\Moduli_d^N[\Pcal]$ for Portrait Models}
\label{section:modulispace} 

In this section we construct moduli spaces for dynamical systems
on~$\PP^N$ with a specified portrait.

\begin{theorem}
\label{theorem:MNdwithportrait}
Continuing with the notation from Theorem~$\ref{theorem:EndENDpt}$, we
assume as in that theorem that~$m$ is an integer satisfying
\begin{equation}
  \label{eqn:mgt1d1VVoMd}
  m > \frac{1}{d-1} \#(\Vcal\setminus\Vcal^\circ). 
\end{equation}
\begin{parts}
\Part{(a)}
There are well-defined geometric quotient schemes\footnote{If we need
  to indicate the dependence of the structure sheaf on~$m$, for
  example as in Theorem~\ref{theorem:MNdwithportrait}(c), then we write
  $\MODULI_d^N[m,\Pcal]$ and $\Moduli_d^N[m,\Pcal]$.}
\begin{align*}
  \MODULI_d^N[\Pcal] &:= \END_d^N[\Pcal]\GITQuot\SL_{N+1},\\*
  \Moduli_d^N[\Pcal] &:= \End_d^N[\Pcal]\GITQuot\SL_{N+1},
\end{align*}
satisfying\textup:
\begin{parts}
  \Part{(i)} The structure sheaf of $\MODULI_d^N[\Pcal]$ is the maximal subsheaf
  of $\Pi^*_\Pcal\Ocal_{\PP^\dimEnd \times (\PP^N)^{\Vcal \setminus \Vcal^\circ}}(m,1,\ldots,1)$ that is $\SL_{N+1}$-invariant.
  \Part{(ii)} $\MODULI_d^N[\Pcal]$ is proper over~$\Moduli_d^N$.
  \Part{(iii)} $\Moduli_d^N[\Pcal]$ is an open subscheme of $\MODULI_d^N[\Pcal]$.
\end{parts}
\Part{(b)}
The scheme~$\Moduli_d^N[\Pcal]$ is a coarse moduli space for the
$(N,d,\Pcal)$-moduli problem.
\Part{(c)}
Let~$\Pcal_1$ and~$\Pcal_2$ be portraits, and assume
that~$m$ satisfies~\eqref{eqn:mgt1d1VVoMd} for both~$\Pcal_1$
and~$\Pcal_2$.  Then the map described in
Proposition~$\ref{proposition:EnddNP1toP2}$ descends to give a natural
map
\[
  \Hom(\Pcal_1,\Pcal_2) \longrightarrow
  \Hom\bigl(\Moduli_d^N[m,\Pcal_2],\Moduli_d^N[m,\Pcal_1]\bigr).
\]
N.B. In order to obtain a map between the two moduli spaces, it is
necessary to use the same~$m$ for both.
\end{parts}
\end{theorem}

\begin{definition}
We set the convenient notation
\[
  \<\;\cdot\;\> : \End_d^N[\Pcal]\longrightarrow\Moduli_d^N[\Pcal]
\]
for the quotient map, i.e., if $(f,P)\in\End_d^N[\Pcal]$ is a model
for~$\Pcal$, then $\<f,P\>\in\Moduli_d^N[\Pcal]$ denotes the
$\SL_{N+1}$-equivalence class of $(f,P)$.
\end{definition}

\begin{remark}
One can define interesting subschemes of~$\Moduli_d^N[\Pcal]$ by
specifying that the automorphism group of the map~$f$ in the
pair~$(f,\bfP)$ has a specified subgroup. Thus for each finite
group~$\Gcal$, we define\footnote{We apologize for this clumsy
  notation, which will not be used further in this paper, but we want
  to keep~$\Gcal$ separate from the subgroup~$\Acal$ of~$\Aut(\Pcal)$
  that appears in the notation defined in
  Section~\ref{subsection:MdNPmodA}.}
\[
  {}_\Gcal\Moduli_d^N[\Pcal]
  := \left\{ (f,\bfP)\in\Moduli_d^N[\Pcal] :
  \begin{tabular}{@{}l@{}}
    $\Aut(f)$ contains a sub-\\ group isomorphic to $\Gcal$\\
  \end{tabular}
  \right\}.
\]
Levy~\cite{MR2741188} proves that ${}_\Gcal\Moduli_d^N = {}_\Gcal\Moduli_d^N[\emptyset]$ is a Zariski
closed subvariety of~$\Moduli_d^N$ for all~$\Gcal$, and that it is
non-empty for at most finitely many~$\Gcal$, hence the same is true
for ${}_\Gcal\Moduli_d^N[\Pcal] \subseteq \Moduli_d^N[\Pcal]$.

A nice example is ${}_{\ZZ/2\ZZ}\Moduli_2^1[\emptyset]$, which is a
cuspidal cubic curve in
$\Moduli_2^1[\emptyset]=\Moduli_2^1\cong\AA^2$. The cusp is the unique
point in~${}_{\Scal_3}\Moduli_2^1[\emptyset]$.  Let~$p$ be prime, and
let $\Ccal_n$ denote the unweighted portrait consisting of a single
$n$-cycle. Then Manes~\cite{MR2524186} shows that the moduli space
${}_{\ZZ/p\ZZ}\Moduli_d^1[\Ccal_{p\ell}]$ is reducible for infinitely
many primes~$\ell$.
See also~\cite{arxiv1607.05772} for a catalog of
the schemes ${}_\Gcal\Moduli_2^2[\emptyset]$ for groups
satisfying~$\#\Gcal\ge3$.  For example, the largest group for
which~${}_\Gcal\Moduli_2^2[\emptyset]$ is non-empty is the group
$\Gcal=(\ZZ/7\ZZ)\rtimes(\ZZ/3\ZZ)$ of order~$21$.
\end{remark}

\begin{proof}[Proof of Theorem~$\ref{theorem:MNdwithportrait}$]
(a)\enspace
This follows directly from Theorem~\ref{theorem:EndENDpt}, combined with 
standard facts from geometric invariant theory regarding geometric quotients
of schemes of stable points.
\par\noindent(b)\enspace
This is immediate from the fact that~$\End_d^N[\Pcal]$ is a universal
parameter space for models of~$\Pcal$, combined with the fact
that~$\Moduli_d^N[\Pcal]$ is a geometric quotient of~$\End_d^N[\Pcal]$
for the action of~$\SL_{N+1}$ that matches the equivalence relation on
models of~$\Pcal$ used to define the~$(N,d,\Pcal)$-moduli problem.
\par\noindent(c)\enspace
Proposition~\ref{proposition:EnddNP1toP2} says that
every portrait morphism  $\a:\Pcal_1\to\Pcal_2$
induces a scheme-theoretic morphism of the associated parameter spaces
\[
  \a^* : \End_d^N[\Pcal_2] \xrightarrow{\;\;\sim\;\;} \End_d^N[\Pcal_1] 
\]
via $\a^*(f,\bfP)=(f,\bfP\circ\a)$.  Further, for~$\f\in\SL_{N+1}$
and~$(f,\bfP)\in\End_d^N[\Pcal_2]$ we have
\begin{align*}
  \a^*\circ\f \bigl(f,(P_v)_{v\in\Vcal}\bigr)
  &= \a^*\bigl(f^\f,(\f^{-1}P_v)_{v\in\Vcal}\bigr)
  = \bigl(f^\f,(\f^{-1}P_{\a(v)})_{v\in\Vcal}\bigr), \\
  \f\circ\a^* \bigl(f,(P_v)_{v\in\Vcal}\bigr)
  &= \f \bigl(f,(P_{\a(v)})_{v\in\Vcal}\bigr)
  = \bigl(f^\f,(\f^{-1}P_{\a(v)})_{v\in\Vcal}\bigr).
\end{align*}
Thus the action of~$\a^*$ commutes with the action of~$\SL_{N+1}$,
so~$\a^*$ descends to a morphism of the quotient schemes. (We should
also point out that~$\a^*$ preserves~$\Ocal(m,1,\ldots,1)$, since it
simply permutes the factors of~$\PP^N$.)
\end{proof}

\subsection{The Moduli Space $\Moduli_d^N{[\Pcal|\Acal]}$ for Portraits Modulo Automorphisms}
\label{subsection:MdNPmodA}

Applying Theorem~\ref{theorem:MNdwithportrait}(c) with
$\Pcal_1=\Pcal_2=\Pcal$, we see that every automorphism of an
unweighted portrait induces a scheme-theoretic automorphism of the
associated moduli space~$\Moduli_d^N[\Pcal]$.  The group~$\Aut(\Pcal)$
is finite, since it is a subgroup of the symmetric group on the
set~$\Vcal$.\footnote{We generally assume that~$\Vcal\ne\emptyset$,
  and in any case we define the symmetry group of the empty set to be
  the trivial group.}  Quotients of schemes by finite groups of
automorphisms always exist as GIT geometric quotients. Hence for any
subgroup~$\Acal\subset\Aut(\Pcal)$, we obtain a GIT quotient space.
We formally state this result as a corollary.

\begin{corollary}
\label{corollary:MdNPA}
Let $\Pcal $ be a portrait, let~$\Acal\subseteq\Aut(\Pcal)$ be a subgroup of
the automorphism group of~$\Pcal$. Then the natural action
of~$\Acal$ on~$\End_d^N[\Pcal]$ commutes with the action of~$\SL_{N+1}$,
and hence there exists a GIT geometric quotient space
\[
  \Moduli_d^N[\Pcal|\Acal] := \End_d^N[\Pcal]\GITQuot(\Acal\cdot\SL_{N+1})
  \cong \Moduli_d^N[\Pcal]\GITQuot\Acal.
\]
\end{corollary}

\begin{definition}
\label{definition:MdNP10}
With~$\Moduli_d^N[\Pcal|\Acal]$ as defined in
Corollary~\ref{corollary:MdNPA}, we set the following
notation:\footnote{This notation is motivated by the notation for the
  elliptic modular curves~$X_1(n)$ and~$X_0(n)$, where~$X_1(n)$
  classifies pairs~$(E,P)$ consisting of an elliptic curve and a point
  of exact order~$n$, and~$X_0(n)=X_1(n)\GITQuot(\ZZ/n\ZZ)^*$, the
  action of~$k\in(\ZZ/n\ZZ)^*$ being to send~$(E,P)$ to~$(E,kP)$.}
\[
  \Moduli_d^N[\Pcal|1] := \Moduli_d^N[\Pcal]
  \quad\text{and}\quad
  \Moduli_d^N[\Pcal|0] := \Moduli_d^N[\Pcal|\Aut(\Pcal)].
\]
\end{definition}

By analogy with theorems and conjectures on moduli spaces of curves
and abelian varieties, we make the following conjecture.  See
Theorem~\ref{theorem:M21Cn}(c) for an example of
a~$\Moduli_d^1[\Pcal]$ that is a surface of general type.

\begin{conjecture}
  Let $N\ge1$ and $d\ge2$.
\begin{parts}
  \Part{(a)}
  There is a constant $C_1(N,d)$ so that for all unweighted
  preperiodic portraits~$\Pcal=(\Vcal,\F)$ satisfying
  $\#\Vcal\ge{C_1(N,d)}$ and $\Moduli_d^N[\Pcal]\ne\emptyset$, the variety
  $\Moduli_d^N[\Pcal]\otimes_\ZZ\CC$ is an irreducible variety of general type.
  \Part{(b)}
  There is a constant $C_0(N,d)$ so that for all unweighted
  preperiodic portraits~$\Pcal=(\Vcal,\F)$ satisfying
  $\#\Vcal\ge{C_0(N,d)}$ and  $\Moduli_d^N[\Pcal|0]\ne\emptyset$, the variety
  $\Moduli_d^N[\Pcal|0]\otimes_\ZZ\CC$ is an irreducible variety of general type.
\end{parts}
\end{conjecture}

\subsection{The Moduli Space for Maps with Points of Exactly Specified Multiplicities}
\label{subsection:exactmult}

In certain situations one may wish to classify maps~$f$ and
points~$P_i$ whose multiplicities (ramification indices) are specified
exactly, rather than by inequalities.  See for example
Section~\ref{section:goodreduction} and~\cite{arxiv1703.00823} for
applications with~$N=1$.  We start with a definition.

\begin{definition}
We put a partial order on the set of weighted portraits
by saying that $\Pcal'\ge\Pcal$ if there is a portrait morphism
$\a:\Pcal\to\Pcal'$ that is a bijection on vertices.
Referring to Definition~\ref{definition:portraitmorphism}
and writing $\Pcal=(\Vcal^\circ,\Vcal,\Phi,\e)$ and
$\Pcal'=({\Vcal'}^\circ,\Vcal',\Phi',\e')$,
this means that~$\a$ is a bijection
$\a:\Vcal\xrightarrow{\;\sim\;}\Vcal'$ satisfying
\[
  \a(\Vcal^\circ)={\Vcal'}^\circ,\quad
  \a\circ\Phi=\Phi'\circ \a,\quad\text{and}\quad
  \e'\circ \a\ge \e.
\]
The intuition is that $\Pcal'\ge\Pcal$ if~$\Pcal$ and~$\Pcal'$ have
the same vertices and the same maps, but the weights assigned to the
vertices in~$\Pcal'$ are allowed to be greater than the weights
assigned in~$\Pcal$.  
\end{definition} 

\begin{proposition}
\label{proposition:Pcalprime}
Let $N\ge1$ and $d\ge2$, and let $\Pcal$ be a portrait.
\begin{parts}  
\Part{(a)}
Let $\Pcal'\ge\Pcal$. Then the map $\a:\Pcal\to\Pcal'$ induces an
isomorphism of $\End_d^N[\Pcal']$ with a closed subscheme
of~$\End_d^N[\Pcal]$, and similarly allows us to
view~$\Moduli_d^N[\Pcal']$ as a closed subscheme
of~$\Moduli_d^N[\Pcal]$.
\Part{(b)}
For a given~$(N,d)$, 
there are only finitely many portraits~$\Pcal'$ satisfying $\Pcal'\ge\Pcal$
and $\End_d^N[\Pcal']\ne\emptyset$.
\end{parts}  
\end{proposition}
\begin{proof}
(a)\enspace  
We identify the vertices and maps of~$\Pcal$ and~$\Pcal'$, say
$\Pcal=(\Vcal^\circ,\Vcal,\F,\e)$ and
$\Pcal'=(\Vcal^\circ,\Vcal,\F,\e')$. We also let
$\Pcal''=(\Vcal^\circ,\Vcal,\F,1)$ be the associated unweighted portrait,
i.e., it has the same vertices and map as~$\Pcal$ and~$\Pcal'$, but every vertex
in~$\Vcal^\circ$ has weight~$1$. Then
Theorem~\ref{theorem:EnddNPweighted} tells us that~$\End_d^N[\Pcal]$
and~$\End_d^N[\Pcal']$ are closed subschemes of~$\End_d^N[\Pcal'']$
whose geometric points are specified by the multiplicity
condition~\eqref{eqn:EnddNPefPgeev} appearing in
Theorem~\ref{theorem:EnddNPweighted}. Since~$\e'\ge\e$ by assumption,
we see from~\eqref{eqn:EnddNPefPgeev} that
$\End_d^N[\Pcal']\subseteq\End_d^N[\Pcal]$,
and then taking~$\SL_{N+1}$-quotients as in Theorem~\ref{theorem:MNdwithportrait} gives the
analogous inclusion $\Moduli_d^N[\Pcal']\subseteq\Moduli_d^N[\Pcal]$.
\par\noindent(b)\enspace
For any~$Q\in\PP^N$, the multiplicities of the points in the inverse
image of~$Q$ satisfy
\[
  \sum_{P\in f^{-1}(Q)} e_f(P) = (\text{topological degree of $f$}) = \deg(f)^N,
\]
so taking $Q=f(P)$ yields $e_f(P)\le\deg(f)^N$. Hence if
$\End_d^N[\Pcal']\ne\emptyset$, then the weight of every vertex~$v'$
in~$\Pcal'$ must satisfy $\e'(v')\le d^N$, and
since~$\Vcal'=\a(\Vcal)$, there are a fixed number of vertices
in~$\Vcal'$ to which weights may be assigned. Hence there are only
finitely many possibilities.
\end{proof}

\begin{definition}
For $N\ge1$ and $d\ge2$, we define
\begin{align*}
  \End_d^N[\Pcal]^\circ &= \End_d^N[\Pcal] \setminus \bigcup_{\Pcal'\ge\Pcal,\,\Pcal'\ne\Pcal} \End_d^N[\Pcal'], \\
  \Moduli_d^N[\Pcal]^\circ &= \Moduli_d^N[\Pcal] \setminus \bigcup_{\Pcal'\ge\Pcal,\,\Pcal'\ne\Pcal} \Moduli_d^N[\Pcal'].
\end{align*}
It follows from Proposition~\ref{proposition:Pcalprime} that
$\End_d^N[\Pcal]^\circ$ is an open subscheme of $\End_d^N[\Pcal]$ and
that $\Moduli_d^N[\Pcal]^\circ$ is an open subscheme of
$\Moduli_d^N[\Pcal]$.
\end{definition}

\section{Refined Stability Criteria for $\PP^1$}
\label{section:refinedP1}

Corollary~~\ref{corollary:dgenm1z} gives a sufficient numerical
criterion for $\End_d^N\times(\PP^N)^n$ to be $\SL_{N+1}$-stable for
the line bundle~$\Ocal(\bfm):=\Ocal(m_0,\ldots,m_n)$. We used this in
Theorem~\ref{theorem:EndENDpt} to prove that~$\END_d^N[\Pcal]$
is~$\SL_{N+1}$-stable for the pull-back of~$\Ocal(\bfm)$, where we
recall that~$\END_d^N[\Pcal]$ is the Zariski closure
of~$\End_d^N[\Pcal]$ in~$\End_d^N\times(\PP^N)^n$. In particular, the
scheme~$\END_d^n[\Pcal]$ may include geometric points~$(f,\bfP)$ in
which~$\bfP$ lies on the big diagonal, i.e., the~$n$ points
in~$\bfP=(P_1,\ldots,P_n)$ need not be distinct.

In this section we consider maps on~$\PP^1$, i.e., we take~$N=1$, and
we restrict to points~$\bfP$ lying on the complement of the big
diagonal. In this case, we give a refined description of
the~$\SL_2$-stable locus of~$\End_d^1[\Pcal]$ relative
to~$\Ocal(m,1,\ldots,1)$ with~$m\in\{1,2\}$, and we construct the
associated moduli spaces.

\begin{proposition}
\label{proposition:stableEnd1P1nz}
Let $d\ge2$, let $n\ge0$, and let~$(\PP^1)^n_\D$ be the complement of
the big diagonal as in Section~$\ref{subsection:unwtEnddNP}$.
\begin{parts}
  \Part{(a)}
  $\End_d^1\times(\PP^1)^n_\D$ is $\SL_2$-semistable relative
  to $\Ocal(1,1,\ldots,1)$.
  \Part{(b)}
  $\End_d^1\times(\PP^1)^n_\D$ is $\SL_2$-stable relative
  to $\Ocal(1,1,\ldots,1)$ except for the case \text{$(d,n)=(2,1)$}.
  \Part{(c)} 
  $\End_2^1\times\PP^1$ is $\SL_2$-stable relative to $\Ocal(2,1)$.
  For the sheaf $\Ocal(1,1)$, we have
  \begin{multline*}
  \bigl\{ (f,P)\in \End_2^1\times\PP^1 : \text{$(f,P)$ is $\SL_2$-stable relative to $\Ocal(1,1)$} \bigr\} \\
  =
  \bigl\{ (f,P)\in \End_2^1\times\PP^1 : f(P)\ne P \bigr\}.
  \end{multline*}
\end{parts}
\end{proposition}

\begin{proof}
For $n=0$, we already know that~$\End_d^1$ is stable, so we assume
henceforth that $n\ge1$.  We are going to apply
Theorem~\ref{theorem:stableEnddNtimesPNnz} with~$N=1$
and~$\bfm=(1,1,\ldots,1)$.  In particular,~$m_0=1$ and~$m_\Sigma=n$.
The only proper linear subspaces $L\subsetneq\PP^1$ consist of single
points, say $L_Q=\{Q\}$. We are only considering pairs~$(f,\bfP)$
where $\bfP=(P_1,\ldots,P_n)$ consists of~$n$ \emph{distinct} points,
so for a given~$L_Q$, there is at most one~$\nu$ with~$P_\nu\in L_Q$.
Thus
\[
  C_\bfP(L_Q) 
  = \begin{cases}
  1&\text{if $Q=P_\nu$ for some $\nu$,} \\
  0&\text{otherwise.} \\
  \end{cases}
\]
(We are using the fact that every~$m_\nu=1$.)
We also note that the definition~\eqref{eqn:CPLDLdef2z} of~$D_\e(L)$
for these parameters gives
\[
  D_0(L_Q) = \frac{n+d-1}{2}\ge 1,
\]
where the lower bound comes from our assumption that $d\ge2$ and $n\ge1$.
Hence the estimate
\begin{equation}
  \label{eqn:PnuinLQleD0z}
  C_\bfP(L_Q) \le D_0(L_Q)
\end{equation}
is always true, so Theorem~\ref{theorem:stableEnddNtimesPNnz}(a) tells
us that~$(f,\bfP)$ is semistable relative to~$\Ocal(1,1,\ldots,1)$,
which proves~(a). Further, if~$(d,n)\ne(2,1)$, then
$D_0(L_Q)\ge3/2$, so~\eqref{eqn:PnuinLQleD0z} is a strict
inequality, and Theorem~\ref{theorem:stableEnddNtimesPNnz}(b) tells us
that~$(f,\bfP)$ is stable relative to~$\Ocal(1,1,\ldots,1)$. This
proves~(b).

It remains to consider stability in the case that~$(d,n)=(2,1)$ for
the invertible sheaf~$\Ocal(1,1)$.  Thus~$(f,P)$ consists of a map and
a single point,
and we need to compute
$\mu^{\Ocal(1,1)}\bigl((f,P),\ell)$ for non-trivial 1-parameter
subgroups~$\ell$. After a change of coordinates, we may assume that
\begin{equation}
  \label{eqn:ellalpha}
  \ell(\a) := \begin{pmatrix} \a^{-1} & 0 \\ 0 & \a \\ \end{pmatrix}.
\end{equation}
We write $P= [U(0),U(1)]$ and
\[
  f = \bigl[a_0(2)X^2+a_0(1)XY+a_0(0)Y^2,a_1(2)X^2+a_1(1)XY+a_1(0)Y^2\bigr],
\]
which is  an abbreviated version of the notation used in the proof of
Theorem~\ref{theorem:stableEnddNtimesPNnz}.

A basis of $\ell$-diagonalized global sections of~$\Ocal(1,1)$
on~$\End_2^1\times\PP^1$ is given by
\[
\bigl\{ a_\rho(e)U(i) : 0\le\rho\le1,\; 0\le e\le2,\; 0\le i\le 1\bigr\}
\]
and the action of~$\ell$ on these sections is via
\[
  \ell\star a_\rho(e)U(i)
  = \a^{2(2-e-\rho-i)} a_\rho(e)U(i).
\]
Hence
\[
  \mu^{\Ocal(1,1)}\bigl((f,P),\ell\bigr)
  = \max_{a_\rho(e)U(i)\ne0} 2(e+\rho+i-2).
\]
Suppose first that $a_1(2)\ne0$. Then using the
sections $a_1(2)U(i)$ gives
\[
\mu^{\Ocal(1,1)}\bigl((f,P),\ell\bigr)
\ge \begin{cases}
  2(2+1+0-2) = 2 &\text{if $U(0)\ne0$,} \\
  2(2+1+1-2) = 4 &\text{if $U(1)\ne0$.} \\
\end{cases}
\]
But~$U(0)$ and~$U(1)$ cannot both be~$0$, so
$\mu^{\Ocal(1,1)}\bigl((f,P),\ell\bigr)>0$, and the numerical
criterion (Theorem~\ref{theorem:numcrit}) tells us that~$(f,P)$ is
stable.

We are reduced to the case that~$a_1(2)=0$. This implies
that $a_0(2)\ne0$, since otherwise~$f$ would not be a morphism
at~$[0,1]$. Looking at the sections~$a_0(2)U(i)$ yields
\[
\mu^{\Ocal(1,1)}\bigl((f,P),\ell\bigr)
\ge \begin{cases}
  2(2+0+0-2) = 0 &\text{if $U(0)\ne0$,} \\
  2(2+0+1-2) = 2 &\text{if $U(1)\ne0$.} \\
\end{cases}
\]
Hence if~$U(1)\ne0$, we again conclude that~$(f,P)$ is stable.  And we
note that since~$a_1(2)=0$, the condition~$U(1)=0$ 
implies that~$f(P)=P$, so this shows that if~$f(P)\ne{P}$, then~$(f,P)$ is
stable.

Finally, suppose that $f(P) = P$. After a change of variables we may
assume $P = [1,0]$, hence we also have~$a_1(2) = 0$.  We now take 
the $1$-parameter subgroup $\ell$ given in \eqref{eqn:ellalpha} and
compute~$\mu^{\Ocal(1,1)}\bigl((f,P),\ell\bigr)$.  There are five
$\ell$-diagonalized sections that are potentially non-zero, as listed
in the following table with their corresponding values of~$e+\rho+i-2$:
{\small
\[
\begin{array}{|c||c|c|c|c|c|} \hline 
  \text{Section}   &U(0)a_0(2) & U(0)a_0(1) & U(0)a_0(0) & U(0)a_1(1) & U(0)a_1(0) \\ \hline
  e + \rho + i - 2 & 0 & -1 & -2 & 0 & -1 \\ \hline
\end{array}
\]}
Hence
\[
  \mu^{\Ocal(1,1)}\bigl((f,P),\ell\bigr) \le \max\{0,-1,-2,0,-1\} = 0,
\]
so~$(f,P)$ is not stable. This completes the proof that if the pair
$(f,P)\in\End_2^1\times\PP^1$ satisfies~$f(P)=P$, then~$(f,P)$ is not
stable relative to~$\Ocal(1,1)$.
\end{proof}

\begin{corollary}
\label{corollary:Md1PGIT}
Let~$\Pcal=(\Vcal^\circ,\Vcal,\F,\e)$ be a portrait, let $d\ge2$, and let
$\Lcal=\Ocal(1,1,\ldots,1)$ unless $d=2$ and
$\#(\Vcal\setminus\Vcal^\circ)=1$, in which case
let~$\Lcal=\Ocal(2,1)$. Then~$\End_d^1[\Pcal]$ is~$\SL_2$-stable
relative to~$\Pi_\Pcal^*\Lcal$, and hence there is a GIT moduli
quotient scheme
\[
  \Moduli_d^1[\Pcal]:=\End_d^1[\Pcal]\GITQuot\SL_2
  \quad\text{with structure sheaf}\quad  (\Pi_\Pcal^*\Lcal)^{\SL_2}.
\]
\end{corollary}
\begin{proof}
We follow the proof of Theorem~~\ref{theorem:EndENDpt}(c),
\emph{mutatis mutandis}, but using $\End_d^1$ instead of $\END_d^1$,
since 
Proposition~\ref{proposition:stableEnd1P1nz}  only gives stability of
$\End_d^1\times(\PP^1)^n_\D$, i.e., on the complement of the big
diagonal. Thus
\begin{align*}
  \End_d^1[\Pcal]
  &= \Pi_\Pcal^{-1}(\End_d^1\times(\PP^1)^\Zcal_\D) \\*
  &\subseteq \Pi_\Pcal^{-1}\bigl(\text{$\SL_{N+1}$-stable points of $\End_d^1\times(\PP^1)^\Zcal$ for $\Lcal$} \bigr)\\*
  &\omit\hfill\text{from Proposition \ref{proposition:stableEnd1P1nz},}\\*
  &\subseteq \bigl\{\text{$\SL_{N+1}$-stable points of $\End_d^1[\Pcal]$ for $\Pi_\Pcal^*\Lcal$}\bigr\} \quad\text{from \eqref{eqn:finvstable}}.
\end{align*}
This completes the proof of Corollary~\ref{corollary:Md1PGIT}.
\end{proof}

\section{The Moduli Space of $\PP^1$-Endomorphisms with Marked Critical Points}
\label{section:P1critmarkedc}

In characteristic~$0$, a degree~$d$ map $f:\PP^1\to\PP^1$ has~$2d-2$
geometric critical points, counted with appropriate multiplicity. It
is often useful to have a moduli space that classifies maps with
marked critical points. Our general construction allows us to describe
these spaces as GIT quotients.

Thus suppose that~$\Crit(f)=\{P_1,\ldots,P_r\}$ consists of~$r$
distinct points, and let $e_i:=e_f(P_i)\ge2$ denote the multiplicity
of~$P_i$. Then
\[
  e_1,\ldots,e_r\ge2\quad\text{and}\quad \sum_{i=1}^r (e_i-1) = 2d-2.
\]

\begin{definition}
Let~$\Pcal=(\Vcal^\circ,\Vcal,\F,\e)$ be a portrait. We say
that~$\Pcal$ is a \emph{complete critical point portrait of
  degree~$d$} if
\[
  \sum_{v\in\Vcal^\circ} \bigl(\e(v)-1\bigr) = 2d-2
  \quad\text{and}\quad
  \Vcal = \bigcup_{v\in\Crit(\Pcal)} \Orbit_\F(v).
\]
\textup(See Definition~\ref{definition:critorbitportrait} for the definition
of~$\Crit(\Pcal)$ and $\Orbit_\F(v)$.\textup)
We say that~$\Pcal$ is \emph{critically primitive} if it has the further property
\[
  \e(v)\ge2~\text{for all $v\in\Vcal^\circ$.}
\]
We write
\begin{align*}
  \mathsf{P}^\crit_d &= \{\Pcal : \text{$\Pcal$ is a complete critical
    point portrait of degree~$d$} \}, \\
  \mathsf{P}^\prim_d &= \{\Pcal\in \mathsf{P}^\crit_d : \text{$\Pcal$ is critically  primitive} \}.
\end{align*}
\end{definition}

\begin{remark}
See~\cite{MR1149891,MR2496235} for a description of the set of
preperiodic portraits in~$\mathsf{P}^\crit_d$ that are represented by
a polynomial map~$f(x)\in\Kbar[x]$, or equivalently, preperiodic
portraits~$\Pcal\in\mathsf{P}^\crit_d$ such that~$\Pcal$ contains a
fixed point of multiplicity~$d$ and $\End_d^1[\Pcal]\ne\emptyset$.
\end{remark}

\begin{lemma}
\label{lemma:primincrit}
Let $d\ge2$.
\begin{parts}
  \Part{(a)}
  Let $\Pcal=(\Vcal^\circ,\Vcal,\F,\e)\in\mathsf{P}^\crit_d$. Then
  every $v\in\Vcal$ has either an in-arrow or an out-arrow, i.e.,
  either $v\in\Vcal^\circ$ or there is some $v'\in\Vcal^\circ$ with
  $\F(v')=v$.
  \Part{(b)}
  $\mathsf{P}^\prim_d$ contains only finitely many isomorphism classes
  of portraits.
  \Part{(c)}
  Every portrait in $\mathsf{P}^\crit_d$ contains a unique subportrait
  in $\mathsf{P}^\prim_d$, which we called the \emph{frame of~$\Pcal$}
  and denote by~$\Frame(\Pcal)$.
\end{parts}
\end{lemma}
\begin{proof}
(a)\enspace
If~$v$ is a critical point, i.e, $\e(v)\ge2$, then
necessarily~$v\in\Vcal^\circ$, so~$v$ has an out-arrow.  And if~$v$ is
not a critical point, then it is in the orbit of a critical point, so
it has an in-arrow.
\par\noindent(b)\enspace
Let $\Pcal=(\Vcal^\circ,\Vcal,\Phi,\e)\in\mathsf{P}^\prim_d$ be critically primitive. Then
\[
  2d-2=\sum_{v\in\Vcal^\circ} \bigl(\e(v)-1\bigr) \ge \#\Vcal^\circ,
\]
since each~$\e(v)\ge2$ by assumption. This bounds~$\#\Vcal^\circ$, and hence
also bounds
\[
  \#\Vcal=\#\bigl(\Vcal^\circ\cup\F(\Vcal^\circ)\bigr)\le2\#\Vcal^\circ.
\]
Further,
\[
  \sum_{v\in\Vcal^\circ} \e(v) = \sum_{v\in\Vcal^\circ}
  \bigl(\e(v)-1\bigr) + \#\Vcal^\circ = 2d-2+\#\Vcal^\circ
\]
shows that the multiplicities~$\e(v)$ are also bounded.
\par\noindent(c)\enspace
Let $\Pcal=(\Vcal^\circ,\Vcal,\Phi,\e_\Vcal)\in\mathsf{P}^\crit_d$. We
define a subportrait
\[
\Pcal'=(\Wcal^\circ,\Wcal,\Phi_\Wcal,\e_\Wcal)
\]
of~$\Pcal$ by taking
\begin{align*}
  \Wcal^\circ &:= \bigl\{ v\in\Vcal^\circ : \e(v)\ge2 \bigr\},  \\
  \Wcal &:= \Wcal^\circ \cup \bigl\{\F(w) : w\in\Wcal^\circ\bigr\}, \\
  \Phi_\Wcal &:= \Phi|_{\Wcal^\circ},\qquad
  \e_\Wcal := \e_\Vcal|_{\Wcal^\circ}.
\end{align*}
We have
\[
  \sum_{w\in\Wcal^\circ} \bigl(\e_\Wcal(w)-1\bigr)
  = \sum_{\substack{v\in\Vcal^\circ\\ \e(w)\ge2\\}} \bigl(\e_\Vcal(v)-1\bigr)
  = \sum_{v\in\Vcal^\circ} \bigl(\e_\Vcal(v)-1\bigr)
  = 2d-2,
\]
so~$\Pcal'$ is a complete critical point portrait, and the definitions
of~$\Wcal$ and~$\Wcal^\circ$ show that~$\Pcal'$ is critically primitive.

Next suppose that
$\Pcal''=(\Ucal^\circ,\Ucal,\e_\Ucal)\in\mathsf{P}^\prim_d$ is a
subportrait of~$\Pcal$.  The assumption that~$(\Ucal^\circ,\Ucal,\e_\Ucal)$ is
critically primitive tells us that $\e_\Ucal(u)\ge2$ for all
$u\in\Ucal^\circ$, and hence $\Ucal^\circ\subseteq\Wcal^\circ$ from
the definition of~$\Wcal^\circ$. This allows us to compute
\begin{align*}
  \sum_{u\in\Ucal^\circ} & \overbrace{\bigl(\e_\Vcal(u)-\e_\Ucal(u)\bigr)}^{\hidewidth\text{${}\ge0$ since $\Pcal''\subseteq\Pcal$}\hidewidth}
  + \sum_{w\in\Wcal^\circ\setminus\Ucal^\circ} \overbrace{\bigl(\e_\Wcal(w)-1\bigr)}^{\hidewidth\text{${}\ge1$ since  $\e_\Wcal(w)\ge2$ for $w\in\Wcal^\circ$ }\hidewidth} \\
  &= \sum_{u\in\Ucal^\circ}  \bigl(\e_\Wcal(u)-\e_\Ucal(u)\bigr)
  + \sum_{w\in\Wcal^\circ\setminus\Ucal^\circ} \bigl(\e_\Wcal(w)-1\bigr)
  \quad\text{since $\e_\Wcal=\e_\Vcal$,} \\
  &= \sum_{u\in\Ucal^\circ}  \bigl(\e_\Wcal(u)-1\bigr)
  - \sum_{u\in\Ucal^\circ}  \bigl(\e_\Ucal(u)-1\bigr)
  + \sum_{w\in\Wcal^\circ\setminus\Ucal^\circ} \bigl(\e_\Wcal(w)-1\bigr) \\
  &= \sum_{u\in\Wcal^\circ}  \bigl(\e_\Wcal(u)-1\bigr)
  - \sum_{u\in\Ucal^\circ}  \bigl(\e_\Ucal(u)-1\bigr) \\
  &= (2d-2) - (2d-2)
  \quad\text{since $\Wcal$ and $\Ucal$ are complete,} \\
  &= 0.
\end{align*}
Hence
\[
  \Ucal^\circ=\Wcal^\circ
  \quad\text{and}\quad
  \e_\Ucal(u) = \e_\Vcal(u) = \e_\Wcal(u)
  \quad\text{for all $u\in\Ucal^\circ$.}
\]
Further, we see that
\[
  \Ucal = \Ucal^\circ\cup\F(\Ucal^\circ)
  = \Wcal^\circ\cup\F(\Wcal^\circ) = \Wcal,
\]
where for the first and third equalities we use~(a).  This completes
the proof that~$\Pcal''=\Pcal'$, and hence that~$\Pcal'$ is the unique
primitive complete critical point subportrait of~$\Pcal$.
\end{proof}

\begin{definition}
Let $\Pcal\in\mathsf{P}^\crit_d$, and
let~$\Frame(\Pcal)\in\mathsf{P}^\prim_d$ be the frame of~$\Pcal$ as
described in Lemma~\ref{lemma:primincrit}. A \emph{critical decomposition
  series for~$\Pcal$} is a sequence of portraits
\begin{equation}
  \label{eqn:decompser}
  \Pcal_0 \subsetneq \Pcal_1 \subsetneq \cdots \subsetneq \Pcal_n,
\end{equation}
with the portraits
$\Pcal_i=(\Vcal_i^\circ,\Vcal_i,\F_i,\e_i)\in\mathsf{P}^\crit_d$
satisfying
\[
  \Pcal_0=\Frame(\Pcal),\quad \Pcal_n=\Pcal,\quad\text{and}\quad
  \#(\Vcal_{i+1}\setminus\Vcal_i)=0~\text{or}~1.
\]
\end{definition}

\begin{remark}
The intuition for adjacent portraits in a critical decomposition
series~\eqref{eqn:decompser} is that $\Pcal_{i+1}$ has exactly one
more out-arrow than~$\Pcal_i$.  This can happen in one of three ways:
\begin{itemize}
  \setlength{\itemsep}{0pt}
\item
  Adjoin an existing point $v\in\Vcal_i\setminus\Vcal_i^\circ$ to
  $\Vcal_{i+1}^\circ$ and have $\F_{i+1}$ map~$v$ to a point in $\Vcal_{i+1} = \Vcal_i$.
\item
  Create a new point~$v'$ in $\Vcal_{i+1}^\circ$ and have $\F_{i+1}$ map~$v'$ to
  a point in $\Vcal_i$.
\item
  Create a new point~$v'$ in $\Vcal_{i+1}$, adjoin an existing point
  $v\in\Vcal_i\setminus\Vcal_i^\circ$ to $\Vcal_{i+1}^\circ$, and have
  $\F_{i+1}$ map~$v$ to~$v'$.
\end{itemize}
\end{remark}

Theorem~\ref{theorem:MNdwithportrait}(c) tells us that the chain of
inclusions in a critical decomposition series~\eqref{eqn:decompser}
gives a chain of inclusions of portrait moduli spaces in the opposite
direction,
\[
  \Moduli_d^1[\Pcal_n] \subseteq \cdots\subseteq \Moduli_d^1[\Pcal_1]
  \subseteq \Moduli_d^1[\Pcal_0],
\]
i.e., inclusions that start at $\Moduli_d^1[\Pcal]$ and end at
$\Moduli_d^1\bigl[\Frame(\Pcal)\bigr]$.  We note that
$\Moduli_d^1\bigl[\Frame(\Pcal)\bigr]$ is the largest portrait moduli
space using a portrait from~$\mathsf{P}^\crit_d$ in which
$\Moduli_d^1[\Pcal]$ naturally sits. Each inclusion
$\Moduli_d^1[\Pcal_{i+1}]\subseteq\Moduli_d^1[\Pcal_i]$ imposes at
most one additional critical point relation on the portrait.

\section{Examples of $\Moduli_d^1[\Pcal]$}
\label{section:Md1Pexamples}

In this section we illustrate our general theory by describing various
instances of~$\Moduli_d^1[\Pcal]$ and~$\Moduli_d^1[\Pcal|\Acal]$.

\begin{example}[Dynatomic Curves]
\label{example:dynatomiccurve}
One of the most studied families of dynamical systems is the set of
``unicritical'' polynomial maps $f_{d,c}(x):=x^d+c$, especially
for~$d=2$.  (We have put quote marks around ``unicritical'' because as
a map of~$\PP^1$, the polynomial~$f_{d,c}(x)$  has two
critical points,~$0$ and~$\infty$.) We briefly explain how these
curves fit into our general setup.

For each $d\ge2$ and $n\ge1$, there are associated \emph{dynatomic
  curves} $Y_{1,d}^\dyn(n)$ and $Y_{0,d}^\dyn(n)$. The former is
frequently defined to be the affine curve consisting of the
points~$(c,x)$ in~$\AA^2$ satisfying the equation
\begin{equation}
  \label{eqn:Y1ddynn}
 \prod_{m\mid n} \bigl( f_{d,c}^{n}(x) - x \bigr)^{\mu(n/m)} = 0;
\end{equation}
see for example~\cite{arxiv1711.04233}.  However, for~$d\ge3$,
we note that this is not quite a moduli space, since for every
$(d-1)$st root of unity~$\z$, we can conjugate~$f_{d,c}(x)$ by
$x\to\z^{-1}x$, and thus we should identify~$(c,x)$
and~$(\z{c},\z{x})$.  So we define~$Y_{1,d}^\dyn(n)$ to be the
quotient of the curve~\eqref{eqn:Y1ddynn} by the action of~$(d-1)$st
roots of unity.  Of course, for the much studied case of~$d=2$, there
is no distinction.

The curve~$Y_{0,d}^\dyn(n)$ classifies degree~$d$
unicritical polynomials with a marked cycle of formal period~$n$. It
is defined to be the quotient
\[
  Y_{0,d}^\dyn(n) := Y_{1,d}^\dyn(n)\GITQuot(\ZZ/n\ZZ),
\]
where the action of~$\ZZ/n\ZZ$ is given by
\[
  \ZZ/n\ZZ\longhookrightarrow\Aut\bigl(Y_{1,d}^\dyn(n)\bigr),\quad
  k \longmapsto \Bigl( (c,x) \mapsto \bigl(c,f_{d,c}^k(x)\bigr) \Bigr).
\]

To describe the dynatomic curve~$Y_{1,d}^\dyn(n)$ in our general
setting, we consider the following
portrait~$\Pcal=(\Vcal^\circ,\Vcal,\F,\e)$:
\begin{align*}
  \Vcal^\circ &= \{\a,\b_1,\g_0,\ldots,\g_{n-1}\},\hidewidth
  &&&\Vcal &= \Vcal^\circ\cup\{\b_2\},\\
  \F(\a)&=\a,& \F(\b_1)&=\b_2,& \F(\g_i)&=\g_{i+1\bmod n},\\
  \e(\a)&=d,& \e(\b_1)&=d,& \e(\g_0)&=\cdots=\e(\g_{n-1})=1.      
\end{align*}    
Let $(f,P)\in\End_d^1[\Pcal]$ be a geometric point of
characteristic~$0$.  Then~$f$ has a maximally ramified critical fixed
point, hence is a polynomial, a second 
maximally ramified critical point, hence is unicritical, and a marked
periodic cycle of exact period~$n$. Changing coordinates, i.e.,
conjugating by an element of~$\PGL_2$, we may assume that~$(f,P)$ has
the form~$(f_{d,c},P)$, where~$c$ is well-defined up to a $(d-1)$st
root of unity.  In this way we obtain a well-defined morphism
$\Moduli_d^1[\Pcal]\hookrightarrow Y_{1,d}^\dyn(n)$, but the map is
not surjective. The failure of surjectivity is due to the fact that we have required 
the points in our model to be distinct: First, the
space~$\Moduli_d^1[\Pcal]$ is missing the points of formal
period~$n$. Second, 
it is missing the points
where the second critical point is either in the $n$-cycle or maps into the $n$-cycle
after one step. Finally, $\Moduli_d^1[\Pcal]$ is missing the points where the second
critical point is fixed. (Note that in the second and third cases $f$ is PCF.)
These difficulties are overcome by using
instead the space~$\MODULI_d^1[\Pcal]$, which is finite
over~$\Moduli_d^1$. In this way, we obtain an isomorphism
$\MODULI_d^1[\Pcal]\xrightarrow{\sim}{Y_{1,d}^\dyn(n)}$. Similarly, if
we let~$k\in\ZZ/n\ZZ$ act on~$\Pcal$ via $\g_i\mapsto\g_{i+k\bmod n}$,
then we obtain an isomorphism
$\MODULI_d^1[\Pcal|\ZZ/n\ZZ]\xrightarrow{\sim}{Y_{0,d}^\dyn(n)}$;
cf.\ the notation set in Definition~\ref{definition:MdNP10}.
\end{example}

\begin{example}
\label{example:milnorM21eqA2}
The space~$\Moduli_2^1$ has dimension~$2$, and indeed is isomorphic
to~$\AA^2$ via the Milnor isomorphism
\[
 \bfs=(s_1,s_2) : \Moduli_2^1 \xrightarrow{\;\sim\;} \AA^2,
\]
where $s_1$ and $s_2$ are elementary symmetric functions in the
multipliers of the fixed points for the equivalence class of maps
$\langle f\rangle \in \Moduli_2^1$;
see~\cite[Theorem~4.56]{MR2316407}, for example.
The next result, which is due to
Blanc, Canci, and Elkies, describes the finite coverings
of~$\Moduli_2^1$ associated to portraits consisting of a single cycle
of small size.

\begin{theorem}
[\textup{Blanc, Canci, and Elkies~\cite{MR3431627}}]
\label{theorem:M21Cn}
Let $\Ccal_n$ be the unweighted portrait consisting of a single
$n$-cycle, i.e., $\Vcal^\circ=\Vcal=\ZZ/n\ZZ$ and $\F(v)=v+1\bmod n$.
Note that~$\Aut(\Ccal_n)$ is the group generated by~$\Phi$.\footnote{Note
  that~$\Aut(\Ccal_n)$ is not the dihedral group, because our $n$-gon
  is a directed graph, and automorphisms need to preserve the
  orientation.}
  \begin{parts}
  \Part{(a)}
  For $1\le n\le 5$, the variety $\Moduli_2^1[\Ccal_n|1]$ is a rational surface over~$\QQ$.
  \Part{(b)}
  For $1\le n\le 6$, the variety $\Moduli_2^1[\Ccal_n|0]$ is a rational surface over~$\QQ$.
  \Part{(c)} 
  The variety $\Moduli_2^1[\Ccal_6|1]$ is a surface of general
  type. Thus the Bombi\-eri--Lang conjecture implies that
  $\Moduli_2^1[\Ccal_6|1](\QQ)$ is not Zariski dense.
  \Part{(d)}
  The set $\Moduli_2^1[\Ccal_6|1](\QQ)$ is infinite. More precisely, the
  space $\Moduli_2^1[\Ccal_6|1]$ contains a rational curve and three
  elliptic curves, each of which has infinitely many~$\QQ$-rational
  points.
  \Part{(e)}
  Let $\Ccal_n'$ denote the union of~$\Ccal_n$ and a single fixed point.
  Then the set $\Moduli_2^1[\Ccal_6'|1](\QQ)$ is infinite.
\end{parts}
\end{theorem}
\end{example}

\begin{example}
Let~$\Pcal$ be the portrait consisting of three fixed points,
\[
  \Pcal:\quad
  \xymatrix{ {\bullet}  \ar@(dr,ur)[]_{}} \quad
  \xymatrix{ {\bullet}  \ar@(dr,ur)[]_{}} \quad
  \xymatrix{ {\bullet}  \ar@(dr,ur)[]_{}}.
\]
(This is the portrait labeled~$\Pcal_{3,6}$
in~\cite[Table~2]{arxiv1703.00823}.)
Every rational map of degree~$2$ on~$\PP^1$ has three fixed points,
counted with multiplicities, so the projection map
\begin{equation}
  \label{eqn:End21PtoEnd213fixed}
  \End_2^1[\Pcal] \to \End_2^1
\end{equation}
is generically 6-to-1, corresponding to the six ways to label the
three fixed points. However, since~$\End_2^1[\Pcal]$ is, by
definition, a subscheme of~$\End_2^1\times(\PP^1)^3_\D$, the three
specified points are required to be distinct. It follows that the
map~\eqref{eqn:End21PtoEnd213fixed} is \'etale, and the image is a
Zariski open subset of~$\End_2^1$, namely
\[
  \Image\bigl(\End_2^1[\Pcal]\bigr)
  = \bigl\{f\in\End_2^1 : \text{$f$ has  3 distinct fixed points}\bigr\}.
\]
Taking the (partial) closure~$\END_2^1[\Pcal]$ fills in the missing
maps. Note that there is no issue with formal versus exact period
here, since we are only dealing with fixed points.

All of this respects the natural action of~$\PGL_2$, and we note
that $\Aut(\Pcal)\cong\Scal_3$  permutes the three fixed points
in~$\Pcal$, so we obtain
\begin{equation}
  \label{eqn:M21P61M21}
  \begin{array}{ccccc}
    \Moduli_2^1[\Pcal] &\xrightarrow[\text{6-to-1}]{\text{\'etale}} &   \Moduli_2^1[\Pcal|\Scal_3] \\
    \bigcap && \bigcap \\
    \MODULI_2^1[\Pcal] &\xrightarrow[\phantom{\text{6-to-1}}]{} &   \Moduli_2^1 &\hspace{-1em}\cong & \AA^2.
  \end{array}
\end{equation}
The complement $\Moduli_2^1\setminus\Moduli_2^1[\Pcal|\Scal_3]$
corresponds to maps having~$1$ or~$2$ fixed points. Up to~$\PGL_2$-conjugation,
the only map with~$1$ fixed point is~$z+z^{-1}$, and the maps with~$2$ fixed
points comprise the family\footnote{Move the fixed points to~$0$ and~$\infty$, with~$\infty$ a double
fixed point.  The maps that fix~$0$ and~$\infty$ are the maps
$f(z)=(az^2+bz)/(cz+d)$, and~$\infty$ is a double fixed point if and
only if~$a=c$. So setting~$a=c$ and dehomogenizing~$a=1$, we have
$f(z)=(z^2+bz)/(z+d)$. Conjugating by~$\f(z)=dz$ yields $(1/d)f(dz) =
(z^2+bz/d)/(z+1)$, so relabeling yields the stated result.}
\[
  \bigl\{(z^2+bz)/(z+1) : b\ne1 \bigr\}.
\]
Using the Milnor isomorphism,
\[
  \bfs\left(z+z^{-1}\right) = (3,3)
  \quad\text{and}\quad
  \bfs\left(\frac{z^2+bz}{z+1}\right) = (b+2,2b+1),
\]
so we see that
\[
  \Moduli_2^1[\Pcal|\Scal_3] \cong
  \AA^2 \setminus \{2s_1=s_2+3\} .
\]

In general, given a map~$f$ with three distinct marked fixed points,
we can use~$\PGL_2$ to move the fixed points
to~$\{0,1,\infty\}$. Then~$f$ has the form $f(z)=(az^2+bz)/(cz+d)$
with $a+b=c+d$. Dehomogenizing~$a=1$, and then for convenience letting
$b=u-1$ and $d=v$, so $c=u-v$, this puts~$f$ into the normal form
\[
  f(z) = \frac{z^2+(u-1)z}{(u-v)z+v}
\]
with resultant $\Resultant(f)=uv(1-u+v)$ and with Milnor coordinates
\begin{align*}
  s_1(f) &= \frac{ (u-1)(v+1)(u-v) }{ uv } + 2, \\
  s_2(f) &= \frac{ (u-v)(u^2+v^2-u+v+1) - 1 }{ uv} + 1 \\
  &= \frac{ (u-1)(u^2+v^2-uv+v+1) - v^3 }{ uv } + 1 \\
  &= \frac{ -(v+1)(u^2+v^2-uv-u+1) + u^3} { uv } + 1.
\end{align*}
We observe that the excluded locus~$2s_1-s_2-3 = 0$ pulls back to
\[
  2s_1-s_2-3
  = \frac{(1-u+v)^3}{uv}.
\]
Hence
\[
  \Moduli_2^1[\Pcal] \cong \AA^2 \setminus \bigl\{ uv(1-u+v) = 0\bigr\},
\]
where~$uv(1-u+v)=0$ corresponds to maps of degree strictly less than~$2$.

To illustrate the $6$-to-$1$ nature of the map~\eqref{eqn:M21P61M21},
we observe that the Milnor map sends each of the lines $\{u=1\}$,
$\{v=-1\}$, and $\{u=v\}$ in $(u,v)$-space as a double cover onto the
line $s_1=2$, which is the locus of polynomial maps in~$\Moduli_2^1$.
Explicitly,
\[
  u=1\Rightarrow s_2=1-v^2,\quad
  v=-1\Rightarrow s_2=1-u^2,\quad
  u=v\Rightarrow s_2=1-u^{-2}.
\]
The lines~$\{u=1\}$ and~$\{v=1\}$ map $2$-to-$1$ onto the line
$s_1=2$. The line $\{u=v\}$ also maps $2$-to-$1$ to the line $s_1 = 2$,
but its image misses the point $(s_1,s_2) = (2,1)$.
\end{example}

\begin{example}
Let~$\Pcal$ be the portrait consisting of two fixed points and one
2-cycle,
\begin{equation}
  \label{eqn:P2fixe1twocycle}
  \Pcal:\quad
  \xymatrix{ {\bullet}  \ar@(dr,ur)[]_{}} \hspace{1.5em}
  \xymatrix{ {\bullet}  \ar@(dr,ur)[]_{}} \hspace{1.5em}
  \xymatrix{
     {\bullet} \ar@(dl,dr)[r]_{}   & {\bullet}   \ar@(ur,ul)[l]_{} \\
  } .
\end{equation}
(This is the portrait labeled~$\Pcal_{4,13}$
in~\cite[Table~3]{arxiv1703.00823}.)
A generic rational map of degree~$2$ on~$\PP^1$ has three fixed points
and one cycle of period~$2$. There are~$6$ ways to choose a labeled
pair of fixed points and~$2$ ways to label the~$2$-cycle, so the
projection map
\begin{equation}
  \label{eqn:End21PtoEnd212fixed1per2}
  \End_2^1[\Pcal] \to \End_2^1
\end{equation}
is generically $12$-to-$1$.

The automorphism group $\Aut(\Pcal)\cong(\ZZ/2\ZZ)^2$ is generated by
the involution~$\f_1$ that swaps the fixed points and the
involution~$\f_2$ that swaps the two points in the $2$-cycle.
This gives the commutative diagram of morphisms, where we have indicated
the degree of each map:
\[
  \begin{array}{cccccccc}
    && \Moduli_2^1[\Pcal|\f_1] \\
    &\nearrow_2 && \searrow^2 \\
    \Moduli_2^1[\Pcal] &&&& \Moduli_2^1[\Pcal|\Aut(\Pcal)] & \xrightarrow{\;3\;} \Moduli_2^1 \cong \AA^2 \\
    &\searrow^2 && \nearrow_ 2 \\
    && \Moduli_2^1[\Pcal|\f_2] \\
  \end{array}
\]

We can use~$\PGL_2$ to move the points in the $2$-cycle
to~$\{0,\infty\}$ and one of the fixed points to~$1$, so~$f$
has the form
\[
  f(z) = \frac{az+b}{cz^2+dz}\quad\text{with $a+b=c+d$.}
\]
The fact that~$\deg(f)=2$ tells us that~$c\ne0$, so we can dehomogenize~$c=1$
to obtain
\[
  f(z) = \frac{az+b}{z^2+(a+b-1)z}\quad\text{with}\quad
  \Resultant(f) = (1-a)b(a+b).
\]

The image of~$f$ via the Milnor map is given by the unenlightening
formulas
{\tiny
\begin{align*}
s_1(f) &= \frac{ \left(b + 1\right) a^3 + \left(3 b^2 + b\right) a^2 +
  \left(3 b^3 - 2 b^2 - b\right) a + \left(b^4 - 3 b^3 - 3 b^2 -
  b\right) }{-b a^2 + \left(-b^2 + b\right) a + b^2}, \\
s_2(f) &= \frac{-a^4 - 4 b a^3 + \left(-7 b^2 + b - 1\right) a^2 +
  \left(-6 b^3 + 5 b^2 + b\right) a + \left(-2 b^4 + 6 b^3 + 6 b^2 + 2
  b\right) }{-b a^2 + \left(-b^2 + b\right) a + b^2}.
\end{align*}
}

The fixed points of~$f$ are the roots of the equation
\[
  (z-1)(z^2+(a+b)z+b) = 0.
\]
In order to specify a second fixed point, we need to choose
a root of $z^2-(a+b)z+b$, which requires taking a square root. Formally,
there is a model for~$\Moduli_2^1[\Pcal]$ as an open subvariety of
\[
  V := \bigl\{ (a,b,t) \in \AA^3 : (a+b)^2-4b = t^2 \bigr\},
\]
and we have a dominant birational map
\[
\begin{array}{ccc}
  V & \longrightarrow & \Moduli_2^1[\Pcal], \\
  (a,b,t) & \longmapsto &
  \left( \dfrac{az+b}{z^2+(a+b-1)z}, 1, \dfrac{-a-b+t}{2}, 0, \infty \right),
  \end{array}
\]
where we have labeled the four points in the
picture~\eqref{eqn:P2fixe1twocycle} for~$\Pcal$ from left to right.
\end{example}

\begin{example}
\label{example:34critptsdeg3}
We consider maps of degree~$3$ on~$\PP^1$.
Let~$\Pcal$ be the portrait consisting of three fixed points,
each taken with multiplicity~$2$:
\[
  \Pcal:\quad
  \xymatrix{ {\bullet}  \ar@(dr,ur)[]_{2}} \quad
  \xymatrix{ {\bullet}  \ar@(dr,ur)[]_{2}} \quad
  \xymatrix{ {\bullet}  \ar@(dr,ur)[]_{2}} .
\]
Without loss of generality, we may move the fixed points to~$\{0,1,\infty\}$,
and then every map~$f\in\Image(\End_3^1[\Pcal]\longrightarrow\End_3^1)$
has the form
\begin{equation}
  \label{eqn:fax3bz23a2b}
  f(z) = \frac{az^3+bz^2}{(3a+2b)z-(2a+b)}
\end{equation}
after an appropriate change of variables.
The resultant of~$f$ is
\[
  \Resultant(f) = -2a^2(a+b)^2(2a+b)^2,
\]
so we find that
\[
  \Moduli_3^1[\Pcal] = \PP^1 \setminus\bigl\{ [0,1],[1,-1], [1,-2] \bigr\}.
\]
\end{example}

\begin{example}
\label{example:M31with4fixedpts}
A generic map of degree~$3$ has~$4$ fixed points and total ramification
index $\sum \bigl(e_f(P)-1\bigr)=2d-2=4$, so in principle we could require
that all four fixed points be critical points, i.e., we
could consider $\Moduli_3^1[\Pcal']$ for the portrait
\[
  \Pcal':\quad
  \xymatrix{ {\bullet}  \ar@(dr,ur)[]_{2}} \quad
  \xymatrix{ {\bullet}  \ar@(dr,ur)[]_{2}} \quad
  \xymatrix{ {\bullet}  \ar@(dr,ur)[]_{2}} \quad
  \xymatrix{ {\bullet}  \ar@(dr,ur)[]_{2}}.
\]
This imposes one additional constraint on the 1-parameter family of
maps~$f(z)$ described by~\eqref{eqn:fax3bz23a2b}, so we might expect
examples to exist. However, it turns out
that~$\Moduli_3^1[\Pcal']=\emptyset$, since one easily computes that
the fourth fixed point is $2+a^{-1}b$, and the multiplier at that
point is
\[
  f'\left(2+\frac{b}{a}\right) = \frac32.
\]
Hence the fourth fixed point is never critical.\footnote{More
  precisely, the fourth fixed point is not critical if
  $\characteristic(K)\ne3$, and it is always critical if
  $\characteristic(K)=3$.}  This example illustrates a general way to
construct obstructions to~$\Moduli_d^1[\Pcal]\ne\emptyset$ using the
fixed point formula described in~\cite[Theorem~1.14]{MR2316407}
or~\cite[Section~12]{MR2193309}; see also Section~\ref{section:multiplierrelatons:}.
\end{example}

\begin{example}
\label{example:M31gentype}
We investigate a more complicated  space $\Moduli_d^1[\Pcal]$ in which the
portrait~$\Pcal$ is the union of~$3$ fixed points and a $3$-cycle
containing~$2$ critical points of multiplicity~$2$.  We note that the
three points in the $3$-cycle of~$\Pcal$ are rigidly specified, since
one is a critical point that is not a critical value, one is a critical
value that is not a critical point, and one is both a critical point and a
critical value. So given any~$f\in\End_3^1[\Pcal]$, there is
exactly one change of variables taking the $3$-cycle to the
points~$0,1,\infty$ such that~$0$ and~$\infty$ are critical points
and~$\infty$ and~$1$ are critical values. Labeling the three marked
fixed points as~$\a,\b,\g$, we have the picture illustrated in
Figure~\ref{eqn:3cycle}.

\begin{figure}[ht]
\[
  \raisebox{-20pt}{$\displaystyle \Pcal:\quad
    \xymatrix{  {\alpha}  \ar@(dr,ur)[]_{} } \quad
    \xymatrix{  {\beta}  \ar@(dr,ur)[]_{} } \quad
    \xymatrix{  {\gamma}  \ar@(dr,ur)[]_{} }
    $ }\quad
  \xymatrix{ {\infty} \ar[r]^{2} & {1} \ar[dl]_{}\\
     {0}  \ar[u]^{2}\\
    }
\]
\caption{Portrait $\Pcal$ with 3 fixed points and a doubly-critical 3-cycle}
\label{eqn:3cycle}
\end{figure}
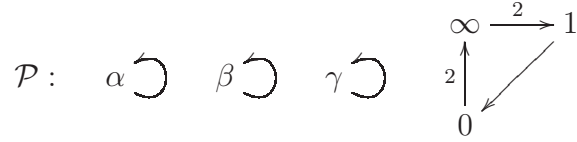

We first observe that every map~$f\in\End_3^1$ having the $3$-cycle
in Figure~\ref{eqn:3cycle} has the form
\[
  f_{u,v}(x) = \frac{x^3+u x^2-(v+1) x-(u-v)}{x^3+ux^2}
\]
with
\begin{equation}
  \label{eqn:resfuv}
  \Resultant(f_{u,v}) = -(u+1)(u-v)^2v \ne 0.
\end{equation}
Imposing the conditions that~$\a,\b,\g$ are fixed points of~$f_{u,v}$
gives three linear equations for the two quantities~$u$
and~$v$. Explicitly
\begin{equation}
  \label{eqn:abcmatrix1}
  \begin{pmatrix}
    \a^3-\a^2+ 1 & \a-1 & \a(\a^3-\a^2+ 1) \\
    \b^3-\b^2+ 1 & \b-1 & \b(\b^3-\b^2+ 1) \\
    \g^3-\g^2+ 1 & \g-1 & \g(\g^3-\g^2+ 1) \\
  \end{pmatrix}
  \begin{pmatrix}
    u\\ v\\ 1\\
  \end{pmatrix} =
  \begin{pmatrix}
    0\\ 0\\ 0\\
  \end{pmatrix}.
\end{equation}
We let~$A$ denote the matrix in~\eqref{eqn:abcmatrix1}.

This proves that $(f_{u,v},\a,\b,\g,0,1,\infty)\in\End_3^1[\Pcal]$ if
and only if the equation~\eqref{eqn:abcmatrix1} has a solution using
distinct values~$\a,\b,\g\notin\{0,1\}$. Clearly a necessary condition
for this to occur is to have~$\det(A)=0$, and further, if~$A$ has rank
exactly~$2$, then~$(\a,\b,\g)$ determines a unique pair~$(u,v)$.

We pause momentarily to discuss what happens if $\rank(A)\le1$. Since
we are not allowing~$\a,\b,\g$ to be~$1$ and are requiring them to be
distinct, a calculation using the $2$-by-$2$ minors of~$A$ shows that
$\rank(A)\le1$ if and only if~$\a,\b,\g$ are the three roots of the
polynomial~$t^3-t^2+1$, in which case~$A$ has rank exactly~$1$ and its
null space is the set of vectors of the form~$(u,0,w)$. But the
resultant~\eqref{eqn:resfuv} of~$f_{u,v}$ vanishes when~$v=0$, so
these~$\a,\b,\g$ values do not yield a point
in~$\End_3^1[\Pcal]$.\footnote{More precisely, we get
  $(1-x^{-2},\a,\b,\g,0,1\,\infty)\in\End_2^1[\Pcal]$.}



We conclude that~$\Moduli_3^1[\Pcal]$ is isomorphic to an open subset
of the set of $(\a,\b,\g)\in\AA^3$ such that $\det(A)=0$. More
precisely, since we want~$\a,\b,\g$ to be distinct, we take the
subvariety of $\AA^3$ defined by~$\frac{\det(A)}{(\a-\b)(\a-\g)(\b-\g)} = 0$.
This gives the following formula for a surface~$S$ in~$\AA^3$ that
contains a Zariski open subset isomorphic to~$\Moduli_3^1[\Pcal]$:
\begin{multline*}
  S : \a \b^2 \g^2 + \a^2 \b \g^2 + \a^2 \b^2 \g  
  - (\a^2 \b^2 + \a^2 \g^2 + \b^2 \g^2) \\  
  - 2 (\a^2 \b \g + \a \b^2 \g + \a \b \g^2)   
  + 3 \a \b \g + \a^2 \b + \a^2 \g + \b^2 \g + \a \b^2 + \a \g^2 + \b \g^2 \\ 
  - (\a^2 + \b^2 + \g^2) - 2 (\a \b + \a \g + \b \g)  
  + 2 (\a + \b + \g) -1 = 0.  
\end{multline*}


Next let $\Acal=\Aut(\Pcal)\cong\Scal_3$ be the set of automorphisms
that permute the three fixed points of~$\Pcal$. The
surface~$\Moduli_3^1[\Pcal|\Acal]$ is obtained by writing the equation
of~$\Moduli_3^1[\Pcal]$ in terms of the elementary symmetric functions
$X:=\a+\b+\g$, $Y:=\a\b+\a\g+\b\g$, and $Z:=\a\b\g$, which can be done
by hand, or using the command \texttt{IsSymmetric()} in
Magma~\cite{MR1484478}.  In any case, we find the equation
\[
  X^2 - X Y + Y^2 - Y Z - 2 X + 1 = 0.
\]
Thus~$\Moduli_3^1[\Pcal|\Acal]$ is a rational surface, since we
can solve for~$Z$ in terms of~$X$ and~$Y$. This exhibits~$\Moduli_3^1[\Pcal]$
as a ramified cover of~$\AA^2$.

\end{example}

\section{The Dimension of $\Moduli_d^N[\Pcal]$ for Unweighted Portraits}
\label{section:dimMdNP}

In this section we make the following assumption:
\begin{center}
  \framebox{\parbox{.78\hsize}{\noindent
      All varieties are defined over a field of
      characteristic $0$.}}
\end{center}

\begin{definition}
\label{definition:nudNn}  
For integers~$d \ge 2$,~$N \ge 1$, and~$n \ge 1$, we set
\begin{equation}
  \label{eqn:nudNn}
  \nu_d^N(n) := \sum_{k\mid n}\left( \mu(n/k) \sum_{j=0}^N d^{jk}\right).
\end{equation}
For~$m \ge 1$ we set
\[
  \nu_d^N(m,n) := d^{N(m-1)}(d^N - 1)\nu_d^N(n),
\]
and we extend this to $m = 0$ by setting~$\nu_d^N(0,n) := \nu_d^N(n)$.
\end{definition}

\begin{theorem}
\label{theorem:UnweightedNonempty} 
Let $\Pcal=(\Vcal^\circ,\Vcal,\Phi)$ be an unweighted portrait, and let
\begin{align*}
  D_\Pcal &:= \max\bigl\{\#\Phi^{-1}(P) : P \in \Vcal\bigr\}, \\
  C_\Pcal(n) &:= \#\{ v\in\Vcal^\circ : \text{$v$ has exact period $n$ for $\Phi$} \}.
\end{align*}
Then the following are equivalent\textup:
\begin{parts}
\Part{(a)}
$\End_d^N[\Pcal]\ne\emptyset$.
\Part{(b)}
$\Mcal_d^N[\Pcal]\ne\emptyset$.
\Part{(c)}
The finite morphism
\begin{equation}
  \label{eqn:PiPcalnonempty}
   \Pi_\Pcal :  \END_d^N[\Pcal] \longrightarrow \End_d^N \times (\PP^N)^{\Vcal\setminus\Vcal^\circ}
\end{equation}
described in Theorem~\textup{\ref{theorem:EndENDpt}(a)} is surjective.
\Part{(d)}
$\dim\End_d^N[\Pcal] = \dim\End_d^N + N\cdot\#(\Vcal\setminus\Vcal^\circ)$.
\Part{(e)}
$\dim\Moduli_d^N[\Pcal] = \dim\Moduli_d^N + N\cdot\#(\Vcal\setminus\Vcal^\circ)$.
\Part{(f)}  
$D_\Pcal \le d^N$ and $C_\Pcal(n) \le \nu_d^N(n)$ for all $n \ge 1$.
\end{parts}
\end{theorem}

The remainder of this section is devoted to proving
Theorem~\ref{theorem:UnweightedNonempty} and a corollary that
describes the dimensions of the fibers and images of maps between
portrait moduli spaces.

\begin{definition}
For $f \in \End_d^N$, we let $\PrePer(f)$ denote the set of
preperiodic points for~$f$. For each preperiodic point $P \in
\PrePer(f)$, we say~$P$ has \emph{preperiodic type}~$(m,n)$ if~$P$
enters an~$n$-cycle after precisely~$m$ steps. We will also refer to
$m$ and $n$ as the \emph{preperiod} and (\emph{eventual})
\emph{period}, respectively, of $P$.
\end{definition}
      
\begin{lemma}
\label{lemma:NumPrepPts}
Fix integers~$d \ge 2$ and~$N \ge 1$, and let
\[
  f : \PP^N_{\End_d^N}\longrightarrow \PP^N_{\End_d^N}
\]
be the generic endomorphism of degree~$d$. Then for all
integers~$m\ge0$ and~$n\ge1$, the map~$f$ has exactly~$\nu_d^N(m,n)$
geometric points of preperiodic type~$(m,n)$.
\end{lemma}

\begin{proof}
We first consider the case~$m = 0$. The number of points of period~$n$
for~$f$, counted with the appropriate multiplicity, is equal
to~$\nu_d^N(n)$; see \cite[Theorems~4.3~and~4.17]{arxiv:0801.3643}. By
\cite[page~3513]{arxiv:1011.5155}, for a general degree-$d$
endomorphism of~$\PP^N$, each point of exact period~$n$ occurs to
multiplicity~$1$, so there must be exactly~$\nu_d^N(n) = \nu_d^N(0,n)$
such points.

By \cite[Corollary~3.5]{MR3218803}, the generic map~$f$ has no
ramified preperiodic points, so every preperiodic point has
precisely~$d^N$ preimages under~$f$. Every point of type~$(1,n)$ is
the preimage of a point of period~$n$; conversely, given a point~$P$
of period~$n$, all but one of its preimages is a point of
type~$(1,n)$. (The remaining preimage is the periodic
point~$f^{n-1}(P)$.) Thus~$f$ has $(d^N - 1)\nu_d^N(n)=\nu_d^N(1,n)$
points of type~$(1,n)$.

Now assume that~$m \ge 2$. The points of type~$(m,n)$ are precisely
the preimages of the points of type~$(m-1,n)$. If we assume that there
are $\nu_d^N(m-1,n)$ points of type~$(m-1,n)$, and if we again use the
fact that~$f$ is unramified at its preperiodic points, it follows that
there are exactly
\[
  d^N \cdot \nu_d^N(m-1,n) = \nu_d^N(m,n)
\]
points of type~$(m,n)$. Thus the stated result follows by induction
on~$m$.
\end{proof}

\begin{proof}[Proof of Theorem~$\ref{theorem:UnweightedNonempty}$]
We prove various implications.
\par\noindent\framebox{(a) $\Longleftrightarrow$ (b)}\enspace
This is immediate from the surjectivity of the map
\[
\End_d^N[\Pcal]\to\Moduli_d^N[\Pcal].
\]
\par\noindent\framebox{(d) $\Longrightarrow$ (c)}\enspace
Theorem~\ref{theorem:EndENDpt} says that the
map~\eqref{eqn:PiPcalnonempty} is a finite morphism, and~(d) says
that~$\END_d^N[\Pcal]$ contains an open subset (namely, $\End_d^N[\Pcal]$) whose
dimension is equal to the dimension of the codomain
$\End_d^N\times(\PP^N)^{\Vcal\setminus\Vcal^\circ}$. Further, the
codomain is irreducible, since~$\End_d^N$ is an open subset
of~$\PP^\dimEnd$.  Hence the map~\eqref{eqn:PiPcalnonempty} is
surjective.
\par\noindent\framebox{(d) $\Longleftrightarrow$ (e)}\enspace
These implications follow from the fact that the generic fibers of
$\End_d^N[\Pcal]\to\Moduli_d^N[\Pcal]$ and $\End_d^N\to\Moduli_d^N$
have dimension equal to the dimension of~$\PGL_{N+1}$.
\par\noindent\framebox{(c) $\Longrightarrow$ (a)}\enspace
The assumed surjectivity of the map~\eqref{eqn:PiPcalnonempty}
implies in particular that $\END_d^N[\Pcal]\ne\emptyset$.
But~$\END_d^N[\Pcal]$ is the Zariski closure of~$\End_d^N[\Pcal]$
in $\End_d^N\times(\PP^N)^{\Vcal\setminus\Vcal^\circ}$, so~$\End_d^N[\Pcal]$
cannot be the empty set.
\par\noindent\framebox{(a) $\Longrightarrow$ (f)}\enspace
A degree~$d$ map~$f$ on~$\PP^N$ has topological degree~$d^N$,
and Lemma~\ref{lemma:NumPrepPts} tells us that~$f$ can have at
most~$\nu_d^N(n)$ points of period~$n$. Hence~(a) implies~(f).
\par\noindent\framebox{(f) $\Longrightarrow$ (a)}\enspace
To see the opposite implication, we suppose that the inequalities
in~(f) are true. Let~$\Vcal'\subseteq\Vcal$ be the set of preperiodic
vertices of~$\Pcal$, let $\Pcal'=(\Vcal',\Phi)$ be the preperiodic
subportrait of~$\Pcal$, and let $m_0$ and $n_0$ be the maximal
preperiod and period, respectively, for a vertex $v \in \Vcal'$.  We
relabel the vertices in~$\Vcal$ so that
\[
  \Vcal=\{1,\ldots,s\},\quad
  \Vcal^\circ=\{q,\ldots,s\},\quad
  \Vcal'=\{r,\ldots,s\},\quad
  \text{with $q\le r\le s$,}
\]
i.e., the vertices are labeled from~$1$ to~$s$ so that
\[
  \underbrace{1,2,\ldots,q-1}_{\hidewidth\text{vertices with no out-arrow}\hidewidth},
  \underbrace{q,q+1,\ldots,r-1,\overbrace{r,r+1,\ldots,s}^{\hidewidth\text{preperiodic vertices}\hidewidth}}_{\text{vertices having an out-arrow}}.
\]

The fact that the generic degree~$d$ map on~$\PP^N$ has no ramified
preperiodic points implies that for each~$(m,n)$ with~$0\le m\le m_0$
and~$1\le n\le n_0$, there is a open subset of~$f\in\End_d^N$ such
that the map~$f$ has precisely~$\nu_d^N(m,n)$ points of
type~$(m,n)$. Hence we can find a map~$f$ and
points~$P_r,\ldots,P_s\in\PP^N$ so that~$(f,P_r,...,P_s)$ is a model
for~$P'$.

We next construct points~$P_k$ for 
vertices~$k\in\Vcal\setminus\Vcal^\circ=\{1,\ldots,q-1\}$ that do not
have an out-arrow. For each such~$k$, we let
\[
  m_k :=\max\bigl\{m : \text{$\Phi^m(v)=k$ for some $v\in\Vcal$}\bigr\}.
\]
We then successively choose points~$P_{q-1},\ldots,P_2,P_1\in\PP^N$ so that
\begin{equation}
  \label{eqn:PermfmCrit}
  P_k \notin
  \bigcup_{1\le m\le m_k} \Per_m(f) \cup
  \bigcup_{m=0}^{m_k} \bigcup_{i=k+1}^{s} f^{m}(P_i) \cup
  \bigcup_{m=0}^{m_k}  f^m\bigl(\Crit(f)\bigr),
\end{equation}
where~$\Crit(f)$ is the critical locus (also called the ramification
locus) of~$f$.

It remains to construct points~$P_q,\ldots,P_{r-1}\in\PP^N$ that
complete our model for~$\Pcal$. These correspond to the vertices
$\{q,\ldots,r-1\}$ that have out-arrows, but whose orbits terminate at
points with no out-arrow. We do this recursively as follows:
\begin{parts}
\Part{(1)}
Initialize $\Lcal:=\{q,\ldots,r-1\}$.
\Part{(2)}
Choose a vertex $\ell\in\Lcal$ such that $P_{\Phi(\ell)}$ has already been
assigned.
\Part{(3)}
Let $v_1,\ldots,v_t$ be the vertices in
$\Phi^{-1}\bigl(\Phi(\ell)\bigr)$. Note that~$t\ge1$,
since~$\ell\in\Phi^{-1}\bigl(\Phi(\ell)\bigr)$.
\Part{(4)}
Choose distinct points $P_{v_1},\ldots,P_{v_t}\in\PP^N$ satisfying
\[
f(P_{v_1})=\cdots=f(P_{v_t})=P_\ell.
\]
\Part{(5)}
Remove $v_1,\ldots,v_t$ from~$\Lcal$. If~$\Lcal=\emptyset$ terminate;
otherwise go to Step~(2) and choose another vertex in~$\Lcal$.
\end{parts}

We claim that this process yields a model~$(f,P_1,\ldots,P_s)$
for~$\Pcal$.  First, we verify that it is in fact possible to choose
the~$P_i$ as we have described.  For~$1\le k<q$, i.e., for vertices
with no out-arrows, the set~\eqref{eqn:PermfmCrit} that we need to
avoid when choosing~$P_k$ is a set of codimension~$1$ in~$\PP^N$,
since it consists of a union of a finite set of points together with
some forward images of the critical locus, which has codimension~$1$.
Hence these~$P_k$ may be chosen from a non-empty Zariski open set.

For vertices~$\ell$ with an out-arrow whose orbits terminate at a
point with no out-arrow, i.e., $q\le\ell<r$, Step~(4) requires that we
choose~$t$ distinct points in~$f^{-1}(P_\ell)$, where
$t=\#\F^{-1}\bigl(\F(\ell)\bigr)$.  Our assumption that $\Dcal_P\le
d^N$ means that $t\le d^N$, while~\eqref{eqn:PermfmCrit} tells us that
$P_\ell\notin{f}\bigl(\Crit(f)\bigr)$, i.e.,~$P_\ell$ is not a critical
value of~$f$, so~$\#f^{-1}(P_\ell)=d^N$. Hence there are enough
distinct points in the inverse image~$f^{-1}(P_\ell)$ to accomplish
Step~(4).

We also note that the algorithm~(1)--(5) terminates, since each time
we get to Step~(5), we remove at least one vertex from the set~$\Lcal$.

Finally, it is clear from the construction that for
each~$v\in\Vcal^\circ$, i.e., for each vertex with an out-arrow, we
have $f(P_v)=P_{\Phi(v)}$. And the points~$P_1,\ldots,P_s$ are
distinct due to the restrictions imposed by~\eqref{eqn:PermfmCrit}
and Step~(4). Therefore~$(f,P_1,\ldots,P_s)$ is a model for~$\Pcal$,
so~$\End_d^N[\Pcal]$ is non-empty. 

We have now completed the proof of the following implications:
\[
  \text{(e)}
  \Longleftrightarrow
  \text{(d)}
  \Longrightarrow
  \text{(c)}
  \Longrightarrow
  \text{(a)}
  \Longleftrightarrow
  \text{(b)}
  \Longleftrightarrow
  \text{(f)}
\]
In order to complete the proof of
Theorem~\ref{theorem:UnweightedNonempty}, it suffices to prove
that~(a) implies~(d). We first prove some lemmas describing various
codimensions.

\begin{lemma}
\label{lemma:codim}
Fix integers~$d \ge 2$ and~$N \ge 1$. For each pair of integers
$(m,n)$ with~$m \ge 0$ and~$n \ge 1$, define~$R_d^N(m,n)
\subset \End_d^N$ to be the set of endomorphisms~$f$ of~$\PP^N$ with
fewer than~$\nu_d^N(m,n)$ points of preperiodic portrait~$(m,n)$. Then
$R_d^N(m,n)$ has codimension~$1$ in~$\End_d^N$.
\end{lemma}

\begin{proof}
First let~$m = 0$. If~$f$ has fewer than~$\nu_d^N(n)$ points of
period~$n$, then \cite[Theorems~2.20~and~3.1]{arxiv:0801.3643} tells us
that one of the period-$n$ orbits of~$f$ has a multiplier 
that is an~$n$th root of unity.

For each point~$\alpha$ of period~$n$ for the generic map~$f$,
write~$J_\alpha$ for the Jacobian matrix of~$f^n$ at~$\alpha$, and for
each~$n$th root of unity~$\zeta$, let~$\lambda_{\alpha,\zeta}$ be the
determinant of~$J_\alpha - \zeta I_N$, where~$I_N$ is the~$N\times N$
identity matrix. Let~$\lambda$ denote the product
of~$\lambda_{\alpha,\zeta}$ over all~$\nu_d^N(n)$ points~$\alpha$ of
period~$n$ for~$f$ and all~$\zeta \in \mu_n$. Then~$\lambda$ lies in
the function field of~$\End_d^N$, since it is symmetric in the
periodic points of~$f$, so its vanishing will cut out~$R_d^N(0,n)$ as
a closed subscheme of~$\End_d^N$ of codimension at most~$1$. By
Lemma~\ref{lemma:NumPrepPts}, however,~$R_d^N(0,n)$ is a \emph{proper}
subscheme of~$\End_d^N$, hence the codimension is equal to~$1$.

Now let~$m \ge 1$.  For each~$\beta\in\PP^N_{\End_d^N}$, we impose a
polynomial condition~$\lambda_\beta$ by forcing the Jacobian~$J_\beta$
of $f$ at $\beta$ to have a nontrivial kernel, and we take~$\lambda$
to be the product of~$\lambda_\beta$ over all~$\beta$ of eventual
period~$n$ and preperiod at most~$m$. As in the previous
paragraph,~$\lambda$ cuts out a codimension~$1$ closed
subscheme~$S_d^N(m,n)$ of~$\End_d^N$, and we have
\[
R_d^N(m,n) = R_d^N(0,n) \cup S_d^N(m,n),
\]
which completes the proof of Lemma~\ref{lemma:codim}.
\end{proof}

\begin{lemma}
\label{lemma:RndPcodim1}
Let~$\Pcal$ be an unweighted portrait such that~$\End_d^N[\Pcal]$ is
nonempty, and let
\[
R_d^N(\Pcal) := \End_d^N \setminus \Image\bigl( \End_d^N[\Pcal]\to\End_d^N \bigr),
\]
i.e., let~$R_d^N(\Pcal)$ be the subscheme of~$\End_d^N$ consisting of
those maps~$f$ that cannot be extended to give a model
of~$\Pcal$. Then~$R_d^N(\Pcal)$ has codimension at least~$1$
in~$\End_d^N$.
\end{lemma}

\begin{proof}
If~$f \in R_d^N(\Pcal)$, then there exist some integers $m \ge 0$ and
$n \ge 1$ such that $f$ has fewer points of preperiodic type~$(m,n)$
than the portrait~$\mathcal P$, which itself has at most
$\nu_d^N(m,n)$ points of type~$(m,n)$ from the implication
\text{(a) $\Longrightarrow$ (f)} that we already proved in 
Theorem~\ref{theorem:UnweightedNonempty}. Thus, if~$n_0$
and~$m_0$ denote the maximal period and preperiod of a preperiodic
vertex of~$\Pcal$, then
\[
  R_d^N(\Pcal) \subseteq \bigcup_{\substack{0 \le m \le m_0\\1 \le n \le n_0}} R_d^N(m,n),
\]
and Lemma~\ref{lemma:codim} says that each~$R_d^N(m,n)$ has
codimension~$1$.
\end{proof}

\begin{lemma}
\label{lemma:piPcal}
Let~$\Pcal = (\Vcal^\circ, \Vcal, \Phi)$ be an unweighted
portrait. Let
\[
  \pi_\Pcal : \End_d^N[\Pcal] \to \End_d^N,\quad
  \pi_\Pcal(f,P_1,\ldots,P_k) = f,
\]
be the natural projection map. Then for
every~$f \in \End_d^N$ we have
\[
  \dim\bigl(\pi_\Pcal^{-1}(f)\bigr) =
    \begin{cases}
        N \cdot \#(\Vcal \setminus \Vcal^\circ) &\text{if $f \notin R_d^N(\Pcal)$,}\\
        0 &\text{if $f \in R_d^N(\Pcal)$}.
    \end{cases}
\]
\end{lemma}

\begin{proof}
We first note that by definition of~$R_d^N(\Pcal)$, we have~$f \in
R_d^N(\Pcal)$ if and only if the fiber~$\pi_\Pcal^{-1}(f)$ is
empty. We may therefore assume that
\text{$f\notin{R}_d^N(\Pcal)$}. Let~$r:=\#(\Vcal\setminus\Vcal^\circ)$.

First suppose that~$\Vcal^\circ = \emptyset$, i.e., suppose that~$\Pcal$
consists of~$r$ vertices with no arrows. Then clearly
\[
  \End_d^N[\Pcal] \cong \End_d^N \times (\PP^N)^r_\Delta,
\]
and the map~$\pi_\Pcal$ is just projection onto the~$\End_d^N$
factor. Therefore each fiber is isomorphic to~$(\PP^N)^r_\Delta$,
which is an open subset of~$(\PP^N)^r$, hence has dimension~$Nr$.

Next suppose~$\Vcal^\circ \ne \emptyset$. Write
\[
\Vcal=\{P_1,P_2,\ldots,P_r,Q_1,Q_2,\ldots,Q_s\},
\]
where~$\Vcal^\circ=\{Q_1,\ldots,Q_s\}$. Let~$\Pcal'$ be the
subportrait of~$\Pcal$ consisting only of the vertices
in~$\Vcal\setminus\Vcal^\circ$. Thus~$\Vcal'=\{P_1,\ldots,P_r\}$
and~${\Vcal'}^\circ=\emptyset$. Then~$\pi_\Pcal = \pi_{\Pcal'} \circ
\phi$, where~$\phi$ is the map
\begin{align*}
  \phi : \End_d^N[\Pcal] &\longrightarrow \End_d^N[\Pcal']\\
  \bigl(f,(x_1,\ldots,x_r,y_1,\ldots,y_s)\bigr) &\longmapsto \bigl(f,(x_1,\ldots,x_r)\bigr).
\end{align*}
It remains to show that the fibers of~$\phi$ are finite.

Suppose we have fixed the map~$f$ and the
points~$x_1,\ldots,x_r$. Since~$\#\Vcal$ is finite, the orbit of
each~$y_i$ under~$f$ must eventually reach either one of the~$x_j$ or
a periodic point. Since there are finitely many points with periods
specified by~$\Pcal$, and each point has at most~$d^N$ preimages
under~$f$, we are done with the proof of Lemma~\ref{lemma:piPcal}.
\end{proof}

We now have the tools required to prove the final implication that
completes the proof of Theorem~\ref{theorem:UnweightedNonempty}.

\par\noindent\framebox{(a) $\Longrightarrow$ (d)}\enspace
Let~$U_d^N(\Pcal) = \End_d^N \setminus R_d^N(\Pcal)$. By definition,
the image of the map~$\pi_\Pcal$ in Lemma~\ref{lemma:piPcal} is equal
to~$U_d^N(\Pcal)$.  Further, Lemma~\ref{lemma:RndPcodim1} tells us
that~$U_d^N(\Pcal)$ has full dimension, i.e.,
that~$\dim{U}_d^N(\Pcal)=\dim \End_d^N$. The desired result follows
from this and the fact that each fiber over $U_d^N(\Pcal)$ has
dimension~$N\cdot\#(\Vcal\setminus\Vcal^\circ)$.
\end{proof}

We conclude this section with a general result describing the
dimensions of the fibers and images of maps between portrait moduli
spaces.

\begin{theorem}
\label{theorem:MPprimetoMP}
Let $\Pcal'\subseteq\Pcal$ be an inclusion of unweighted portraits,
let
\begin{equation}
  \label{eqn:MPprimetoMP}
  M_{\Pcal,\Pcal'} : \Moduli_d^N[\Pcal]\longrightarrow \Moduli_d^N[\Pcal']
\end{equation}
be the natural induced map on moduli spaces as described in
Theorem~$\ref{theorem:MNdwithportrait}$, and assume that
$\Moduli_d^N[\Pcal]$ is non-empty.  Viewing~$\Pcal$ as a directed
graph, let~$n$ be the number of connected components of~$\Pcal$ that
intersect~$\Pcal'$ and do not contain a cycle. Then
\begin{align}
  \label{eqn:fiberMPpP}
  \dim(\textup{Fiber of $M_{\Pcal,\Pcal'}$})
     &= N\cdot\bigl(\#(\Vcal\setminus{\Vcal}^\circ) - n\bigr), \\*
  \label{eqn:imageMPpP}
  \codim(\textup{Image of $M_{\Pcal,\Pcal'}$})
     &= N\cdot\bigl(\#(\Vcal'\setminus{\Vcal'}^\circ) - n\bigr).
\end{align}
\end{theorem}
\begin{proof} 
Let $\bigl\langle f, (P_{v'})_{v' \in \Vcal'}\bigr\rangle \in
\Mcal_d^N[\Pcal']$ be in the image of $M_{\Pcal,\Pcal'}$, and let
$\bigl\langle f, (P_v)_{v \in \Vcal}\bigr\rangle \in \Mcal_d^N[\Pcal]$
be a preimage.  Much as in the proof of Lemma~\ref{lemma:piPcal}, a
point $P_v \in \PP^N$ corresponding to the vertex $v \in \Vcal$ may be
chosen freely from a Zariski open subset of $\PP^N$ if and only if $v$
satisfies both of the following conditions:
\begin{parts}
  \Part{\textbullet}
  $v$ has no outgoing arrow, i.e., $v\notin{\Vcal}^\circ$.
  \Part{\textbullet}
  $v$ is not in a component of $\Pcal$ containing vertices from $\Pcal'$, since otherwise
  there would be an $f$-orbit relation between~$P_{v}$ and a point~$P_{v'}$ with~$v'\in\Vcal'$.
\end{parts}
Hence~\eqref{eqn:fiberMPpP} is true.
\par
Next, we know from the 
\text{(a) $\Longrightarrow$ (e)}
implication of
Theorem~\ref{theorem:UnweightedNonempty} that
\[
  \dim \Moduli_d^N[\Pcal] = \dim\Moduli_d^N + N \cdot \#\bigl({\Vcal} \setminus {\Vcal}^\circ\bigr).
\]
Using this, we can compute the dimension of the image of~$M_{\Pcal,\Pcal'}$ as follows:
\begin{align*}
  \dim\bigl(\text{Image}&\text{ of $M_{\Pcal,\Pcal'}$}\bigr)\\
  &= \dim\Mcal_d^N[\Pcal] 
  - \dim\bigl(\text{Fiber of $M_{\Pcal,\Pcal'}$}\bigr) \\
  & = \bigl( \dim\Moduli_d^N + N\cdot\#(\Vcal\setminus{\Vcal}^\circ) \bigr)
  - N\cdot \bigl( \#(\Vcal\setminus{\Vcal}^\circ)-n \bigr) \\
  &= \dim\Moduli_d^N+N\cdot n.    
\end{align*}
Subtracting this from the formula for~$\dim\Moduli_d^N[\Pcal']$ in
Theorem~\ref{theorem:UnweightedNonempty}(e) gives the desired value
for the codimension of the image of $M_{\Pcal,\Pcal'}$.
\end{proof}

\section{The Dimension of $\Moduli_d^1[\Pcal]$ for Weighted Portraits}
\label{section:dimMd1P}

In this section we study~$\Moduli_d^1[\Pcal]$, where~$\Pcal$ is now
allowed to be a weighted portrait. We assume throughout that we are
working over a field of characteristic~$0$.  As noted in
Remark~\ref{remark:whenemptyset}, there are geometric constraints that
may force $\Moduli_d^N[\Pcal]=\emptyset$.  For example, if $\Pcal$
contains four fixed points, or if~$\Pcal$ contains two periodic cycles
of length~$2$, then~$\Moduli_2^1[\Pcal]=\emptyset$; see
Theorem~\ref{theorem:UnweightedNonempty}(f).  More generally, we have
the following natural geometric conditions that are necessary, but may
not be sufficient, to ensure that~$\Moduli_d^1[\Pcal]$ is non-empty.

\begin{proposition}
\label{proposition:whenisMPempty}
Let~$d\ge2$, and let~$\Pcal=(\Vcal^\circ,\Vcal,\Phi,\e)$ be a portrait such
that~~$\Moduli_d^1[\Pcal]\ne\emptyset$. Then~$\Pcal$ satisfies the
following conditions:
\[
\textup{(I)}\quad\max_{v\in\Vcal} \sum_{w\in \Phi^{-1}(v)} \e(w) \le d.
\qquad
\textup{(II)}\quad\sum_{w\in\Vcal^\circ} \bigl(\e(w)-1\bigr) \le 2d-2.
\]
For all $n\ge1$, 
\[
  \textup{(III$_n$)}\quad
  \#\left\{w\in\Vcal^\circ :
  \begin{array}{@{}l@{}}
    \Phi^n(w)=w~\text{and}\\\Phi^m(w)\ne w~\text{for all $m<n$}\\
  \end{array}\right\}
  \le \nu_d^1(n).
\]
\textup(See~\eqref{eqn:nudNn} in Definition~$\ref{definition:nudNn}$
for the definition of~$\nu_d^N(n)$.\textup)
\end{proposition}

\begin{proof}
Constraint~I comes from the fact that~$f$ is a map of degree~$d$, so a
point has at most~$d$ preimages.  Constraint~II follows from the
Riemann-Hurwitz formula $\sum\bigl(e_f(P)-1\bigr)=2d-2$;
see~\cite[Theorem~1.1]{MR2316407}.  Constraint~III$_n$ is true because
a degree~$d$ map on~$\PP^1$ has at most the indicated number of points
of exact period~$n$; see~\cite[Remark~4.3]{MR2316407}. We remark that
if~$\Pcal$ is unweighted, then constraint~(II) is vacuous.
\end{proof}

\begin{remark}\label{remark:necessarynotsufficient}
For unweighted portraits,
Theorem~\ref{theorem:UnweightedNonempty}(f) gives succinct
necessary and sufficient condition for~$\Mcal_d^1[\Pcal]$ to be
nonempty. The situation is more complicated in the case of weighted
portraits. For example, if~$\Pcal$ consists of four fixed points, each
having multiplicity~$2$, then the explicit computation in
Example~\ref{example:M31with4fixedpts} shows
that~$\Moduli_3^1[\Pcal]=\emptyset$, despite the fact that~$\Pcal$
satisfies properties~(I),~(II), and~(III$_n$) for all $n \ge 1$ with $d=3$ in
Proposition~\ref{proposition:whenisMPempty}.
\end{remark}

\begin{remark}
\label{remark:thurstontopchar}
As illustrated in Remark~\ref{theorem:UnweightedNonempty}, the
elementary conditions given in
Proposition~\ref{proposition:whenisMPempty} for a portrait to be
realized by a degree~$d$ rational map are necessary but not
sufficient.  A deeper topological characterization of rational
functions was formulated by Thurston and proved by Douady and Hubbard
in \cite{MR1251582}. We give a brief description. Let $F:S^2\to S^2$
be a topological branched cover of the 2-sphere, and suppose that the
forward orbits of the branch points are finite.  Thurston defines~$F$
to be equivalent to a PCF rational map~$f:\PP^1(\CC)\to\PP^1(\CC)$ if
there is a homeomorphism of~$S^2\cong\PP^1(\CC)$ taking~$F$ to~$f$ and
correctly identifying the postcritical set~$P_F$ of~$F$ with the
postcritical set~$P_f$ of~$f$.  Thurston uses~$F$ and~$P_F$ to define
a map
$\s_F:\operatorname{Teich}(S^2,P_F)\to\operatorname{Teich}(S^2,P_F)$
on a certain Teichm\"uller space.  Then his topological
characterization says that~$F$ is equivalent to a PCF rational map if
and only if the \emph{Thurston pullback map}~$\s_F$ has a fixed point.
\end{remark}

\begin{remark}
As noted earlier, in this section we work over a field of
characteristic~$0$. To see why, we note that otherwise there may be
problems with Condition~(II) in
Proposition~\ref{proposition:whenisMPempty}.  For example, let~$\Pcal$
consist of~$3$ fixed points, each of multiplicity~$2$.  Then
$\Moduli_2^1[\Pcal](\CC)=\emptyset$, as specified by
Proposition~\ref{proposition:whenisMPempty}. But
$\Moduli_2^1[\Pcal](\FF_2)$ contains a point. Indeed,
\[
  \Moduli_2^1[\Pcal](\FF_2) = \Bigl\{ \bigl( x^2, (0,1,\infty) \bigr) \Bigr\}.
\]
Wild ramification examples of this sort are closely related to the
failure of the classical Riemann--Hurwitz formula in positive
characteristic.
\end{remark}

\begin{remark}
Note that we can use weighted portraits to easily describe moduli
spaces of polynomial maps on~$\PP^1$, since~$\Moduli_d^1[\Pcal]$
parameterizes polynomial maps if and only
if~$\Pcal=(\Vcal^\circ,\Vcal,\Phi,\e)$ has a vertex~$v\in\Vcal^\circ$
satisfying $\Phi(v)=v$ and $\e(v)=d$.
\end{remark}

Before stating the main result of this section, we need a definition.

\begin{definition}
Let $\Pcal$ be a portrait. We say that a point
$(f,\bfP)\in\End_d^N[\Pcal](\Kbar)$ is a \emph{flexible Latt{\`e}s
  model for~$\Pcal$} if~$f\in\End_d^N(\Kbar)$ is a flexible Latt\`es
map.  We refer the reader to~\cite[Section~6.5]{MR2316407} for the
definition of flexible Latt\`es map, but roughly, it means that~$f$
comes from a multiplication-by-$m$ map on an elliptic curve. In
particular, the locus of flexible Latt\`es maps in~$\Moduli_d^1$ is
empty if~$d$ is not a square, and it is a copy of~$\PP^1$ if~$d$ is a
square.  We then define
\[
  \End_d^1[\Pcal]^{\text{nL}}:=\End_d^1[\Pcal]\setminus(\text{flexible Latt\`es})
\]
to be the complement of the points corresponding to flexible Latt\`es
models, and similarly for~$\Moduli_d^1[\Pcal]^{\text{nL}}$.
\end{definition}

\begin{theorem}
\label{thm:dimMd1}
Let~$d \ge 2$, and let~$\Pcal = (\Vcal^\circ,\Vcal,\Phi,\epsilon)$ be a
portrait such that~$\Mcal_d^1[\Pcal] \ne \emptyset$. Then
\[
  \dim \Mcal_d^1[\Pcal]^{\text{nL}}
  = \dim \Mcal_d^1 - \sum_{v \in \Vcal^\circ} (\epsilon(v) - 1) + \#(\Vcal \setminus \Vcal^\circ).
\]
\end{theorem}

\begin{remark}
Since~$\Mcal_d^1[\Pcal]^{\text{nL}}$ and~$\Mcal_d^1$ have dimension~$3$ less than
$\End_d^1[\Pcal]^{\text{nL}}$ and~$\End_d^1$, respectively, it suffices to prove
Theorem~\ref{thm:dimMd1} with~$\Mcal$ replaced with~$\End$, i.e., to
prove that
\begin{equation}\label{eq:dimEndd1}
  \dim \End_d^1[\Pcal]^{\text{nL}} = \dim \End_d^1
  - \sum_{v\in\Vcal^\circ} (\epsilon(v) - 1) + \#(\Vcal \setminus \Vcal^\circ).
\end{equation}
\end{remark}

Recall that~$\End_d^1[\Pcal]$ is the subvariety of
$\End_d^1\times(\PP^1)^S_\Delta$ determined by the following two
conditions:\footnote{We remark that if we write~$f(x)\in{K}(x)$ as a
  rational function, and if we take a point $\a\in{K}$ satisfying
  $f(\a)\ne\infty$, then an inequality~$e_f(\a)\ge\e$ as
  in~\eqref{eq:End2} simply says that $(d^nf/dx^n)(\a)=0$ for all
  $1\le{n}\le\e-1$.}
\begin{align}
  f(P_{w}) &= P_{\Phi(w)} &&\text{ for all } w \in \Vcal^\circ. \label{eq:End1} \\
  e_f(P_w) &\ge \epsilon(w) &&\text{ for all } w \in \Crit(\Pcal) \label{eq:End2}.
\end{align}
There are $\#\Vcal^\circ + \sum_{w \in \Vcal^\circ} (\epsilon(w) - 1)$
equations given in~\eqref{eq:End1} and~\eqref{eq:End2}; since
$\End_d^1[\Pcal]^{\text{nL}}$ is nonempty, this is an upper bound for the
codimension of $\End_d^1[\Pcal]^{\text{nL}}$ in $\End_d^1 \times
(\PP^1)^S_\Delta$, hence
\begin{multline*}
 \dim \End_d^1[\Pcal]^{\text{nL}} \ge \dim \End_d^1 + S - \left(\#\Vcal^\circ + \sum_{w \in \Vcal^\circ} (\epsilon(w) - 1)\right)\\
         = \dim \End_d^1 - \sum_{w \in \Vcal^\circ} (\epsilon(w) - 1) + \#(\Vcal \setminus \Vcal^\circ).
\end{multline*}
Thus, in order to complete the proof of Theorem~\ref{thm:dimMd1}, it remains to show that
\begin{equation}
  \label{eqn:DimIneq}
  \dim \End_d^1[\Pcal]^{\text{nL}} \le \dim \End_d^1 - \sum_{w \in \Vcal^\circ} (\epsilon(w) - 1) + \#(\Vcal \setminus \Vcal^\circ).
\end{equation}

\subsection{A Reduction Step}
Because our main focus in this section is on critical points and their
corresponding critical relations, we specify, for a weighted portrait
$\Pcal$, a minimal subportrait of $\Pcal$ that captures the same
critical data as $\Pcal$ itself.

\begin{definition}\label{definition:CriticallyGenerated}
Let $\Pcal = (\Vcal,\Vcal^\circ,\Phi,\epsilon)$ be a portrait. The
\emph{critically generated subportrait} of $\Pcal$ is the subportrait
$\Pcal' = (\Vcal', {\Vcal'}^\circ, \Phi', \epsilon')$ defined by
\[
 \Vcal' := \bigcup_{v \in \Crit(\Pcal)} 
 \Ocal_\Phi(v),\quad {\Vcal'}^\circ := \Vcal' \cap
 \Vcal^\circ,\quad \Phi' := \left.\Phi\right|_{{\Vcal'}^\circ},\quad
 \epsilon' := \left.\epsilon\right|_{{\Vcal'}^\circ}.
\]
We say that $\Pcal$ is \emph{critically generated}\footnote{
If $\sum_{v \in \Vcal^\circ} \bigl(\e(v) - 1\bigr) = 2d - 2$,
then $\Pcal$ being critically generated is equivalent to $\Pcal$ being a complete critical
portrait, using the terminology of Section~\ref{section:P1critmarkedc}.} if $\Pcal'=\Pcal$. 
\end{definition}

We now show that for the purposes of proving Theorem~\ref{thm:dimMd1},
it suffices to consider only critically generated portraits.

\begin{lemma}\label{lem:ReduceToCG}
If Condition~\eqref{eqn:DimIneq} holds for all critically generated portraits,
then Condition~\eqref{eqn:DimIneq} holds for all portraits.
\end{lemma}

\begin{proof}
Let $\Pcal$ be any portrait, and let $\Pcal'$ be its critically generated subportrait. By assumption, we have
\[
   \dim \End_d^1[\Pcal']^{\text{nL}} \le \dim \End_d^1 - \sum_{v \in {\Vcal'}^\circ} (\epsilon'(v) - 1) + \#(\Vcal' \setminus {\Vcal'}^\circ).
\]
Consider the natural projection map
(Proposition~\ref{proposition:EnddNP1toP2}) restricted to the
non-Latt\`es locus,
\[
   \pi_{\Pcal,\Pcal'} : \End_d^N[\Pcal]^{\text{nL}} \longrightarrow \End_d^N[\Pcal']^{\text{nL}}.
\]
As in the proof of Theorem~\ref{theorem:MPprimetoMP}, the dimension of
a nonempty fiber of $\pi_{\Pcal,\Pcal'}$ is equal to
$\#(\Vcal\setminus\Vcal^\circ)-\#(\Vcal'\setminus{\Vcal'}^\circ)$. Thus,
noting that
\[
  \sum_{v \in {\Vcal'}^\circ} (\epsilon'(v) - 1) = \sum_{v \in\Vcal^\circ} (\epsilon(v) - 1)
\]
by construction of $\Pcal'$, we have
\begin{align*}
  \dim \End_d^1[\Pcal]^{\text{nL}}
   &= \dim \Image(\pi_{\Pcal,\Pcal'}) + \dim \Fiber(\pi_{\Pcal,\Pcal'})\\
   &\le \dim \End_d^1[\Pcal']^{\text{nL}} + \left(\#(\Vcal \setminus \Vcal^\circ) - \#(\Vcal' \setminus {\Vcal'}^\circ)\right)\\
  &\le \dim \End_d^1 -
  \smash[b]{ \sum_{v \in \Vcal^\circ} (\epsilon(v) - 1) + \#(\Vcal \setminus \Vcal^\circ), }
\end{align*}
so~\eqref{eqn:DimIneq} holds for $\Pcal$.
\end{proof}

\begin{remark}
In the proof of Lemma~\ref{lem:ReduceToCG}, we did not assume that
$\Pcal$ has a vertex with weight greater than $1$.  If $\Pcal$ is
unweighted, i.e., if $\epsilon(v) = 1$ for all $v \in \Vcal^\circ$,
then the critically generated subportrait $\Pcal'$ is the empty
portrait, and in which case we have
$\End_d^1[\Pcal']=\End_d^1[\emptyset]=\End_d^1$. The argument is no
different in this situation.
\end{remark}

\subsection{Critical Relations Realized by Critically Generated Portraits}
For the remainder of this section, we assume that $\Pcal$ is a
critically generated portrait.

In order to prove Theorem~\ref{thm:dimMd1}, we relabel vertices of $\Pcal$ using
\emph{pairs} of positive integers as follows:
First, we decompose~$\Vcal$ as a disjoint union
\[
  \Vcal = \Ccal_1 \cup \cdots \cup \Ccal_r,
\]
consisting of the vertices in the distinct connected components of the
graph attached to~$\Pcal$. Note that every component contains at least one critical point by our assumption that $\Pcal$ is critically generated.

For each~$1\le k\le r$, we label the elements of~$\Ccal_k$ as
\[
  \Ccal_k = \{(1,k),\ldots,(s_k,k)\},
\]
and we choose the labeling so that there is a~$t_k\le s_k$ so that
\[
  \text{$(i,k)\in\Ccal_k$ is a critical point}
  \quad\Longleftrightarrow\quad
  1\le i\le t_k.
\]
In other words, we have
\begin{align*}
	\Vcal &= \{(i,k) : 1 \le k \le r,\; 1 \le i \le s_k\}~\text{and}\\
	\Crit(\Pcal) &= \{(i,k) : 1 \le k \le r,\; 1 \le i \le t_k\}.
\end{align*}
Let
\[
  S := \sum_{k=1}^r s_k
  \quad\text{and}\quad
  T := \sum_{k=1}^r t_k
\]
be, respectively, the total number of vertices in~$\Vcal$ and the
total number of critical points for~$\Vcal$. 
Finally, we let
\begin{equation}
  \label{eqn:zetaPcal}
  \zeta(\Pcal) := \#(\Vcal \setminus \Vcal^\circ)
\end{equation}
be the number of components of~$\Pcal$ that do not contain cycle.

\begin{definition}
A \emph{critical relation realized by~$\Pcal$} is an identity of the form
\begin{equation}
\label{eqn:ikjellmnrel}
\Phi^m(i,k) = \Phi^n(j,\ell)
\end{equation}
with $(i,k), (j,\ell) \in \Crit(\Pcal)$. We also encode such a relation as a tuple
\[
\bigl((i,k), (j,\ell) ; m, n\bigr).
\]
The relation is defined to be \emph{trivial} if $(i,k) = (j,\ell)$
and~$m=n$, and \emph{nontrivial} otherwise.
\end{definition}

Note that for any critical relation~\eqref{eqn:ikjellmnrel}, we must
have~$k=\ell$, since otherwise the vertices~$(i,k)$ and~$(j,\ell)$
would be in different components.

We seek to construct a minimal set of critical relations realized by
$\Pcal$ that completely determines \emph{all} relations realized by
$\Pcal$. First, we make more explicit what we mean for a set
of critical relations to determine another relation. The following
definition is motivated by \cite[\textsection~3.1]{arxiv1702.02582}.

\begin{definition}
Let~$\Scal$ be a collection of critical relations realized by
$\Pcal$. We let~$\sim_\Scal$ denote the smallest equivalence relation
$\sim$ on the set
\[
  \Bigl\{\bigl((i,k), m\bigr) \mid 1 \le k \le r,\, 1 \le i \le t_k,\, m \in \ZZ_{\ge 0}\Bigr\}
\]
with the property that
\begin{multline}
\label{eqn:ikmikmc}
\bigl((i,k),m\bigr) \sim \bigl((j,\ell), n\bigr) \\
{}\Longrightarrow\quad
\bigl((i,k),m+c\bigr)\sim\bigl((j,\ell),n+c\bigr)\text{ for all }
c \in \ZZ_{\ge 0}.
\end{multline}
We refer to the condition~\eqref{eqn:ikmikmc} by saying that
$\sim_\Scal$ is \emph{closed under iteration}. We say that a
relation~$\bigl((i,k), (j,\ell) ; m, n\bigr)$ as
in~\eqref{eqn:ikjellmnrel} is \emph{determined by~$\Scal$} if
$\bigl((i,k), m) \sim_\Scal ((j,\ell), n\bigr)$.
\end{definition}

We now define a particular set of critical relations for~$\Pcal$,
which we will then show determines all critical relations realized by
$\Pcal$.

\begin{definition}\label{definition:S_P}  
We construct a set ~$\Scal_\Pcal$ using the following algorithm:  
\begin{enumerate}
\item[(A)]
  Initialize $\Scal_\Pcal := \emptyset$, $k := 1$, and $i := 1$.
\item[(B)]
  If~$i = 1$ and the component~$\Ccal_k$ has no cycle, skip  to Step (G).
\item[(C)]
  Let~$m_{i,k}$ be the minimal~$m \ge 0$ such that~$\Phi^m(i,k) =
  \Phi^n(j,k)$ for some~$1 \le j \le i$ and~$n \ge 0$, with~$n < m$
  if~$i = j$. (Step (B) ensures that there is such an~$m$.)
\item[(D)]
  Let~$j_{i,k}$ be the minimal~$1 \le j \le i$ such
  that~$\Phi^{m_{i,k}}(i,k) = \Phi^n(j,k)$ for some~$n \ge 0$.
\item[(E)]
  Let~$n_{i,k}$ be the minimal~$n \ge 0$ such that~$\Phi^{m_{i,k}}(i,k) = \Phi^n(j_{i,k},k)$.
\item[(F)]
  Adjoin the tuple~$\bigl((i,k),(j_{i,k},k);m_{i,k},n_{i,k}\bigr)$ to the
  set~$\Scal_\Pcal$.
\item[(G)]
  Increment~$i$. If~$i\le t_k$, go to Step (C).
\item[(H)]
  Increment~$k$. If~$k\le r$, go to Step (B).
\end{enumerate}
\end{definition}

We recall the quantity~$\zeta(\Pcal)$ defined by~\eqref{eqn:zetaPcal},
and we note that the set~$\Scal_\Pcal$ contains
precisely~$T-\zeta(\Pcal)$ tuples, since Step~(B)
eliminates~$\zeta(\Pcal)$ of the critical points. We now show
that~$\Scal_\Pcal$ determines \emph{all} critical relations
for~$\Pcal$.

\begin{proposition}\label{prop:P_determined}
Let~$\Scal_\Pcal$ be as in Definition~$\ref{definition:S_P}$. Then
every critical relation $\Phi^m(i,k) = \Phi^n(j,\ell)$ realized
by~$\Pcal$ is determined by $\Scal_\Pcal$.
\end{proposition}

\begin{proof}
We drop the subscript and write $\sim$ for the equivalence
relation~$\sim_{\Scal_\Pcal}$. Since~$\sim$ is reflexive, any trivial
relation is immediately determined by~$\Scal_\Pcal$. Henceforth, we
consider only nontrivial critical relations.

As noted previously, for any critical relation $\Phi^m(i,k) =
\Phi^n(j,\ell)$ we must have~$k=\ell$. Now, fix~$k\in\{1,\ldots,r\}$,
and assume that~$i \ge j$. We proceed by induction on~$i$.

First, we take~$i = 1$, hence also~$j = 1$. There is a nontrivial
critical relation of the form~$\Phi^m(1,k) = \Phi^n(1,k)$ if and only
if the component~$\Ccal_k$ contains a cycle. We assume that~$m \ge
n$. In this case, if we let~$(M,N)$ denote the preperiodic type of
$(1,k)$, then we must have~$n \ge M$ and~$m = n + cN$ for some~$c \ge
1$. (If~$c = 0$, then the relation is trivial.) Therefore, it suffices
to show that~$\bigl((1,k), n\bigr) \sim \bigl((1,k), n + cN\bigr)$ for
all integers~$n \ge M$ and~$c \ge 1$.

Using the notation from Definition~\ref{definition:S_P} we have 
\[
  j_{1,k} = 1,\qquad
  m_{1,k} = M + N,\qquad
  n_{1,k} = M.
\]
Thus~$\bigl((1,k), M\bigr) \sim \bigl((1,k), M + N\bigr)$. Since~$n
\ge M$, and since~$\sim$ is closed under iteration, it follows also
that~$\bigl((1,k), n\bigr) \sim \bigl((1,k), n + N\bigr)$. In fact, we
have~$\bigl((1,k), n + (c-1)N\bigr) \sim \bigl((1,k), n + cN\bigr)$
for all~$c \ge 1$, hence by transitivity we have~$\bigl((1,k), n\bigr)
\sim \bigl((1,k), n + cN\bigr)$ for all~$c \ge 1$ as claimed.

Now take~$2 \le i \le t_k$, and assume that for all~$1 \le j' \le i' <
i$, if~$\Phi^{m'}(i',k) = \Phi^{n'}(j',k)$, then that relation is
determined by~$\Scal_\Pcal$. Since~$i \ge 2$, we must have~$j_{i,k} <
i$. We consider two cases.

\textbf{Case 1.} Assume~$i = j$, in which case~$\Ccal_k$ contains a
cycle. Let~$(M, N)$ be the preperiodic type of~$(i,k)$. By following
the same argument as in the~$i = 1$ case, it will suffice to show
that~$\bigl((i,k), M\bigr) \sim \bigl((i,k), M + N\bigr)$.

The vertex~$\Phi^M(i,k)$ lies in the~$N$-cycle in~$\Ccal_k$, which is
contained in the orbit of~$(1,k)$; thus~$m_{i,k} \le M$. By definition
of~$\Scal_\Pcal$, we have
\[
  \bigl((i,k), m_{i,k}\bigr) \sim \bigl((j_{i,k}, k), n_{i,k}\bigr),
\]
and by closure under iteration we have
\[
  \bigl((i,k), M\bigr) \sim \bigl((j_{i,k}, k), n_{i,k} + M - m_{i,k}\bigr)
\]
and
\[
  \bigl((i,k), M + N\bigr) \sim \bigl((j_{i,k}, k), n_{i,k} + M - m_{i,k} + N\bigr).
\]
It therefore suffices to show that
\begin{equation}
  \label{eq:rel1}
  \bigl((j_{i,k}, k), n_{i,k} + M - m_{i,k}\bigr) \sim \bigl((j_{i,k},
  k), n_{i,k} + M - m_{i,k} + N\bigr).
\end{equation}
Now, by construction we have
that~$\Phi^{n_{i,k}}(j_{i,k},k)=\Phi^{m_{i,k}}(i,k)$; thus, since~$M
\ge m_{i,k}$, we have
\[
\Phi^{M - m_{i,k} + n_{i,k}}(j_{i,k},k) = \Phi^M(i,k),
\]
which is periodic of period~$N$. It follows that
\[
  \Phi^{M - m_{i,k} + n_{i,k} + N}(j_{i,k},k)
  = \Phi^{M - m_{i,k} + n_{i,k}}(j_{i,k},k).
\]
Since~$j_{i,k} < i$, our induction assumption guarantees that~\eqref{eq:rel1} holds.

\textbf{Case 2.} Assume~$i > j$. Since~$\Phi^m(i,k) = \Phi^n(j,k)$, we
have~$m_{i,k} \le m$ by definition, and furthermore
\[
  \bigl((i,k), m_{i,k}\bigr) \sim \bigl((j_{i,k},k), n_{i,k}\bigr).
\]
By closure under iteration, since~$m \ge m_{i,k}$ we get
\[
  \bigl((i,k), m\bigr) \sim \bigl((j_{i,k},k), n_{i,k} + m - m_{i,k}\bigr).
\]
On the other hand, we have
\[
  \Phi^{n_{i,k} + m - m_{i,k}}(j_{i,k}, k)
  = \Phi^m(i,k) = \Phi^n(j,k),
\]
and since~$j_{i,k}, j< i$, our induction hypothesis implies that
\[
  \bigl((j_{i,k},k), n_{i,k} + m - m_{i,k}\bigr) \sim \bigl((j,k),n\bigr).
\]
By transitivity, we get~$\bigl((i,k),m\bigr) \sim \bigl((j,k),n\bigr)$, as claimed.
\end{proof}

\subsection{Proof of Theorem~\ref{thm:dimMd1}}
\label{section:proofthm:dimMd1}
The key to proving inequality~\eqref{eqn:DimIneq}, and therefore
Theorem~\ref{thm:dimMd1}, is to note that the relations coming from
\eqref{eq:End1} and~\eqref{eq:End2} intersect transversally at points
in $\End_d^1[\Pcal]$. This, in turn, relies on Thurston rigidity; see
\cite{MR1251582} and \cite{arxiv1702.02582}.

\begin{proof}[Proof of Theorem~$\ref{thm:dimMd1}$]
Lemma~\ref{lem:ReduceToCG} says that it suffices to assume that
$\Pcal$ is critically generated, and thus we need to show
that~\eqref{eqn:DimIneq} is satisfied by such~$\Pcal$.

Let
\[
  \left(f, (P_{i,k})_{\substack{1 \le k \le r\\1 \le i \le s_k}}\right) \in \End_d^1[\Pcal]^{\text{nL}}.
\]
The portrait $\Pcal$ determines $T$ \emph{distinct} critical points
for $f$, accounting for at least
\[
  \sum_{k=1}^r\sum_{i=1}^{t_k}(\epsilon(i,k) - 1)
\]
critical points when counted \emph{with multiplicity}. This means that
the map $f$ has at most
\[
  T' := T + 2d - 2 - \sum_{k=1}^r\sum_{i=1}^{t_k} (\epsilon(i,k) - 1)
\]
distinct critical points. By \cite[Theorem 2.1]{arxiv1702.02582}, the
(generally reducible) subvariety $\End_{d, T'}^1 \subset \End_d^1$ of
maps with at most $T'$ distinct critical points has dimension at most
$T' + 3$. It follows that the subvariety
$\Xcal\subset\End_d^1\times(\PP^1)^T_\Delta$ defined by
\[
  e_f(P_{i,k}) \ge \epsilon(i,k) \quad\text{for all $1 \le k \le r$, $1 \le i \le t_k$,}
\]
has dimension at most $T' + 3$, since the projection map
$\Xcal\to\End_d^1$ has finite fibers, and its image lies in
$\End_{d,T'}^1$.

Let $\Zcal$ be the subvariety of $\Xcal$ cut out by
\begin{equation}
  \label{eqn:PcalRels}
   f^m(P_{i,k}) = f^n(P_{j,\ell}) \quad\text{for all $((i,k),(j,\ell) ; m,n) \in \Scal_\Pcal$}.
\end{equation}
Then the image of the map $\pi_\Xcal : \End_d^1[\Pcal] \to \Xcal$
obtained by projecting onto the appropriate coordinates is contained
in $\Zcal$. Furthermore, by construction of $\Scal_\Pcal$, there can
be no redundancy among the relations; that is, there is no subset of
distinct elements of $\Scal_\Pcal$ of the form
\begin{align*}
   ((i_1,k),(i_2,k)&;m_1,n_1), ((i_2,k), (i_3,k); m_2,n_2),\\
   &\ldots, ((i_{\ell - 1},k),(i_\ell,k);m_{\ell - 1},n_{\ell - 1}), ((i_\ell,k), (i_1,k); m_\ell,n_\ell).
\end{align*}
Thus, Thurston rigidity implies\footnote{More explicitly, one can
  extend $\Scal_\Pcal$ to a larger set of relations
  $\Scal_f\supseteq\Scal_\Pcal$ which is \emph{minimally full} in the
  sense of \cite[Definition 3.6]{arxiv1702.02582}, and then
  transversality follows from \cite[Lemma 4.1]{arxiv1702.02582} and
  the proof of \cite[Theorem 3.2]{arxiv1702.02582}.} that the
relations defining $\Zcal$ intersect transversally in $\Xcal$ at each non-Latt\'es
point in $\Image(\pi_\Xcal)$. Therefore, we have
\begin{align}
  \label{eqn:FinalIneq}
  \dim \Image(\pi_\Xcal|_{\End_d^1[\Pcal]^{\text{nL}}})\notag
         &= \dim \Xcal - \#\Scal_\Pcal\notag\\
         &\le (T' + 3) - (T - \zeta(\Pcal))\notag\\
         &= \dim \End_d^1 - \sum_{k=1}^r\sum_{i=1}^{t_k} (\epsilon(i,k) - 1) + \zeta(\Pcal)\notag\\
         &= \dim \End_d^1 - \sum_{v \in \Vcal^\circ} (\epsilon(v) - 1) + \#(\Vcal \setminus \Vcal^\circ),
\end{align}
where the last equality comes from the fact that we defined
$\zeta(\Pcal)$ to be equal to $\#(\Vcal \setminus \Vcal^\circ)$. The fact that $\Pcal$ is critically generated
implies that for every vertex $(j,\ell) \in \Vcal$, there exist
$(i,k) \in \Crit(\Pcal)$ and $m \ge 0$ such that
$\Phi^m(i,k)=(j,\ell)$. This implies that $P_{j,\ell} = f^m(P_{i,k})$,
and thus every coordinate $P_{j,\ell}$ is determined by just the
critical point coordinates. It follows that
$\End_d^1[\Pcal]^{\text{nL}}$ is isomorphic to its image under
$\pi_\Xcal$, and therefore the inequality~\eqref{eqn:FinalIneq}
implies~\eqref{eqn:DimIneq}.
\end{proof}

\section{An FOD/FOM Degree Bound for Dynamical Systems with Portraits}
\label{section:fomfoddynportdegbd}

In this section, we turn to a discussion of the field-of-moduli versus
field-of-definition problem associated to the
space~$\Moduli_d^N[\Pcal]$ with an unweighted portrait~$\Pcal$. In the
next section we apply this material to explain how the uniform
boundedness conjecture is related to rational points
on~$\Moduli_d^N[\Pcal]$.

The main theorem of~\cite{fomfodPN2018} says that if~$K$ is a number
field and~$\xi\in\Moduli_d^N(K)$, then there is a field
extension~$L/K$ whose degree is bounded in terms of~$N$ and~$d$ so
that~$\xi$ comes from a point in~$\End_d^N(L)$. We generalize to
portrait moduli spaces~$\Moduli_d^N[\Pcal]$, obtaining a result that will be
needed in the next section for comparing two uniform boundedness
conjectures.

\begin{theorem}
\label{theorem:FODFOMbound} 
Fix integers~$N\ge1$ and $d\ge2$. There is a constant~$C_2(N,d)$ such that
for every unweighted portrait~$\Pcal$ and every number field~$K$ we
have
\[
  \Moduli_d^N[\Pcal](K) \subset \Image \left(
  \bigcup_{[L:K]\le C_2(N,d)} \End_d^N[\Pcal](L)
  \longrightarrow \Moduli_d^N[\Pcal](\Kbar) \right).
\]
In other words, every $\xi\in\Moduli_d^N[\Pcal](\Kbar)$ whose field of
moduli is contained in~$K$ has a field of definition whose degree
over~$K$ is bounded in terms of~$N$ and~$d$.
\end{theorem}
\begin{proof}
Recall that we use angle brackets to denote the maps
\[
  \<\;\cdot\;\> : \End_d^N[\Pcal]\longrightarrow\Moduli_d^N[\Pcal]
  \quad\text{and}\quad
  \<\;\cdot\;\> : \End_d^N\longrightarrow\Moduli_d^N.
\]
We start with an element $\<f,\bfP\>\in\Moduli_d^N[\Pcal](K)$, i.e.,
the field of moduli of~$(f,\bfP)$ is contained
in~$K$. Then~$\<f\>\in\Moduli_d^N(K)$,
so~\cite[Theorem~1]{fomfodPN2018} says that there is
a~$\psi\in\PGL_{N+1}(\Kbar)$ so that~$f^\psi\in\End_d^N(L)$ for a
field~$L$ whose degree over~$K$ is bounded in terms of~$N$
and~$d$. Replacing~$K$ with~$L$ and~$(f,\bfP)$
with~$\bigl(f^\psi,\psi^{-1}(\bfP)\bigr)$, we may assume without loss
of generality that~$\<f,\bfP\>\in\Moduli_d^N[\Pcal](K)$
and~$f\in\End_d^N(K)$.

It follows that for all~$\s\in\Gal(\Kbar/K)$, we have
\[
  \<f,\bfP^\s\> = \<f^\s,\bfP^\s\> = \<f,\bfP\>^\s = \<f,\bfP\>.
\]
Hence there is a exists $\f_\s\in\PGL_{N+1}(\Kbar)$ such that
\[
  (f,\bfP^\s) = \f_\s\star(f,\bfP) = \bigl(f^{\f_\s},\f_\s^{-1}(\bfP)\bigr).
\]
In particular, we see that~$\f_\s\in\Aut(f)$, and hence
\[
  \#\bigl\{\bfP^\sigma : \sigma \in \Gal(\Kbar/K)\bigr\}
  \le \#\bigl\{\f^{-1}\bfP : \f\in\Aut(f)\bigr\}
  \le \#\Aut(f).
\]
It follows that
\[
  \bigl[ K(\bfP) : K \bigr] \le \#\Aut(f),
\]
and now we need merely invoke Levy's theorem~ \cite{MR2741188},
mentioned already in the proof of
Theorem~\ref{theorem:EnddNstableAutfinite}(a), that~$\#\Aut(f)$ is
bounded by a constant that depends on only~$N$ and~$d$.
\end{proof}

\section{Uniform Boundedness of Preperiodic Points and Its Relation to Algebraic Points on $\Moduli_d^N[\Pcal]$}
\label{section:uniformbdedness}

The strong uniform boundedness conjecture gives a strong uniform bound
for the number of preperiodic points of morphisms of~$\PP^N$ over
number fields of bounded degree. In this section we relate this
conjecture to the existence of algebraic points on moduli
spaces~$\Moduli_d^N[\Pcal]$.

\begin{conjecture}[Strong Uniform Boundedness Conjecture; Silverman--Morton \cite{mortonsilverman:rationalperiodicpoints}]
  \label{conjecture:strongUBC}
For every $D\ge1$,
  $N\ge1$, and $d\ge2$ there is a constant~$C_3(D,N,d)$ such that for
  all number fields~$K/\QQ$ satisfying $[K:\QQ]\le D$ and all
  endomorphisms $f\in\End_d^N(K)$, we have
  \[
  \# \Bigl( \PrePer(f) \cap \PP^N(K) \Bigr) \le C_3(D,N,d).
  \]
\end{conjecture}

\begin{remark}
Fakhruddin~\cite{MR1995861} has shown that
Conjecture~\ref{conjecture:strongUBC} implies an analogous uniform
boundedness theorem for torsion points on abelian varieties.  He has
also shown that if Conjecture~\ref{conjecture:strongUBC} is true for
$D=1$, i.e., for $K=\QQ$, then it is true for all number fields.
Levy, Manes, and Thompson~\cite[Theorem~2.9]{MR3223364} have shown
that Conjecture~\ref{conjecture:strongUBC} is true if we fix a map~$f$
and consider only the family of~$\Kbar/K$-twists of~$f$, i.e., the set
of maps~$f^\f$ such that~$\f\in\PGL_{N+1}(\Kbar)$
and~$f^\f\in\End_d^N(K)$.
\end{remark}

We next consider a uniformity conjecture for points on
$\Moduli_d^N[\Pcal]$ over fields of bounded degree.

\begin{conjecture}[Strong Moduli Boundedness Conjecture]
  \label{conjecture:strongMBC}
  For every $D\ge1$, $N\ge1$, and $d\ge2$ there is a constant~$C_4(D,N,d)$
  such that for all number fields~$K/\QQ$ satisfying $[K:\QQ]\le D$ and
  all unweighted preperiodic portraits $\Pcal=(\Vcal,\F)$ satisfying
  $\#\Vcal\ge C_4(D,N,d)$, we have
  \[
  \Moduli_d^N[\Pcal](K) = \emptyset.
  \]
\end{conjecture}

\begin{example}
Let~$\Ccal_n$ denote an unweighted portrait containing a single
$n$-cycle, and let $\Pcal_1$ and~$\Pcal_2$ be the following portraits:
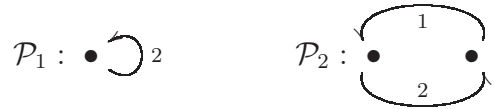
\[
  \Pcal_1 :\xymatrix{
    {\bullet}  \ar@(dr,ur)[]_{2}
  }
  \hspace{4em}
  \Pcal_2 : \xymatrix{
     {\bullet} \ar@(dl,dr)[r]^{2}   & {\bullet}   \ar@(ur,ul)[l]^{1} \\ }
\]
\par
Flynn, Poonen, and Schaefer~\cite{flynnetal:cyclesofquads} conjecture
that $\Moduli_2^1[\Pcal_1\sqcup\Ccal_n](\QQ)=\emptyset$ for all
$n\ge4$, or equivalently, that a polynomial of the form~$x^2+c$
with~$c\in\QQ$ has no~$\QQ$-rational periodic points of period~$4$ or
greater. Assuming this conjecture,
Poonen~\cite{poonen:uniformboundrefined} proves that there are
exactly~$12$ preperiodic portraits~$\Pcal'_1\supseteq\Pcal_1$ (up to
isomorphism) with $\Moduli_2^1[\Pcal'_1](\QQ)\ne\emptyset$, and that
among these portraits, the maximum number of vertices is~$9$.
\par
Similarly, Canci and Vishkautsan~\cite{MR3656200} conjecture that
for all $n\ge2$ we have
$\Moduli_2^1[\Pcal_2\sqcup\Ccal_n](\QQ)=\emptyset$, i.e.,
if~$f\in\Moduli_2^1(\QQ)$ has a critical $2$-cycle, then~$f$ contains no
further $\QQ$-rational non-fixed periodic points.  Assuming this
conjecture, they prove that there are exactly~$13$ possible preperiodic
portraits~$\Pcal'_2\supseteq\Pcal_2$ (up to isomorphism) with
$\Moduli_2^1[\Pcal'_2](\QQ)\ne\emptyset$, and that among these portraits,
the maximum number of vertices is~$9$.
\end{example}

As one might expect, Conjectures~\ref{conjecture:strongUBC}
and~\ref{conjecture:strongMBC} are equivalent, but the proof involves
a subtlety due to field of moduli versus field of definition issues.

\begin{theorem}
\label{theorem:UBCifffMBC}
Conjecture~$\ref{conjecture:strongUBC}$ is true if and only if
Conjecture~$\ref{conjecture:strongMBC}$ is true.
\end{theorem}
\begin{proof}
The logic of the proof of Theorem~\ref{theorem:UBCifffMBC} is as
follows:
\begin{align*}
  \text{(I):\enspace}&
  (\text{Conjecture~\ref{conjecture:strongUBC}})
  \;\textbf{AND}\;
  \textbf{NOT}(\text{Conjecture~\ref{conjecture:strongMBC}})
  \;\Longrightarrow\; \rightarrow\leftarrow.\\
  \text{(II):\enspace}&
  (\text{Conjecture~\ref{conjecture:strongMBC}})
  \;\textbf{AND}\;
  \textbf{NOT}(\text{Conjecture~\ref{conjecture:strongUBC}})
  \;\Longrightarrow\; \rightarrow\leftarrow.
\end{align*}

Since the statements of the two conjectures involve a lot of
quantifiers, we take a moment to succinctly state them and their
converses, where to ease notation we assume that all portraits are
preperiodic and unweighted; see Figure~\ref{figure:ubcversions}.

\begin{figure}[ht]
\begin{center}
\begin{tabular}{|l@{\hspace{3em}}l|}\hline
  \multicolumn{2}{|@{}l|}{\textbf{Conjecture~\ref{conjecture:strongUBC} (Strong Uniform Boundedness Conjecture)}}\\*
  &$\forall D\ge1\,\forall N\ge1\,\forall d\ge2 \; \exists C_3\; \forall K/\QQ,\,[K:\QQ]\le D$ \\*
  &\hspace{5em}$\forall f\in\End_d^N(K) \;:\; \# \Bigl( \PrePer(f) \cap \PP^N(K) \Bigr) \le C_3$.\\* \hline
  \multicolumn{2}{|@{}l|}{\textbf{Converse of Conjecture~\ref{conjecture:strongUBC}}}\\*
  &$\exists D\ge1\,\exists N\ge1\,\exists d\ge2 \; \forall B\; \exists K/\QQ,\,[K:\QQ]\le D$ \\*
  &\hspace{5em}$\exists f\in\End_d^N(K) \;:\; \# \Bigl( \PrePer(f) \cap \PP^N(K) \Bigr) \ge B$.\\[2\jot] \hline\hline
  \multicolumn{2}{|@{}l|}{\textbf{Conjecture~\ref{conjecture:strongMBC} (Strong Moduli Boundedness Conjecture)}}\\*
  &$\forall D\ge1\,\forall N\ge1\,\forall d\ge2 \; \exists C_4\; \forall K/\QQ,\,[K:\QQ]\le D$ \\*
  &\hspace{5em}$\forall\Pcal=(\Vcal,\F),\,\#\Vcal\ge C_4 \;:\; \Moduli_d^N[\Pcal](K)=\emptyset$.\\* \hline
  \multicolumn{2}{|@{}l|}{\textbf{Converse of Conjecture~\ref{conjecture:strongMBC}}}\\*
  &$\exists D\ge1\,\exists N\ge1\,\exists d\ge2 \; \forall B\; \exists K/\QQ,\,[K:\QQ]\le D$ \\*
  &\hspace{5em}$\exists\Pcal=(\Vcal,\F),\,\#\Vcal\ge B \;:\; \Moduli_d^N[\Pcal](K)\ne\emptyset$.\\  \hline
\end{tabular}
\caption{Conjectures~\ref{conjecture:strongUBC} and~\ref{conjecture:strongMBC} and their converses}
\label{figure:ubcversions}
\end{center}
\end{figure}

\par\vspace{2\jot}\noindent \textbf{Part I}:\enspace
In the converse to Conjecture~\ref{conjecture:strongMBC}, we fix
the~$(D,N,d)$ whose existence is asserted.  Then for each~$B$, the
field~$K$ and portrait~$\Pcal$ depend on the choice of~$B$, so we
label them~$K_B$ and~$\Pcal_B=(\Vcal_B,\F_B)$ to indicate this
dependence. The converse of Conjecture~\ref{conjecture:strongMBC}
tells us that
\[
  [K_B:\QQ]\le D,\quad \#\Vcal_B \ge B,
  \quad\text{and}\quad \Moduli^N_d[\Pcal_B](K_B)\ne\emptyset.
\]

Let $\xi_B\in\Moduli^N_d[\Pcal_B](K_B)$.
Theorem~\ref{theorem:FODFOMbound} says that there is a field of
definition~$L_{\xi_B}/K$ for~$\xi_B$ whose degree satisfies
\[
  [L_{\xi_B}:K] \le C_2(N,d).
\]
We fix a representative~$\bar\xi_B$ for~$\xi_B$ that is defined
over~$L_{\xi_B}$, say
\[
  \bar\xi_B = (f,P_1,\ldots,P_n) \in \End_d^N[\Pcal_B](L_{\xi_B}),
\]
where $n=\#\Vcal_B$.  The fact that~$\bar\xi_B$ is defined over~$L_{\xi_B}$
means that $f\in\End_d^N(L_{\xi_B})$ and $P_1,\ldots,P_n\in\PP^N(L_{\xi_B})$,
and our assumption that~$\Pcal_B$ is a preperiodic portrait means
that~$P_1,\ldots,P_n$ are in~$\PrePer(f)$. Hence
\begin{equation}
  \label{eqn:PPPNxBgeB}
  \#\Bigl( \PrePer(f) \cap \PP^N(L_{\xi_B}) \Bigr) \ge \#\Vcal_B \ge B.
\end{equation}
On the other hand, we have
\[
  [L_{\xi_B}:\QQ] = [L_{\xi_B}:K]\cdot[K:\QQ] \le C_2(N,d)\cdot D.
\]

We now assume that Conjecture~\ref{conjecture:strongUBC} is true and
derive a contradiction.  Applying
Conjecture~\ref{conjecture:strongUBC} with the field~$L_{\xi_B}$ in
place of~$K$ and the degree constant $C_2(N,d)\cdot D$ in place
of~$D$, we find that
\begin{equation}
  \label{eqn:PPPleC5C1}
  \#\Bigl( \PrePer(f) \cap \PP^N(L_{\xi_B}) \Bigr)
  \le C_3\bigl( C_2(N,d)\cdot D,N,d\bigr).
\end{equation}

Combining~\eqref{eqn:PPPNxBgeB} and~\eqref{eqn:PPPleC5C1} yields
\[
  C_3\bigl( C_2(N,d)\cdot D,N,d\bigr) \ge B.
\]
We have proven that this inequality is true for some $(D,N,d)$ and
all~$B$, which is impossible. This completes the proof that
Conjecture~\ref{conjecture:strongUBC} and the converse of
Conjecture~\ref{conjecture:strongMBC} give a contradiction.

\par\vspace{2\jot}\noindent \textbf{Part II}:\enspace
We assume that Conjecture~\ref{conjecture:strongUBC} is false, and we
fix $(D,N,d)$ whose existence is asserted in the converse to
Conjecture~\ref{conjecture:strongUBC}.  Then for every choice of~$B$, the
field~$K$ and endomorphism~$f$ in the converse to
Conjecture~\ref{conjecture:strongUBC} depend on the choice of~$B$, so we
label them~$K_B$ and~$f_B$ to indicate this dependence. They satisfy
\[
  [K_B:\QQ]\le D,\quad f_B\in\End_d^N(K_B),\quad
  \# \Bigl( \PrePer(f_B) \cap \PP^N(K_B) \Bigr) \ge B.
\]
We label the $K_B$-rational  preperiodic points of~$f_B$ as
\[
  \PrePer(f_B) \cap \PP^N(K_B) = \{P_1,\ldots,P_n\},
\]
and we let~$\Pcal_B=(\Vcal_B,\F_B)$ be the abstract portrait
determined by the action of~$f_B$ on~$(P_1,\ldots,P_n)$. Then
\[
  \bar\xi_B := (f,P_1,\ldots,P_n) \in \End_d^N[\Pcal_B](K_B),
\]
and in particular $\Moduli_d^N[\Pcal_B](K_B)\ne\emptyset$.

We are also assuming that Conjecture~\ref{conjecture:strongMBC} is
true, and Conjecture~\ref{conjecture:strongMBC} tells us that if
$\#\Vcal\ge C_4(D,N,d)$, then $\Moduli_d^N[\Pcal_B](K_B)$ would be
empty. We conclude that $\#\Vcal< C_4(D,N,d)$.  To recapitulate, we
have proven that there exists a triple $(D,N,d)$ such that for
all~$B$,
\[
  C_4(D,N,d) \ge \#\Vcal = n
  = \# \Bigl( \PrePer(f_B) \cap \PP^N(K_B) \Bigr) \ge B.
\]
This is clearly false, which completes the proof that
Conjecture~\ref{conjecture:strongMBC} and the converse of
Conjecture~\ref{conjecture:strongUBC} lead to a contradiction.
\end{proof}

\section{Good Reduction of Dynamical Systems}
\label{section:goodreduction}

Good reduction for dynamical systems with portraits, which is the
topic of the paper~\cite{arxiv1703.00823}, is closely related to the
existence of integral points on the moduli spaces constructed in the
present paper. We discuss this connection, and refer the reader
to~\cite{arxiv1703.00823} for further details, especially for the
dimension~$1$ case.

\begin{definition}
Let~$K$ be a number field with ring of integers~$R$, let~$S$ be a
finite set of places of~$K$, including all archimedean places, and
let~$R_S$ be the ring of~$S$-integers of~$K$.   

A map $f\in\End_d^N(K)$ has \emph{good reduction outside of~$S$}
if~$\langle{f}\rangle$ is in the image of the map
\[
  \End_d^N(R_S) \longrightarrow \Moduli_d^N(K).
\]
This is equivalent to there being a change of variables so that the
new~$f$ induces a morphism~$\PP^N_{R_S}\to\PP^N_{R_S}$ of
$R_S$-schemes.

Let~$\Pcal$ be an unweighted portrait.  According to the definition
in~\cite{arxiv1703.00823}, a map-and-portrait pair
$(f,\bfP)\in\End_d^N[\Pcal|0](K)$ has \emph{good reduction outside
  of~$S$} if~$\langle{f,\bfP}\rangle$ is in the image of the natural
map
\[
  \End_d^N[\Pcal|0](R_S) \longrightarrow \Moduli_d^N[\Pcal|0](K).
\]
What this means, essentially, is that after a change of coordinates,
the map~$f$ is defined over~$K$, has good reduction outside~$S$, and
the points~$\bfP=(P_v)_{v\in\Vcal}$ form a~$\Gal(\Kbar/K)$-invariant set
whose elements remain distinct modulo~$\gP$ for all primes~$\gP$ lying
over primes~$\gp\notin S$.

A stricter version of good reduction asks that we start with a pair
$(f,\bfP)\in\End_d^N[\Pcal|1](K)$, i.e., we insist that the individual
points~$(P_v)_{v\in\Vcal}$ satisfy~$P_v\in\PP^N(K)$;
cf.~\cite{arxiv1703.00823}.
\end{definition}

In general, there are infinitely many $\PGL_{N+1}(R_S)$-conjugacy
classes of maps~$f\in\End_d^N(K)$ having good reduction outside~$S$.
For example, every monic polynomial in~$R_S[x]$ gives an endomorphism
of~$\PP^1$ having good reduction outside~$S$. On the other hand, if we
add enough level structure, i.e., add a portrait with enough vertices,
then at least for~$\PP^1$ there are only finitely many elements
of~$\End_d^1[\Pcal|1](K)/\PGL_2(R_S)$ having good reduction;
see \cite[Theorem~2(a)]{arxiv1703.00823}.  In any case, one might
expect the maps having good reduction outside~$S$ to be reasonably
rare. In~\cite{arxiv1703.00823} we defined the Shafarevich dimension
as a way to measure the abundance or scarcity of good reduction maps
with portraits.

\begin{definition}
Let $N\ge1$, let $d\ge2$, let~$\Pcal$ be an unweighted portrait, and 
let~$\Acal\subset\Aut(\Pcal)$. The \emph{$(N,d)$-Shafarevich dimension
  of~$(\Pcal,\Acal)$} is
\begin{multline*}
  \shafdim^N_d[\Pcal|\Acal] \\ :=
  \sup_{\substack{\text{$K$ a number field}\\\hidewidth \text{$S$ a finite set of places}\hidewidth\\}}
  \dim \overline{ \Image \Bigl( \End_d^N[\Pcal|\Acal](R_S) \longrightarrow \Moduli_d^N(K) \Bigr) },
\end{multline*}
where the overscore indicates Zariski closure. If~$\Acal=\{1\}$, we
may omit it from the notation, and if~$\Pcal=\emptyset$, we write
simply~$\shafdim^N_d$.
\end{definition}

For example, it is proven in~\cite[Proposition~12]{arxiv1703.00823} that
\[
  \shafdim_2^1=2,\quad\text{and that}\quad
  \shafdim_d^1\ge d+1 \text{ for all $d\ge3$.}
\]
Since $\dim\Moduli_d^1=2d-2$, these estimates are optimal for $d=2$
and $d=3$.  It would be interesting to know if they are optimal for
larger values of~$d$.

For $N=1$ and weighted portraits, there are three versions of good
reduction described in~\cite{arxiv1703.00823}. Roughly speaking, in
addition to requiring that~$f$ have good reduction and that the points
in~$(P_v)_{v\in\Vcal}$ remain distinct for primes not in~$S$, we want
to impose a condition relating the ramification indices of~$f$ at the
marked points to the weight function of the portrait.  Thus we can
require that the ramification indices of~$f$ exceed the weights
specified by~$\e$, or we can require that the ramification indices
of~$f$ exactly equal the weights specified by~$\e$, or we can require
that the ramification indices remain exactly equal to the weights even
after reduction modulo every prime not in~$S$. These cases are
assigned symbols~\cite{arxiv1703.00823}, which we summarize by:
\begin{align*}
  \text{$\bullet$-good reduction}
  :\quad
  &\text{$e_f(P_v)\ge\e(P_v)$ for all $P_v\in \Vcal^\circ$.} \\
  \text{$\circ$-good reduction}
  :\quad
  &\text{$e_f(P_v)=\e(P_v)$ for all $P_v\in \Vcal^\circ$.} \\
  \text{$\star$-good reduction}
  :\quad
  &\text{$e_{\tilde f}(\tilde P_v\bmod\gp)=\e(P_v)$ for all $P_v\in \Vcal^\circ$.}
\end{align*}

We now translate these definitions into moduli-theoretic terms.
We fix a weighted portrait~$\Pcal$, and we 
let~$(f,P)\in\End_d^1[\Pcal|\Acal](K)$. Then, after an appropriate
change of coordinates for~$(f,P)$:
\begin{align*}
  \text{$(f,P)$ has $\bullet$-good reduction}
  &\quad\Longleftrightarrow\quad
  (f,P)\in\End_d^1[\Pcal|\Acal](R_S).\quad \\
  \text{$(f,P)$ has $\circ$-good reduction}
  &\quad\Longleftrightarrow\quad
  (f,P)\in\End_d^1[\Pcal|\Acal]^\circ(R_S). \\
  \text{$(f,P)$ has $\star$-good reduction}
  &\quad\Longleftrightarrow\quad
  \widetilde{(f,P)}\in\End_d^1[\Pcal|\Acal]^\circ(\bar\FF_\gp) \;\forall\gp\notin S. 
\end{align*}
For each type of good reduction, we define an associated Shafarevich dimension.
Thus following the notation set in~\cite{arxiv1703.00823}, 
for each symbol $x\in\{\bullet,\circ,\star\}$, we let
\begin{multline*}
  \GR_d^1[\Pcal|\Acal]^x(K,S) \\ :=
  \left\{(f,P)\in\End_d^1[\Pcal|\Acal](K) : \begin{tabular}{@{}l@{}}
    $(f,P)$ has $x$-good\\ reduction outside $S$\\
  \end{tabular}
  \right\},
\end{multline*}
and then we set
\begin{multline*}
  \shafdim^1_d[\Pcal|\Acal]^x \\ :=
  \sup_{\substack{\text{$K$ a number field}\\\hidewidth \text{$S$ a finite set of places}\hidewidth\\}}
  \smash[t]{
  \dim \overline{ \Image \Bigl(
    \GR_d^1[\Pcal|\Acal]^x(K,S)
    \longrightarrow \Moduli_d^N(K) \Bigr) }.
  }
\end{multline*}
We note that these dimensions
satisfy~\cite[Proposition~17]{arxiv1703.00823}
\[
  \shafdim_d^1[\Pcal|\Acal]^\star \le \shafdim_d^1[\Pcal|\Acal]^\circ \le
  \shafdim_d^1[\Pcal|\Acal]^\bullet \le 2d-2.
\]

It is proven in~\cite[Theorem~1]{arxiv1703.00823} that if
$\Pcal=(\Vcal^\circ,\Vcal,\F,\e)$ is a weighted portrait with total
weight
\[
\wt(\Pcal):=\sum_{v\in\Vcal^\circ} \e(v)\ge2d+1,
\]
then $\shafdim_d^1[\Pcal]^\bullet=0$, but that there are always portraits
with $\wt(\Pcal)=2d$ satisfying $\shafdim_d^1[\Pcal]^\bullet\ge1$.
The paper~\cite{arxiv1703.00823} contains an exhaustive analysis of
the~$35$ portraits~$\Pcal$ satisfying~$\wt(\Pcal)\le4$
and~$\Moduli_2^1[\Pcal]\ne\emptyset$, including a calculation of all
three Shafarevich dimensions.  For example, it is shown that there are:
\begin{center}
  \begin{tabular}{l}
    $10$ portraits $\Pcal$ with $\wt(\Pcal)=4$ and $\shafdim^1_2[\Pcal]^\bullet=1$; \\
    $12$ portraits $\Pcal$ with $\wt(\Pcal)=4$ and $\shafdim^1_2[\Pcal]^\bullet=0$; \\
    $\phantom{0}6$ portraits $\Pcal$ with $\wt(\Pcal)=4$ and $\shafdim^1_2[\Pcal]^\star=1$. \\
  \end{tabular}
\end{center}

\section{Multiplier Systems for Periodic Portraits}
\label{section:dynportraits}
A fundamental tool in studying the dynamics of rational maps
on~$\PP^1$ is the use of multipliers at periodic points, since
multipliers are~$\PGL_2$-invariants.  Over a valued field such
as~$\CC$ or~$\CC_p$, the magnitude of the multiplier helps to determine the
short-term local behavior of the map. Algebraically, appropriate
symmetric functions of the multipliers give $\PGL_2$-invariant regular
functions on~$\End_d^1$, and hence descend to functions
on~$\Moduli_d^1$. Further, one can decompose~$\End_d^1$
or~$\Moduli_d^1$ into a (nicely) varying union of slices by specifying
one or more multiplier values.
Thus one might look at
\begin{equation}
  \label{eqn:Perdmlambda}
  \Per_{d,\pd}(\l) := \left\{ f\in\Moduli_d^1 :
  \begin{tabular}{@{}l@{}}
    $f$ has a point $P$ of (formal)\\period $\pd$ and multiplier $\l$\\
  \end{tabular}
  \right\},
\end{equation}
or more generally one can specify the multipliers for a list of periodic points.
See for example~\cite[Appendix~D]{milnor:quadraticmaps} for a
description of the geometry of the curves~$\Per_{2,\pd}(\l)$
in~$\Moduli_2^1\cong\AA^2$, and~\cite{MR718838,MR732343} for the use
of these spaces in studying the stability loci of families of maps on~$\PP^1$.

The classical definition of the multiplier of a rational
function~$f(z)$ at a~$\pd$-periodic point~$\a$ is as the
derivative~$\l_f(\a):=(f^{ \pd})'(\a)$ of the $\pd$th iterate of~$f$
at~$\a$, but a more intrinsic definition, which also obviates the need
to treat~$\infty$ separately, is to note that~$f^{{\pd}}$ induces a
linear transformation
\[
\Df^{{\pd}}_\a : \Tcal_\a\longrightarrow \Tcal_\a
\]
of the tangent space~$\Tcal_\a$ of~$\PP^1$ at~$\a$. The
space~$\Tcal_\a$ has dimension~$1$, and~$\Df^{{\pd}}_\a$ is
multiplication by the scalar~$\l_f(\a)$. This characterization also
makes it clear that the multiplier is~$\PGL_2$-invar\-iant, i.e., we
have $\l_{f^\f}\bigl(\f^{-1}(\a)\bigr)=\l_f(\a)$.  Our goal in this
section is to generalize this construction to endomorphisms
of~$\PP^N$ for all $N\ge1$.

\subsection{The Multiplier Map for a Single $\pd$-Cycle}

We start with the case of a single cycle, and we let~$\Ccal_\pd$ denote
the portrait consisting of a $\pd$-cycle, say
\[
  \Ccal_\pd:=(\ZZ/\pd\ZZ,\F_\pd)
  \quad\text{with}\quad \F_\pd(i)=i+1\bmod \pd.
\]
A geometric point of~$\End_d^N[\Ccal_\pd]$ is a map~$f$ and a~$\pd$-tuple
of distinct points that are cyclically permuted by~$f$. Our general
construction describes~$\End_d^N[\Ccal_\pd]$ and~$\END_d^N[\Ccal_\pd]$ as
subschemes of~$\End_d^N\times(\PP^N)^\pd$, but for the purposes of this
section, it is convenient to observe that they may be naturally
described as subschemes of~$\End_d^N\times\PP^N$ by simply projecting
onto the first copy of~$\PP^N$. Alternatively, we can start with a
scheme
\begin{equation}
  \label{eqn:hatenddN1}
  \widehat{\End}_d^N[\Ccal_\pd]
  := \bigl\{ (f,P)\in\End_d^N\times\PP^N : f^{\pd}(P)=P \bigr\}
\end{equation}
that is ``too large'' and remove the parts that we don't want. Thus
\[
  \End_d^N[\Ccal_\pd] \cong
  \widehat{\End}_d^N[\Ccal_\pd] \setminus
  \bigcup_{\substack{q\mid \pd\\q< \pd}}  \widehat{\End}_d^N[\Ccal_q].
\]
We note that~$\END_d^N[\Ccal_\pd]$ is the closure of $\End_d^N[\Ccal_\pd]$ in
$\widehat{\End}_d^N[\Ccal_\pd]$, and that there is a natural action of~$\ZZ/\pd\ZZ$ on each of
these spaces via
\begin{equation}
  \label{eqn:ZpZaction}
  \ZZ/\pd\ZZ\longrightarrow \Aut\left(\widehat{\End}_d^N[\Ccal_\pd]\right),\quad
  j\longmapsto \Bigl( (f,P)\mapsto \bigl(f,f^j(P)\bigr) \Bigr).
\end{equation}

With this construction, a geometric point of~$\End_d^N[\Ccal_\pd]$
consists of a map~$f$ and a point~$P$ whose orbit is a cycle of exact
period~$\pd$.  In particular, the point~$P$ is fixed by the $\pd$th
iterate~$f^{ \pd}$, so we get an induced map on the tangent space
of~$\PP^N$ at~$P$,
\[
  \Df^{ \pd}_P : \Tcal_{P} \longrightarrow \Tcal_{P}.
\]
We write the characteristic polynomial of $\Df^{ \pd}_P$ as
\[
  \det\bigl(X-\Df^{ \pd}_P\bigr) = \sum_{i=0}^{N} (-1)^i\Lambda_i(f^{\pd},P)X^{N-i},
\]
so~$\Lambda_i(f^{\pd},P)$ is the $i$th symmetric polynomial of
the eigenvalues of the linear transformation~$\Df_P^{\pd}$. These
eigenvalues are typically called the \emph{multipliers of~$f$ at~$P$}.

Each coefficient of the characteristic polynomial defines a regular
function
\[
  \End_d^N[\Ccal_\pd] \longrightarrow \AA^1,\quad
  (f,P) \longmapsto \Lambda_i(f^{\pd},P).
\]
We call these~$\Lambda_i(f^{\pd},P)$ the \emph{multiplier
  coefficients of~$f$ at~$P$}, and we fit them together to define the
\emph{multiplier map}
\[
  \mu_\pd : \End_d^N[\Ccal_\pd] \longrightarrow \AA^N,\quad
  (f,P) \longmapsto \bigl( \Lambda_1(f^{\pd},P),\ldots,\Lambda_N(f^{\pd},P)\bigr).
\]
The fact that the characteristic polynomial is invariant for the
action of $\PGL_{N+1}$, i.e.,
\begin{equation}
  \label{eqn:AfPPGLinv}
  \Lambda_i\bigl(\f^{-1}\circ f^{\pd}\circ\f,\f^{-1}(P)\bigr)=\Lambda_i(f^{\pd},P),
\end{equation}
shows that the multiplier map descends to give a multiplier map
\[
  \mu_\pd:\Moduli_d^N[\Ccal_\pd]\to\AA^N
\]
on the portrait moduli space. Further, since the multipliers are invariant
for the action~\eqref{eqn:ZpZaction}, i.e., we
have~$\Lambda_i(f^p,P)=\Lambda_i\bigl(f^p,f^j(P)\bigr)$ for all
$j\in\ZZ/p\ZZ$,  the multiplier map descends to a map on the
quotient space
\[
  \mu_\pd:\Moduli_d^N[\Ccal_\pd | \ZZ/\pd\ZZ]\to\AA^N.
\]

We claim that the multiplier map extends to~$\END_d^N[\Ccal_\pd]$
and~$\MODULI_d^N[\Ccal_\pd]$, and more generally, it extends to spaces
attached to disjoint unions of cyclic portraits. The easiest way to
see this is to define the multiplier map on a generalization of the
larger scheme~\eqref{eqn:hatenddN1} that we used earlier.  

\subsection{The Multiplier Map for Cycles of Length $\boldsymbol{\pd_1,\ldots,\pd_n}$}
\leavevmode{}\\
Let~$f=[f_0,\ldots,f_N]$ be a generic element of~$\End_d^N$, i.e., the
coefficients of~$f_0,\ldots,f_N$ are
indeterminates~$a_0(\bfe),\ldots,a_N(\bfe)$ as
in~\eqref{eqn:frhoxdef}. We write the $\pd$th iterate of~$f$ as a
polynomial of degree~$d^\pd$ with coefficients that are polynomials in
the~$a_\rho(\bfe)$, say
\[
  f^{ \pd}(\bfx) = \bigl[ F^{(\pd)}_0(\bfx),\ldots,F^{(\pd)}_N(\bfx) \bigr].
\]

As in the proof of Theorem~\ref{theorem:stableEnddNtimesPNnz}, we want to
define various subschemes of $\End_d^N\times(\PP^N)^n$, so for each
$1\le\nu\le{n}$, we let
\[
  \bfU_\nu := \bigl[U_\nu(0),\ldots,U_\nu(N)\bigr]
  = \left\{ \begin{tabular}{@{}l@{}}
    homogeneous coordinates on\\
    the $\nu$th copy of~$\PP^N$ in~$(\PP^N)^n$\\
  \end{tabular} \right\},
\]
and we let
\[
  \bfU := (\bfU_1,\ldots,\bfU_n).
\]
We also fix an $n$-tuple of periods~$\bfpd:=(\pd_1,\ldots,\pd_n)$, and we write
\[
  \Ccal_\bfpd := \Ccal_{\pd_1}\cup\Ccal_{\pd_2}\cup\cdots\cup\Ccal_{\pd_n}
\]
for the portrait consisting of a disjoint union of periodic cycles of
the indicated periods.  We then define a large subscheme of
$\End_d^N\times(\PP^N)^n$, generalizing~\eqref{eqn:hatenddN1}, by 
\[
  \widehat{\End}_d^N[\Ccal_\bfpd] :=
  \bigcap_{\nu=1}^n \bigcap_{i,j=1}^N  \Bigl\{ (f,\bfU) :
     U_\nu(i)  F^{(\pd_\nu)}_j(\bfU_\nu) - U_\nu(j)F^{(\pd_\nu)}_i(\bfU_\nu)  \Bigr\}.
\]
This  formula is how we define the structure of 
\smash[t]{$\widehat{\End}_d^N[\Ccal_\bfpd]$} explicitly as a subscheme
of the product $\End_d^N\times(\PP^N)^n$, but intuitively at the level of
geometric points, we see that 
\begin{multline*}
  \widehat{\End}_d^N[\Ccal_\bfpd] := \\
  \bigl\{ (f,\bfP)\in\End_d^N\times(\PP^N)^n :
  \text{$f^{\pd_\nu}(P_\nu)=P_\nu$ for all $1\le\nu\le{n}$} \bigr\}.
\end{multline*}
We say that the scheme $\widehat{\End}_d^N[\Ccal_\bfpd]$ is ``large''
because it tends to include extra components beyond the closure
of~$\End_d^N[\Ccal_\bfpd]$. This is due to the fact that we do not
require the orbits of the~$P_\nu$ to be disjoint, nor do we
require individual~$P_\nu$ to have exact or formal period~$\pd_\nu$.

For each 
$0 < i \le N$
and each $1\le\nu\le n$, the multiplier
coefficient~$\Lambda_i(f^{\pd_\nu},P_\nu)$ gives a regular function
\[
  \widehat{\End}_d^N[\Ccal_\bfpd] \longrightarrow \AA^1,
  \quad
  (f,\bfP) \longmapsto \Lambda_i(f^{\pd_\nu},P_\nu).
\]
Fitting these together, we obtain a \emph{multiplier map} 
\[
  \mu_\bfpd : \widehat{\End}_d^N[\Ccal_\bfpd] \longrightarrow \AA^{Nn},
  \quad
  (f,\bfP) \longmapsto \bigl( \Lambda_i(f^{\pd_\nu},P_\nu) \bigr)
  _{\substack{0 < i \le N\\1\le\nu\le n\\}}.
\]
N.B. The map~$\mu_\bfpd$ is a morphism on this large closed subscheme
of $\End_d^N\times(\PP^N)^n$. By restriction, we obtain a
morphism on any subscheme of this large scheme, so in particular
we get multiplier maps
\begin{equation}
  \label{eqn:multmapEndEND}
  \mu_\bfpd : \End_d^N[\Ccal_\bfpd] \longrightarrow \AA^{Nn}
  \quad\text{and}\quad
  \mu_\bfpd : \END_d^N[\Ccal_\bfpd] \longrightarrow \AA^{Nn}.
\end{equation}
But~\eqref{eqn:AfPPGLinv} says that these maps are invariant for the
action of~$\PGL_{N+1}$ on~$\End_d^N[\Ccal_\bfpd]$, hence they descend
and give well-defined maps as in the next definition.

\begin{definition}
Let $\bfpd=(\pd_1,\ldots,\pd_n)$. The \emph{multiplier maps} (\emph{morphisms})
on the associated portrait moduli spaces are denoted
\begin{equation}
  \label{eqn:multmapMdNdef}
  \mu_\bfpd : \Moduli_d^N[\Ccal_\bfpd] \longrightarrow \AA^{Nn}
  \quad\text{and}\quad
  \mu_\bfpd : \MODULI_d^N[\Ccal_\bfpd] \longrightarrow \AA^{Nn}.
\end{equation}
\end{definition}

Using these multiplier maps allows us to generalize the classical
decomposition~\eqref{eqn:Perdmlambda} of~$\Moduli_d^1$ in a variety of
ways, for example, by taking inverse images of points
in~$\AA^{Nn}$. More generally, we can fiber~$\AA^{Nn}$ using
$k$-dimensional linear subspaces, leading to a decomposition
of~$\MODULI_d^N$ as a union
\[
  \MODULI_d^N = \pi \biggl(
  \bigcup_{L \in \operatorname{Gr}_k^{Nn}}  \mu_\bfpd^{-1}(L) \biggr),
\]
where $\pi:\MODULI_d^N[\Ccal_\bfpd]\to\MODULI_d^N$ denotes the natural
projection map and $\operatorname{Gr}_k^{Nn}$ is the Grassmanian
parameterizing~$k$-dimensional linear subspaces of~$\AA^{Nn}$.

\subsection{The Multiplier Map for Full $\pd$-Cycle Structure}

For this subsection we again assume:
\begin{center}
  \framebox{\parbox{.78\hsize}{\noindent
      All varieties are defined over a field of
      characteristic $0$.}}
\end{center}
In order to extract the maximum amount of information from periodic
points of a map~$f$, we should use all of the periodic points of a given
period.  For integers $N\ge1$, $d\ge2$, and $\pd\ge1$, we recall
that a map~$f\in\End_d^N$ has exactly~$\nu_d^N(\pd)$ geometric periodic
points of formal period~$\pd$, counted with appropriate multiplicities,
where~$\nu_d^N(\pd)$ is given by~\eqref{eqn:nudNn} in
Definition~$\ref{definition:nudNn}$;
see~\cite[Theorems~4.3~and~4.17]{arxiv:0801.3643}.  Further,
Lemma~\ref{lemma:NumPrepPts} says that a generic map~$f\in\End_d^N$
has exactly this many geometric points of exact period~$\pd$.
To ease notation, for the remainder of this section we let
\[
  \nubar(\pd) := \frac{1}{\pd}\nu_d^N(\pd) =
  \frac{1}{\pd}   \sum_{k\mid \pd}\left( \mu(\pd/k) \sum_{j=0}^N d^{jk}\right)
\]
denote the number of geometric $\pd$-cycles of a generic $f\in\End_d^N$.

By a slight abuse of notation, we let
\[
  \bfpd(N,d) := (\underbrace{\pd,\pd,\pd,\ldots,\pd}_{\hidewidth\text{$\nubar(\pd)$ copies of $\pd$}\hidewidth}),
\]
so the portrait $\Ccal_{\bfpd(N,d)}$ is given by
\[
\Ccal_{\bfpd(N,d)} = \text{a disjoint union
  of~$\nubar(\pd)$ periodic cycles of period~$\pd$.}
\]
The (f) $\Longrightarrow$ (c) implication of
Theorem~\ref{theorem:UnweightedNonempty} tells us that
\[
  \END_d^N[\Ccal_{\bfpd(N,d)}]\to\End_d^N
\]
is a finite surjective morphism. (This is where we use our
characteristic~$0$ assumption, since it is required for
Theorem~\ref{theorem:UnweightedNonempty}.)

We let
\[
  \Acal_{N,d,\pd} = \Aut(\Ccal_{\bfpd(N,d)}) \cong \Scal_{\nubar(\pd)} \wr (\ZZ/\pd\ZZ)
  \cong \Scal_{\nubar(\pd)} \rtimes (\ZZ/\pd\ZZ)^{\nubar(\pd)}
\]
be the full automorphism group of our disjoint union
of~$\pd$-cycles. It is a wreath product consisting of permutations and
rotations of the cycles.

As above, there are multiplier maps~\eqref{eqn:multmapMdNdef}
associated to the full $\pd$-period portraits~$\Ccal_{\bfpd(N,d)}$.
The pure cycle rotations~\eqref{eqn:ZpZaction} in~$\Acal_{N,d,\pd}$
have no effect on the multipliers, so leave the multiplier map
invariant, while the permutation part of~$\Acal_{N,d,\pd}$ corresponds
to permuting the various copies of~$\AA^N$ in the images of the
multiplier maps~\eqref{eqn:multmapEndEND}.  Hence we obtain multiplier
maps (which by abuse of notation we continue to
call~$\mu_{\bfpd(N,d)}$)
\begin{align}
  \label{eqn:multmapMdNANddef1inter}
  \mu_{\bfpd(N,d)} : \Moduli_d^N[\Ccal_\bfpd|\Acal_{N,d,\pd}] &\longrightarrow \left(\AA^{\nubar(\pd)}\right)^N / \Scal_{\nubar(\pd)},\\
  \label{eqn:multmapMdNANddef2inter}
  \mu_{\bfpd(N,d)} : \MODULI_d^N[\Ccal_\bfpd|\Acal_{N,d,\pd}] &\longrightarrow  \left(\AA^{\nubar(\pd)}\right)^N / \Scal_{\nubar(\pd)},
\end{align}
where $\Scal_{\nubar(\pd)}$ acts diagonally on $(\AA^{\nubar(\pd)})^N$.

We next observe that the map $\Moduli_d^N[\Ccal_\bfpd]\to\Moduli_d^N$ factors
through $\Moduli_d^N[\Ccal_\bfpd|\Acal_{N,d,\pd}]$, since the action of~$\Acal_{N,d,\pd}$
on $\bigl(f,P_1,\ldots,P_{\nubar(p)}\bigr)$ leaves~$f$ invariant. On the other hand, both
of the maps
\[
  \Moduli_d^N[\Ccal_\bfpd]\to\Moduli_d^N[\Ccal_\bfpd|\Acal_{N,d,\pd}]
  \quad\text{and}\quad
  \Moduli_d^N[\Ccal_\bfpd]\to\Moduli_d^N
\]
have generic degree equal to $\#\Acal_{N,d,p}=\nubar(p)!\cdot p^{\nubar(p)}$.
Hence the map $\Moduli_d^N[\Ccal_\bfpd|\Acal_{N,d,\pd}]\to\Moduli_d^N$
is an isomorphism, so~\eqref{eqn:multmapMdNANddef1inter} may be viewed
as a multiplier map
\[
  \mu_{\bfpd(N,d)} : \Moduli_d^N\longrightarrow \left(\AA^{\nubar(\pd)}\right)^N / \Scal_{\nubar(\pd)}.
\]

\begin{remark}
When~$N\ge2$, there are some advantages to taking a further quotient
in which we allow~$N$ copies of~$\Scal_{\nubar(\pd)}$ to act
independently on the~$N$ copies of $\AA^{\nubar(\pd)}$, and then use
the general fact that $\AA^\nu/\Scal_\nu\cong\AA^\nu$ via the
elementary symmetric functions. This gives a ``collapsed'' multiplier map
\begin{equation}
  \label{eqn:collapsed}
  \mu'_{\bfpd(N,d)} : \Moduli_d^N 
  \xrightarrow{\mu_{\bfd(N,d)}}
  \left(\AA^{\nubar(\pd)}\right)^N / \Scal_{\nubar(\pd)},
  \longrightarrow
  \left(\AA^{\nubar(\pd)} / \Scal_{\nubar(\pd)}\right)^N
  \xrightarrow{\;\sim\;} \AA^{\nubar(\pd)N}.\\
\end{equation}
We note that these maps coincide for~$N=1$, i.e.,
$\mu'_{\bfpd(1,d)}=\mu_{\bfpd(1,d)}$.
\end{remark}

A famous theorem of McMullen says that for $N=1$ and in
characteristic~$0$, sufficiently many multipliers (almost) determine
the map $\PGL_2$-conjugacy class of the map~$f$. In our notation, McMullen's theorem
says the following.

\begin{theorem}[McMullen \cite{mcmullen:rootfinding}]
For every~$d\ge2$ there is a~$\kappa(d)$ with the property that the map
\[
  \prod_{\pd\le \kappa(d)} \mu_{\bfpd(N,d)} : \Moduli_d^1(\CC)
  \longrightarrow \prod_{\pd\le \kappa(d)} \CC^{\nubar(\pd)}
\]
is quasi-finite away from the locus of flexible Latt\`es maps.  
\end{theorem}

This naturally leads to two questions, as described
in~\cite[Question~2.43]{MR2884382}.

\begin{question}
For $N\ge2$ and $d\ge2$, is there a $\kappa(N,d)$ such that the map
\[
  \prod_{\pd\le \kappa(N,d)} \mu'_{\bfpd(N,d)} : \Moduli_d^N(\CC)
  \longrightarrow \prod_{\pd\le \kappa(N,d)} \CC^{\nubar(\pd)N}
\]
is quasi-finite on a non-empty Zariski open subset?  If so, what does
the complement look like?
\end{question}

\begin{question}[Poonen]
For $N\ge1$ and $d\ge2$, is there a $\kappa(N,d)$ such that the
map
\[
  \prod_{\pd\le \kappa(N,d)} \mu_{\bfpd(N,d)} : \Moduli_d^N(\CC)
  \longrightarrow \prod_{\pd\le \kappa(N,d)} (\CC^{\nubar(\pd)})^N/\Scal_{\nubar(\pd)}
\]
is one-to-one on a non-empty Zariski open subset
of~$\Moduli_d^N(\CC)$? Note that this question is open even for~$N=1$.
\end{question}

\subsection{Fixed Point Multiplier Relations}
\label{section:multiplierrelatons:}
In this section we consider an endomorphism $f:\PP^N\to\PP^N$ of
degree $d\ge2$ and fixed points $P\in\Fix(f)$. We factor the characteristic
polynomial of the linear map $\Df_P:\Tcal_P\to\Tcal_P$ as
\[
  \det\bigl(X-\Df_P\bigr) = \sum_{i=0}^{N} (-1)^i\Lambda_i(f,P)X^{N-i}
  = \prod_{j=1}^N \bigl( X - \l_j(f,P) \bigr),
\]
so $\l_1(f,P),\ldots,\l_N(f,P)$ are the multipliers of~$f$ at~$P$,
with appropriate multiplicities. We say that~$P$ is a \emph{simple fixed point}
if none of its multipliers is equal to~$1$. 

When $N=1$, each fixed point~$P$ of~$f$ has a single
multiplier~$\l(P)$, and assuming that the fixed points are simple,
their multipliers satisfy the following famous classical
relation:\footnote{In his 1919 M\'emoire~\cite[Section~9]{julia:1918},
  Julia says that this relation is ``bien connue.''}
\begin{equation}
  \label{eqn:classicalsumfix}
  \sum_{P\in\Fix(f)} \frac{1}{1-\l(f,P)} = 1.
\end{equation}
Writing~\eqref{eqn:classicalsumfix} as a symmetric function of
the~$\l(f,P)$, we see that it defines a hypersurface containing the
image of
\[
  \mu_{\boldsymbol1(1,d)} : \Moduli_d^1 \longrightarrow \AA^{d+1}.
\]
(Note that an~$f\in\End_d^1$ has $d+1$ fixed points.)  For example,
when~$d=2$ the relation~\eqref{eqn:classicalsumfix} says that the
image lies in a hyperplane, and indeed, in this case Milnor proved
that $\mu_{\boldsymbol1(1,2)}$ gives an
isomorphism~$\Moduli_2^1\to\AA^2$;
cf.\ Example~\ref{example:milnorM21eqA2}.

The classical fixed point multiplier
relation~\eqref{eqn:classicalsumfix} for~$\PP^1$ has been generalized
by Ueda to endomorphisms of~$\PP^N$.

\begin{theorem}
\label{theorem:uedamult}
\textup{(Ueda \cite{MR1479941})}
Let $f\in\End_d^N$, and assume that all of the fixed points of~$f$ are
simple. Then
\begin{equation}
  \label{eqn:sumdetXDfIDf}
  \sum_{P\in\Fix(f)} \frac{\det(X-\Df_P)}{\det(I-\Df_P)}
  = \frac{X^{N+1}-d^{N+1}}{X-d}.
\end{equation}
Equivalently, for each $0\le k\le N$, we have
\begin{equation}
  \label{eqn:sumLambdaifPdi}
  \sum_{P\in\Fix(f)} \frac{\Lambda_k(f,P)}{\det\bigl(I-\Df_P\bigr)} = (-d)^k.
\end{equation}
\end{theorem}

\begin{remark}
Ueda's formula~\eqref{eqn:sumdetXDfIDf} yields many interesting
relations. For example, differentiating~\eqref{eqn:sumdetXDfIDf} with
respect to~$X$ and then setting $X=1$ yields
\begin{equation}
  \label{eqn:sum11liP}
  \sum_{P\in\Fix(f)} \sum_{j=1}^N \frac{1}{1-\l_j(f,P)}
  = d^{N-1}+2d^{N-2}+\cdots+(N-1)d+N,
\end{equation}
which is a natural generalization of the classical $1$-dimensional
result~\eqref{eqn:classicalsumfix}.
\end{remark}

Substituting
\[
  \det(I-\Df_P)=\sum (-1)^i\Lambda_i(f,P)
\]
into Ueda's relation~\eqref{eqn:sumLambdaifPdi} and clearing the
denominators, for each $0\le k\le N$ we obtain a polynomial relation
among the collection of multipliers
\[
  \Bigl\{ \Lambda_i\bigl(f,P_j\bigr) \Bigr\}_{\substack{0\le i\le N\\0\le j\le\nubar(1)\\}},
\]
where $P_1,\ldots,P_{\nubar(1)}$ are the fixed points of $f$.
This relation is invariant under permutation of the points
in~$\Fix(f)$, so it descends to define a hypersurface~$U_k$ having the
property that the image of the fixed-point multiplier map
\[
  \mu_{\boldsymbol1(N,d)} : \Moduli_d^N\longrightarrow \left(\AA^{\nubar(1)}\right)^N / \Scal_{\nubar(1)}
\]
lies in~$U_k$.\footnote{We only know, a priori, that the maps
  $f\in\Moduli_d^N$ with simple fixed points map to points in the
  hypersurface~$U_k$, but the maps with simple fixed points form a
  non-empty Zariski open subset, so the entire image of~$\Moduli_d^N$
  must lie in the closed hypersurface~$U_k$.}  We call~$U_k$ a
\emph{Ueda fixed-point hypersurface}, and we observe that these
hypersurfaces are defined over~$\QQ$.

As noted in~\eqref{eqn:collapsed}, we can collapse the image of the
multiplier map by taking a further quotient. This corresponds to
permuting the elements of the sets
$\bigl\{\Lambda_i(P)\bigr\}_{P\in\Fix(f)}$ independently for each
$0\le{i}\le{N}$, which is somewhat unnatural, since it does not come
from a natural action on~$\Moduli_d^N[\Ccal_{\boldsymbol1}]$. On the
other hand, it leads to a simpler target space, so could be useful.
In any case, we push forward~$U_k$ to obtain the \emph{collapsed Ueda
  fixed-point hypersurfaces}
\[
  U_k' = \Image\Bigl( U_k \subset \left(\AA^{\nubar(1)}\right)^N / \Scal_{\nubar(1)}
  \longrightarrow \left(\AA^{\nubar(1)}/ \Scal_{\nubar(1)} \right)^N
  \cong \AA^{\nubar(1)N}\Bigr).
\]
  
At most~$N$ of the~$N+1$ hypersurfaces~$U_0,\ldots,U_N$ can be
independent, since
\begin{align*}
  \sum_{k=0}^N 
  &\overbrace{\left(\sum_{P\in\Fix(f)} (-1)^k\frac{\Lambda_k(f,P)}{\det\bigl(I-\Df_P\bigr)}-d^k\right)}^{\text{equation that determines $U_k$}} \\
  &=  \sum_{P\in\Fix(f)} \frac{\displaystyle\sum_{k=0}^N (-1)^k\Lambda_k(f,P)}{\det\bigl(I-\Df_P\bigr)}
  - \sum_{k=0}^N d^k \\
  &=  \sum_{P\in\Fix(f)} \frac{\det\bigl(I-\Df_P\bigr)}{\det\bigl(I-\Df_P\bigr)}
  - \sum_{k=0}^N d^k = 0,
\end{align*}
where for the last equality we have used the fact that $\#\Fix(f)=(d^{N+1}-1)/(d-1)$.
It seems likely that there are no other relations.

\begin{conjecture}
\begin{parts}
\Part{(a)}
The  collapsed Ueda hypersurfaces $U_0',\ldots,U_N'$ are all irreducible.
\Part{(b)}
The intersection of the collapsed Ueda hypersurfaces~$U_0',\ldots,U_N'$ is
a subvariety of $\AA^{\nubar(1)N}$ of codimension~$N$.
\Part{(c)}
\textup{\textbf{Question}:} Is the intersection $\bigcap U_k'$ irreducible?
\end{parts}
\end{conjecture}



\begin{acknowledgement}
The authors would like to thank Dan Abramovich, Eric Bedford, David
Eppstein, Noah Giansiracusa, Jeremy Kahn, and Michael Rosen for their
helpful advice.
\end{acknowledgement}




\end{document}